\documentclass[10pt]{amsart}
\usepackage{amsfonts}    
\usepackage{mathrsfs}    
\usepackage{amsmath}     
\usepackage{amssymb}     
\usepackage{latexsym}    
\usepackage{lscape}      
\usepackage{amsthm}      
\usepackage{multirow}
\usepackage{wasysym}     
\usepackage{tikz} 
\usetikzlibrary{matrix,arrows}
\usepackage{dynkin-diagrams}
\usepackage[toc,page]{appendix}    
\usepackage{geometry}   
\usepackage{dsfont}      
\usepackage{color}
\usepackage{graphicx}    
\usepackage{hyperref}    
\usepackage[square,numbers,comma]{natbib}

\theoremstyle{definition}
\newtheorem{theorem}{Theorem}[section]
\newtheorem{lemma}[theorem]{Lemma}
\newtheorem{corollary}[theorem]{Corollary}
\newtheorem{proposition}[theorem]{Proposition}
\newtheorem{definition}[theorem]{Definition}

\newtheorem{remark}[theorem]{Remark}
\newtheorem{conjecture}[theorem]{Conjecture}
\newtheorem*{theorem*}{Theorem}
\newtheorem*{proposition*}{Proposition}

\numberwithin{equation}{section}

\begin{document}

\title[]{Affine Pavings of Hessenberg Ideal Fibers}

\author{Ke Xue}  
\address{Ke Xue, graduated PhD from University of Maryland College Park}
\email{xueke.kent@gmail.com}

\begin{abstract}
We define certain closed subvarieties of the flag variety, Hessenberg ideal fibers, and prove that they are paved by affines. Hessenberg ideal fibers are a natural generalization of Springer fibers. In type $G_2$, we give explicit descriptions of all Hessenberg ideal fibers, study some of their geometric properties and use them to completely classify Tymoczko's dot actions of the Weyl group on the cohomology of regular semisimple Hessenberg varieties. 
\end{abstract}

\maketitle
\setcounter{tocdepth}{1}
\tableofcontents

\setcounter{section}{0}

\section{Introduction} \label{introduction}
Let $G$ be a connected reductive algebraic group over $\mathbb{C}$, $B$ a Borel subgroup, $U$ the unipotent radical of $B$, and let $\mathfrak{g}$, $\mathfrak{b}$ and $\mathfrak{u}$ denote their respective Lie algebras. A Hessenberg ideal is an $\operatorname{ad}(\mathfrak{b})$-stable subspace $I$ of $\mathfrak{u}$. Fix a nilpotent element $N \in \mathfrak{g}$ and a Hessenberg ideal $I$. The Hesseberg ideal fiber $\pi_{I}^{-1}(N)$ is defined to be the fiber over $N$ of the following map,
\begin{displaymath} 
\begin{array}{llll}
\pi_I: & G \times^B I & \longrightarrow & \mathfrak{g} \\
 & (g, X) &  \longmapsto & g \cdot X
\end{array}
\end{displaymath}
where $g \cdot X$ denotes the adjoint action $\operatorname{Ad}(g)(X)$. From this definition, it can be deduced that $\pi_{I}^{-1}(N)$ is a closed subvariety of the flag variety $G/B$ when it is not empty, and that 
\begin{equation} \label{hif1}
\pi_{I}^{-1}(N) = \{\  gB \  | \  g^{-1} \cdot N \in I \} \subset G/B . 
\end{equation}
When the Hessenberg ideal $I$ is the biggest possible option $\mathfrak{u}$, by Equation~\ref{hif1}, $ \pi_{\mathfrak{u}}^{-1}(N)=\{\ gB  \  | \  g^{-1} \cdot N \in \mathfrak{u} \}= \{\text{Borel subalgebras of } \mathfrak{g} \text{ that contain } N \}$. In this case, our Hessenberg ideal fiber $\pi_{\mathfrak{u}}^{-1}(N)$ is exactly the Springer fiber $\mathcal{B}_N$ as in \cite{de1988homology}. In general, $\pi_{I}^{-1}(N)$ is a closed subvariety of $\mathcal{B}_N$.

The main result of this paper can be roughly stated as the following. 

\begin{theorem*} 
Let $G$ be a connected reductive algebraic group over $\mathbb{C}$ whose Lie algebra has no simple component of type $E_7$ or $E_8$. For any Hessenberg ideal $I \subset \mathfrak{u}$ and any nilpotent element $N \in \mathfrak{g}$, the Hessenberg ideal fiber $\pi_{I}^{-1}(N)$ is paved by affines whenever it is not empty. That is, we can decompose $\pi_{I}^{-1}(N)$ into a finite disjoint union of locally closed subvarieties each of which is an affine space.  
\end{theorem*}

In \cite{de1988homology}, de Concini, Lusztig and Procesi showed that Springer fibers for classical groups are paved by affines. The main theorem of this paper is a direct generalization of their result, and its proof is broadly inspired by their arguments. Before this paper, Hessenberg ideal fibers (under different names) have already been considered by other authors, e.g. \cite{fresse2016existence}, \cite{ji2019hessenberg} and \cite{sommers2006equivalence}. Sommers \cite{sommers2006equivalence} pointed out that the image of $\pi_I$ is the closure of a single nilpotent orbit, and the Hessenberg ideal fiber over an element of this orbit is a disjoint union of irreducible smooth varieties. Ji and Precup \cite{ji2019hessenberg} proved that Hessenberg ideal fibers for type $A$ are paved by affines and gave a combinatorial formula for their cohomology groups. 
Fresse \cite{fresse2016existence} generalized Springer fibers to certain closed subvarieties of any partial flag variety $G/P$ and proved that they are paved by affines for the classical groups. In the case of the full flag variety $G/B$, \cite{fresse2016existence}[Theorem 1] implies that Hessenberg ideal fibers for the classical groups are paved by affines. Fresse's proof uses an explicit description of $G/P$ as variety of partial flags and a type-by-type inspection for the classical groups. The proof of this paper is more conceptual and works for the exceptional types $G_2$, $F_4$ and $E_6$ as well. In addition, it naturally leads to a way of computing the cell structures of Hessenberg ideal fibers for low-rank $G$ (see section \ref{g2hif}).
After this paper, de Concini and Maffei \cite{de2022paving} used its result (Lemma \ref{surfacelem}) to prove that Springer fibers, a subset of Hessenberg ideal fibers, are paved by affines for type $E_7$.

The major motivation for studying Hessenberg ideal fibers is that knowledge of them can be used to classify Tymoczko's dot actions of the Weyl group $W$ of $G$ on the cohomology of regular semisimple Hessenberg varieties. Hessenberg ideal has a natural ``dual'' notion of Hessenberg subspace. A Hessenberg subspace is an $\operatorname{ad}(\mathfrak{b})$-stable subspace $M$ of $\mathfrak{g}$ containing $\mathfrak{b}$. Let $y \in \mathfrak{g}$ be a regular semisimple element. The (regular semisimple) Hessenberg variety $\operatorname{Hess}(M, y)$ is defined to be the fiber over $y$ of the following map
\begin{equation*}
\begin{array}{llll}
\pi_M: & G \times^B M & \longrightarrow & \mathfrak{g} \\
 & (g, x) &  \longmapsto & g \cdot x
\end{array}
\end{equation*}
where $g \cdot x$ denotes the adjoint action $\operatorname{Ad}(g)(x)$. The ordinary cohomology of $\operatorname{Hess}(M, y)$ with coefficient $\mathbb{C}$, $H^*(\operatorname{Hess}(M, y))$, is independent of the choice of the regular semisimple element $y$, and Tymoczko \cite{tymoczko2007permutation} defined the dot action of $W$ on $H^*(\operatorname{Hess}(M, y))$. The decomposition of $H^*(\operatorname{Hess}(M, y))$ into irreducible $W$ representations is an interesting question in itself and a crucial ingredient of both the Shareshian-Wachs and the Stanley-Stembridge Conjectures. There is a very useful connection between the decomposition of $H^*(\operatorname{Hess}(M, y))$ and the knowledge of Hessenberg ideal fibers, which we briefly explain in the next paragraph. Readers are referred to section \ref{dotactiong2} for more details.

Firstly, there exists a natural one-one correspondence between Hessenberg subspaces and Hessenberg ideals (see section \ref{dotactiong2}). Consider the maps $\pi_M: G \times^B M \longrightarrow \mathfrak{g}$ and $\pi_I: G \times^B I \longrightarrow \mathfrak{g}$ for a pair of Hessenberg subspace $M$ and Hessenberg ideal $I$ corresponding to each other. Let $d$ and $d^{\vee}$ denote the complex dimensions of $G \times^B M$ and $G \times^B I$ respectively and $\underline{\mathbb{C}}$ denote the constant sheaves on both spaces. We thus have two direct push-forward complexes $R\pi_{M*}\underline{\mathbb{C}}[d]$ and $R\pi_{I*}\underline{\mathbb{C}}[d^{\vee}]$. Let $G \times \mathbb{G}_m$ act on $\mathfrak{g}$ with $G$ acting via the adjoint action and $\mathbb{G}_m$ acting by scaling. Then, by fixing a Killing form, we get an autoequivalence $F$ (the Fourier-Sato transform) from the category $\operatorname{Perv}_{G \times \mathbb{G}_m} (\mathfrak{g})$ of $G \times \mathbb{G}_m$-equivariant perverse sheaves on $\mathfrak{g}$ to itself. We know that $F(R\pi_{M*}\underline{\mathbb{C}}[d])=R\pi_{I*}\underline{\mathbb{C}}[d^{\vee}]$ and that $F$ maps simple summands of $R\pi_{M*}\underline{\mathbb{C}}[d]$ to those of $R\pi_{I*}\underline{\mathbb{C}}[d^{\vee}]$ bijectively. On the one hand, picking a regular semisimple element $y \in \mathfrak{g}$, we have $H^*(R\pi_{M*}\underline{\mathbb{C}}|_y) \cong H^*(\operatorname{Hess}(M, y))$ and that the decomposition of $R\pi_{M*}\underline{\mathbb{C}}[d]$ into simple summands leads directly to the decomposition of $H^*(\operatorname{Hess}(M, y))$ into irreducible $W$ representations. On the other hand, picking any nilpotent element $N \in \mathfrak{g}$, we have $H^*(R\pi_{I*}\underline{\mathbb{C}}|_N) \cong H^*(\pi_{I}^{-1}(N))$. Therefore, the knowledge of Hessenberg ideal fibers can help us determine the decomposition of $R\pi_{I*}\underline{\mathbb{C}}[d^{\vee}]$ into simple summands, which in turn leads to the decompositions of $R\pi_{M*}\underline{\mathbb{C}}[d]$ and of $H^*(\operatorname{Hess}(M, y))$. The detailed process for the ideas just outlined is carried out in section \ref{dotactiong2} in the case when $G$ is of type $G_2$.
In particular, section \ref{dotactiong2} contains a proof of Brosnan's Conjecture (\ref{Brconj}) for type $G_2$. The conjecture was later prove for type $A$ by B\u{a}libanu and Crooks \cite{bualibanu2022perverse} and in all types by Precup and Sommers \cite{precup2022perverse}.

The remainder of the paper is organized as follows. Section \ref{preliminaries} covers preliminary results used in the following sections. In section \ref{maintheorempf} and \ref{F4E6pf}, we prove the complete version of the main theorem stated above (Theorem~\ref{mymainthm}). In section \ref{g2hif}, we explicitly compute the cell structures of all Hessenberg ideal fibers for type $G_2$ and show that one of them has disconnected, un-equidimensional irreducible components (Theorem~\ref{aninterestinghif}). In section \ref{dotactiong2}, we use the results of section \ref{g2hif} to classify Tymoczko's dot actions on the cohomology of regular semisimple Hessenberg varieties for type $G_2$ (Theorem~\ref{thmBX}).

The author would like to thank his advisor, Dr. Patrick Brosnan, for suggesting this project and for his invaluable support, Dr. Xuhua He for asking a question that resulted in Theorem~\ref{aninterestinghif}, and Dr. Jeffrey Adams for very helpful discussions on pseudo-Levi subalgebras. 

\section{Preliminaries} \label{preliminaries}
We state definitions and results that will be used later. In this section, except for Theorem~\ref{dclpthm}, $G$ is assumed to be a connected reductive algebraic group over $\mathbb{C}$, without any restriction on its Lie algebra. $B$, $U$, $\mathfrak{g}$, $\mathfrak{b}$ and $\mathfrak{u}$ are the same as in the previous section. 

\subsection{Notation} \label{not}
Let $B$ be the fixed Borel subgroup of $G$. Let $T \subset B$ be a fixed maximal torus with Lie algebra $\mathfrak{t}$ and denote by $W$ the Weyl group of $G$ associated to $T$. Choose a representative $\dot{w} \in N_G(T)$ for each Weyl group element $w \in W=N_G(T)/T$. Let $\Phi^+$, $\Phi^-$ and $\Delta$ denote the positive, negative and simple roots associated to $T$ and $B$. Let $\mathfrak{g}_{\alpha}$ denote the root space corresponding to $\alpha \in \Phi$. Write $U$ for the unipotent radical of $B$, $U^-$ for its opposite subgroup, $\mathfrak{u}$ and $\mathfrak{u}^-$ for their respective Lie algebras.

Given a standard parabolic subgroup $P$ of $G$, we choose for it a specific Levi decomposition $P=L U_P$. $U_P$ is the unipotent radical of $P$. The Levi factor $L$ is determined in the following way. $P$ corresponds to a unique subset $I \subset \Delta$ such that $P=BW_IB$, where $W_I$ is the subgroup of $W$ generated by the set of simple reflections $ \{\  s_{\alpha} \  | \  \alpha \in I \}$. Let $Z=(\bigcap_{\alpha \in I}\operatorname{Ker}(\alpha))^{\circ}$ and define $L=C_G(Z)$. The Lie algebra $\mathfrak{l}$ of $L$ has a root space decomposition $\mathfrak{l}=\mathfrak{t} \oplus (\bigoplus_{\alpha \in \Psi}\mathfrak{g}_{\alpha})$, in which $\Psi$ is the subsystem of $\Phi$ spanned by $I$. The Weyl group $W_L$ of $L$ can be naturally identified with the subgroup $W_I$ of $W$. We denote the Lie algebras of $P$ and $U_P$ by $\mathfrak{p}$ and $\mathfrak{u}_P$ respectively. $B \cap L$ is the Borel subgroup of $L$ with Lie algebra $\mathfrak{b} \cap \mathfrak{l}$. Denote by $\Phi(\mathfrak{u}_P)$ and $\Phi(L)$ the subsets of roots so that 
\[ \mathfrak{u}_P=\bigoplus_{\alpha \in \Phi(\mathfrak{u}_P)} \mathfrak{g}_{\alpha} \quad \operatorname{and} \quad \mathfrak{l}=\mathfrak{t} \oplus (\bigoplus_{\alpha \in \Phi(L)} \mathfrak{g}_{\alpha}).  \]
In particular, $\mathfrak{l}$ has triangular decomposition $\mathfrak{l}=\mathfrak{u}_L^- \oplus \mathfrak{t} \oplus \mathfrak{u}_L$ where
\[ \mathfrak{u}_L=\bigoplus_{\alpha \in \Phi^+(L)} \mathfrak{g}_{\alpha} \quad \operatorname{and} \quad \mathfrak{u}_L^-=\bigoplus_{\alpha \in \Phi^-(L)} \mathfrak{g}_{\alpha},   \]
with $\Phi^{\pm}(L)=\Phi(L) \cap \Phi^{\pm}$. Let $U_L$ denote the unipotent subgroup of $G$ with Lie algebra $\mathfrak{u}_L$. Then, $U_L$ is the unipotent radical of $B \cap L$, and $\mathfrak{u}=\mathfrak{u}_L \oplus \mathfrak{u}_P$.

Depending on context, we may use either $G/B$ or $\mathcal{B}$ to denote the flag variety. $\mathcal{B}$ is viewed as the set of Borel subgroups of $G$ (or equivalently, the set of Borel subalgebras of $\mathfrak{g}$). $G/B$ is viewed as the set of left $B$-cosets. $G$ acts naturally on the flag variety $\mathcal{B}=G/B$. The action on $\mathcal{B}$ is conjugation on Borel subgroups (or adjoint action on Borel subalgebras), and the action on $G/B$ is left multiplication on left $B$-cosets. They are different presentations of the same action. $g \cdot B \in \mathcal{B}$ stands for the Borel subgroup $gBg^{-1}$, $g \cdot \mathfrak{b} \in \mathcal{B}$ stands for the Borel subalgebra $\operatorname{Ad}(g)(\mathfrak{b})$, and $gB \in G/B$ stands for the left $B$-coset. They are the same point in the flag variety. These notational conventions are kept throughout the paper. 

\subsection{Hessenberg ideals}
In what follows, we restate the definition of Hessenberg ideal and give a simple yet useful lemma about it. 

\begin{definition}
A subspace $I \subset \mathfrak{u}$ is a Hessenberg ideal if it is stable under the adjoint action by $\mathfrak{b}$.
\end{definition}

It follows easily from the definition that a Hessenberg ideal $I$ is also stable under the adjoint action by $B$. We define two sets:

$\mathscr{I}=  \{\  \text{subspaces } I \text{ of } \mathfrak{u} \  | \   I \text{ is } \operatorname{ad}(\mathfrak{b})\text{-stable} \  \}$,

$\mathscr{S}= \{\  \text{subsets } S \subset \Phi^+ \  | \  \text{if } \beta \in S, \, \alpha \in \Phi^+ \text{ and } \beta+\alpha \in \Phi^+, \text{ then } \beta+\alpha \in S  \  \}$.

\begin{lemma} \label{hi1}
There is a one-to-one correspondence between $\mathscr{I}$ and $\mathscr{S}$ given by

\[ I=\bigoplus_{\alpha \in S} \mathfrak{g}_{\alpha}. \]

\end{lemma}

\begin{proof}
Straightforward.
\end{proof}

If an ideal $I$ corresponds to a set $S$ as above, for any root $\alpha \in S$, we say that $I$ has a root $\alpha$ and $\alpha$ is a root of $I$ (or $\alpha$ belongs to $I$).

\subsection{Affine pavings}
\begin{definition}
A finite partition of a variety $X$ into subsets is said to be a paving if the subsets in the partition can be indexed $X_1, \ldots, X_n$ in such a way that $X_1 \amalg X_2 \amalg \cdots \amalg X_i$ is closed in $X$ for $i=1, \ldots, n$. A paving is affine if each $X_i$ is a finite disjoint union of affine spaces. In this case, we can alternatively say that $X$ is paved by affines. 
\end{definition}

\subsection{A brief roadmap} \label{roadmap}
We sketch in this subsection a brief roadmap of the proof of the main theorem. For a Hessenberg ideal fiber $\pi_{I}^{-1}(N)$, we obtain a paving by intersecting it with a nice paving of the flag variety $\mathcal{B}$. For each piece in the paving of $\pi_{I}^{-1}(N)$, we consider its fixed-point subvariety by a certain one-dimensional torus. If the fixed-point subvariety is paved by affines, we are done. Otherwise we continue to decompose and take fixed-point sets until we reach something paved by affines. This process is accomplished by combining Lemma~\ref{1st}, Lemma~\ref{3rdtech} and Lemma~\ref{3rdtechaffinepaving}.

\begin{lemma} \label{1st}
If a variety $X$ has a paving $\left\{X_1, X_2, \ldots, X_n \right\}$ such that each $X_i$ $\left(i=1, 2, \ldots, n \right)$ is paved by affines, the same is true of $X$. 
\end{lemma}
\begin{proof}
Straightforward.
\end{proof}

\begin{lemma} \label{3rdtech}
Let $E$ be a connected smooth variety over $\mathbb{C}$ with an algebraic $\mathbb{C}^{\times}$-action. Assume that $E$ can be covered by $\mathbb{C}^{\times}$-stable quasi-affine open subschemes. Let $E^{\mathbb{C}^{\times}}$ denote the fixed-point subvariety and assume that it is nonempty, connected and smooth. Moreover, assume that $\operatorname{lim}_{t \rightarrow 0} t \cdot x \in E^{\mathbb{C}^{\times}}$ for every point $x \in E$, where $t \cdot x$ denotes the $\mathbb{C}^{\times}$-action. Let $Z \subset E$ be a $\mathbb{C}^{\times}$-stable smooth closed subvariety so that $Z^{\mathbb{C}^{\times}}$ is also smooth. Then if $Z^{\mathbb{C}^{\times}}$ is paved by affines, the same is true of $Z$.
\end{lemma}

The key to the proof of this lemma is \cite{bass1985linearizing}[Theorem 9.1], which is stated below for readers' convenience.

\begin{theorem}[{\cite{bass1985linearizing}}, Theorem 9.1] \label{bhthm}
Let $G$ be a reductive group over $\mathbb{C}$ acting on the affine scheme $X=\operatorname{Spec}(A)$. Let $X_0=\operatorname{Spec}(A/I)$ be a closed subscheme of $X$. Assume:
\begin{itemize}
\item[(1)] $X_0$ is $G$-stable and contains all closed orbits.
\item[(2)] There is a $G$-equivariant retraction $\pi: X \longrightarrow X_0$.
\item[(3)] $X_0$ is a local complete intersection in $X$.
\end{itemize}
Then $\pi: X \longrightarrow X_0$ admits the structure of a $G$-vector bundle over $X_0$.
\end{theorem}

The definition of local complete intersection can be found in \cite{bass1985linearizing}[section 8].

Fix an affine scheme $X=\operatorname{Spec}(A)$ and an ideal $I \subset A$. Define $A_0= A/I$ and the closed subscheme $X_0=\operatorname{Spec}(A_0)$ of $X$. Consider the $A_0$-module $N=I/I^2$. Then the graded $A_0$-algebra $\operatorname{gr}_I(A)=\bigoplus_{n \ge 0} I^n/I^{n+1}$ is generated by $N$ in degree 1, and there is a canonical surjection of graded $A_0$-algebras $\phi: \operatorname{Sym}_{A_0}(N) \longrightarrow \operatorname{gr}_I(A)$. 

\begin{definition} \label{lcidfn}
We say that $X_0$ is a local complete intersection in $X$ if the following conditions are satisfied:
\begin{itemize}
\item[(1)] The $A_0$-module $N=I/I^2$ is projective.
\item[(2)] $\phi: \operatorname{Sym}_{A_0}(N) \longrightarrow \operatorname{gr}_I(A)$ is an isomorphism.
\end{itemize}
\end{definition}
Now we can prove Lemma~\ref{3rdtech}.

\begin{proof}
Since $\operatorname{lim}_{t \rightarrow 0} t \cdot x \in E^{\mathbb{C}^{\times}}$ for every $x \in E$, we can define a set-theoretic map $p: E \longrightarrow E^{\mathbb{C}^{\times}}$ by $p(x)=\operatorname{lim}_{t \rightarrow 0} t \cdot x$. It is clearly a $\mathbb{C}^{\times}$-equivariant retraction of the inclusion $E^{\mathbb{C}^{\times}} \hookrightarrow E$. Since $E$ is smooth and covered by $\mathbb{C}^{\times}$-stable quasi-affine open subschemes, we can apply \cite{bialynicki1973some}[Theorem 4.1] and deduce that $p: E \longrightarrow E^{\mathbb{C}^{\times}}$ is an affine fibration. Therefore, for every point $y \in E^{\mathbb{C}^{\times}}$, the fiber $p^{-1}(y)$ is an affine space with a $\mathbb{C}^{\times}$-action. $y$ is the only fixed point within $p^{-1}(y)$ and every other point $x \in p^{-1}(y)$ is ``flowed'' to $y$ by the $\mathbb{C}^{\times}$-action. Let $\pi: Z \longrightarrow Z^{\mathbb{C}^{\times}}$ denote the restriction of $p$ to $Z$. 

Assume $Z^{\mathbb{C}^{\times}}$ is paved by affines, we want to show the same is true of $Z$. Let $\left\{V_1, V_2, \ldots, V_m \right\}$ be an affine paving of $Z^{\mathbb{C}^{\times}}$. It suffices to show that $\left\{ \pi^{-1}(V_1), \pi^{-1}(V_2), \ldots, \pi^{-1}(V_m)\right\}$ is an affine paving of $Z$. Let $W_i \cong \mathbb{A}^l$ be any affine piece lying in some $V_j$. Then it is enough to show that $Z_i\stackrel{\textup{\tiny def}}{=} \pi^{-1}(W_i)$ is an affine space as well. For this purpose we apply Theorem~\ref{bhthm} with $G=\mathbb{C}^{\times}$, $X=Z_i$ and $X_0=W_i$. Next we show that all assumptions of Theorem~\ref{bhthm} are satisfied.

Firstly, $G=\mathbb{C}^{\times}$ is reductive. $X_0=W_i \cong \mathbb{A}^l$ is an affine space, hence certainly an affine scheme. Because $Z$ is a $\mathbb{C}^{\times}$-stable closed subvariety of $E$, it is also covered by $\mathbb{C}^{\times}$-stable quasi-affine open subschemes. $Z$ is also smooth, so we can apply \cite{bialynicki1973some}[Theorem 4.1] again and deduce that $\pi: Z \longrightarrow Z^{\mathbb{C}^{\times}}$ is an affine fibration. In particular, $\pi$ is an affine morphism and so is its restriction $\pi |_{Z_i}: Z_i \longrightarrow W_i$. Now that $W_i \cong \mathbb{A}^l$ is an affine scheme, so is $Z_i$. Because $Z_i$ is $\mathbb{C}^{\times}$-stable and $W_i=Z_i^{\mathbb{C}^{\times}}$, $W_i$ is a closed subscheme of $Z_i$. Using the notation of Theorem~\ref{bhthm}, let $Z_i=\operatorname{Spec}(A)$ and $W_i=\operatorname{Spec}(A/I)$ for some affine $\mathbb{C}$-algebra $A$ and ideal $I \subset A$. 


Secondly, assumption (1) and (2) are clearly both satisfied.


Thirdly, for assumption (3), we need to show that both condition (1) and (2) of Definition~\ref{lcidfn} are satisfied. Since $\pi$ is an affine fibration, it is a smooth morphism and so is $\pi |_{Z_i}: Z_i \longrightarrow W_i$. Because $W_i \cong \mathbb{A}^l$ is a smooth variety, so is $Z_i$.

Now we know that $Z_i=\operatorname{Spec}(A)$ and $W_i=\operatorname{Spec}(A/I)$ are both smooth affine varieties and $W_i$ is a closed subvariety of $Z_i$ corresponding to the ideal $I$. By \cite{hartshorne1997algebraic}[Chapter II, Theorem 8.17], $N=I/I^2$ is a locally free sheaf over $W_i=\operatorname{Spec}(A/I)$. Therefore $N=I/I^2$ is a projective $A_0$-module ($A_0=A/I$). Thus condition (1) of Definition~\ref{lcidfn} is satisfied.

For condition (2), we must show that $\phi: \operatorname{Sym}_{A_0}(N) \longrightarrow \operatorname{gr}_I(A)$ is an isomorphism. Because both sides of $\phi$ are graded $A_0$-modules, it suffices to show that the localization $\phi_{\mathfrak{m}}$ is an isomorphism for each maximal ideal $\mathfrak{m} \supset I$. Without loss of generality, we may assume that $(A, \mathfrak{m})$ is a Noetherian local ring. Let $d$ be the dimension of $Z_i$ and $k$ be the codimension of $W_i$ in $Z_i$. Since $Z_i$ and $W_i$ are smooth varieties, both $(A, \mathfrak{m})$ and $(A/I, \mathfrak{m}/I)$ are regular local rings. By \cite{matsumura1980commutative}[p.~121, Theorem 36], there exists a regular system of parameters $\left\{a_1, a_2, \ldots, a_d\right\}$ of $\mathfrak{m}$ so that $a_1, a_2, \ldots, a_d$ is an $A$-regular sequence and $I=(a_1, a_2, \ldots, a_k)$. Therefore, the ideal $I$ is generated by an $A$-regular sequence $a_1, a_2, \ldots, a_k$. By \cite{matsumura1980commutative}[p.~98, Theorem 27], $a_1, a_2, \ldots, a_k$ is also an $A$-quasiregular sequence. By the definition of quasiregular sequence \cite{matsumura1980commutative}[p.~98], $\phi$ is an isomorphism.

Now that all assumptions of Theorem~\ref{bhthm} are satisfied, we deduce that $\pi_i: Z_i \longrightarrow W_i$ admits the structure of a vector bundle. By the Quillen-Suslin theorem (that finitely generated projective modules over polynomial rings are free; see \cite{lang2002algebra}[p.~850, Theorem 3.7]), $Z_i$ is an affine space and we have finished the proof.
\end{proof} 

The next paragraph is the outcome of several results by Bia{\l}ynicki-Birula and Iversen, stated in a way that fits the proof of the main theorem in section \ref{maintheorempf}. Refer to \cite{brosnan2005motivic}[Theorem 3.2] for an alternative formulation and a short history of the results.

Let $X$ be a smooth projective variety over $\mathbb{C}$ with an algebraic $\mathbb{C}^{\times}$-action. The fixed-point set $X^{\mathbb{C}^{\times}}$ is smooth (\cite{iversen1972fixed}). For each connected component $Y$ of $X^{\mathbb{C}^{\times}}$, set $F_Y= \{\  x \in X \  | \  \operatorname{lim}_{t \rightarrow 0} t \cdot x \in Y \}$ and define the map $\pi_Y: F_Y \longrightarrow Y$ by $\pi_Y(x)=\operatorname{lim}_{t \rightarrow 0} t \cdot x$. Then each $F_Y$ is a locally closed $\mathbb{C}^{\times}$-stable smooth subvariety of $X$ and $\pi_Y: F_Y \longrightarrow Y$ is a $\mathbb{C}^{\times}$-equivariant affine fibration (\cite{bialynicki1973some}). The partition of $X$ into the subsets $F_Y$ is a paving (\cite{bialynickibirula1976some}).

In addition, we have the following lemma about affine pavings.
\begin{lemma} \label{3rdtechaffinepaving}
In the settings above, if $Y$ is paved by affines, so is $F_Y$. As a consequence, if $X^{\mathbb{C}^{\times}}$ is paved by affines, the same is true of $X$.
\end{lemma}

\begin{proof}
When $X$ is smooth projective with a $\mathbb{C}^{\times}$-action, it follows in particular that $X$ is covered by $\mathbb{C}^{\times}$-stable quasi-affine open subschemes (see \cite{bialynicki1973some}[section 4]). Therefore, we can apply \cite{bialynicki1973some}[Theorem 4.1] to $X$ and deduce that $\pi_Y:  F_Y \longrightarrow Y$ is a $\mathbb{C}^{\times}$-equivariant affine fibration. Now using Lemma~\ref{3rdtech} with $E=Z=F_Y$ and $E^{\mathbb{C}^{\times}}=Z^{\mathbb{C}^{\times}}=Y$, this lemma follows immediately.
\end{proof}

\subsection{A paving of the flag variety} \label{apotfv}
As mentioned in the previous subsection, we can obtain a paving of any Hessenberg ideal fiber $\pi_{I}^{-1}(N)$ by intersecting it with a nice paving of the flag variety $\mathcal{B}$. Here we elaborate on this paving of $\mathcal{B}$ and other related results. 

Firstly, there is a well-known affine paving of the flag variety given by Schubert cells.

$G$ has the Bruhat decomposition $G=\coprod_{w \in W} B \dot{w}B$. By abuse of notation, we also say that $G/B=\coprod_{w \in W} B \dot{w} B$. The latter equation is viewed as a partition of the flag variety $G/B$, in which $\dot{w}B$ is considered as a point in $G/B$ and $B \dot{w}B$ the $B$-orbit of $\dot{w}B \in G/B$ (In the rest of this paper, a coset notation like $\dot{w}B$ always represents a point in $G/B$). The $B$-orbit $B\dot{w}B$ is a Schubert cell, and we denote it by $X_w$. In addition, we have the Schubert variety $\overline{X_w}=\coprod_{w' \le w}X_{w'}$, where $\le$ denotes the (strong) Bruhat order on the Weyl group $W$ (see \cite{bernstein1973schubert}). 

For each $w \in W$, define $U^w=U \cap \dot{w}U^-\dot{w}^{-1}$. Its Lie algebra is $\mathfrak{u}^w=\bigoplus_{\alpha \in \Phi_w} \mathfrak{g}_{\alpha}$ where $\Phi_w= \{\  \gamma \in \Phi^+ \  | \  w^{-1}(\gamma) \in \Phi^- \}$. $\mathfrak{u}^w$ is naturally isomorphic to $U^w$ by the $G$-equivariant exponential map. By \cite{humphreys1981linear}[section 28.4], $X_w$ has a normal form $U^w \dot{w}B$. That is, $U^w$ is isomorphic to $X_w$ via the map $u \mapsto u \dot{w}B$. Therefore, we have natural isomorphisms $X_w \cong U^w \cong \mathfrak{u}^w$ and we know that $\operatorname{dim}X_w=\operatorname{dim}U^w=\operatorname{dim}\mathfrak{u}^w=|\Phi_w|=l(w)$, where $l(w)$ is the length of $w$ in the Coxeter group $W$.

Secondly, let $P$ be a standard parabolic of $G$. The finite set of $P$-orbits on $G/B$, after reordering, makes a paving of the flag variety. This is the paving with which we intersect the Hessenberg ideal fiber $\pi_{I}^{-1}(N)$. Next we elaborate on the properties of these $P$-orbits.

Let $P=LU_P$ be the Levi decomposition. $W_L$ denotes the Weyl group of $L$. Define $W^L= \{\  v \in W \  | \  \Phi_v \subset \Phi(\mathfrak{u}_P) \}$. The elements of $W^L$ form a set of minimal representatives for $W_L\backslash W$ in the following sense. 

\begin{lemma} \label{wl1}
Each $w \in W$ can be written uniquely as $w=yv$ with $y \in W_L$ and $v \in W^L$ such that $l(w)=l(y)+l(v)$. 
\end{lemma}

\begin{lemma} \label{wl2}
Let $w=yv$ be the decomposition of $w \in W$ given above. Then $\Phi_w=y(\Phi_v) \amalg \Phi_y$.
\end{lemma}

The two lemmas above can be found in \cite{precup2013affine}[section 3].

Let $I$ be the unique subset of $\Delta$ corresponding to $P$. Note that there is a natural identification $W_L=W_I$, and that $P=BW_IB=BW_LB$. For any $y \in W_L$ and any $v \in W^L$, combining the equation $l(yv)=l(y)+l(v)$ with \cite{humphreys1981linear}[section 29.3, Lemma A] and arguing with a reduced word of $y$, it is not hard to show that $B\dot{y}B \dot{v}B=B \dot{y} \dot{v}B$. Equipped with this identity, we can give a better description of the $P$-orbits on $G/B$.

\begin{lemma} \label{porbit1}
There is a one-one correspondence between $W^L$ and the set of $P$-orbits on $G/B$. For each $v \in W^L$, the corresponding $P$-orbit is $P\dot{v}B$. Moreover, $P\dot{v}B$ is the disjoint union of certain Schubert cells:
\[ P\dot{v}B=\coprod_{y \in W_L} X_{yv}=\coprod_{y \in W_L} U^{yv} \dot{y}\dot{v}B.  \]
\end{lemma}

\begin{proof}
Since $P=BW_LB$, we have
\[ P\dot{v}B=BW_LB\dot{v}B=\coprod_{y \in W_L}B\dot{y}B\dot{v}B=\coprod_{y \in W_L}B \dot{y}\dot{v}B=\coprod_{y \in W_L} X_{yv}=\coprod_{y \in W_L} U^{yv} \dot{y}\dot{v}B.  \]
By Lemma~\ref{wl1}, $W=W_LW^L=\{\ yv\ |\ y \in W_L,\  v \in W^L\}$. Taking the disjoint union of $P\dot{v}B$ over all $v \in W^L$, we have
\[ \coprod_{v \in W^L}P\dot{v}B=\coprod_{v \in W^L}(\coprod_{y \in W_L}X_{yv} )=\coprod_{w \in W}X_w=G/B. \]
Therefore, $v \mapsto P\dot{v}B$ is a one-one correspondence between $W^L$ and the set of $P$-orbits on $G/B$.
\end{proof}

For simplicity of notation, we will use $\mathcal{O}_v$ to denote $P\dot{v}B$ and $\mathcal{O}$ an arbitrary $P$-orbit.

Next we investigate the fixed-point sets $\mathcal{B}^Z$ and $\mathcal{O}^Z$ where $Z$ is the connected center of $L$.

From subsection \ref{not}, we know that $Z=(\bigcap_{\alpha \in I}\operatorname{Ker}(\alpha))^{\circ}$ is the connected center of $L$ and $L=C_G(Z)$. There exists a one-parameter subgroup $\lambda: \mathbb{C}^{\times}\longrightarrow Z$ so that the $\lambda$-fixed-point set $\mathcal{B}^{\lambda}$ and $Z$-fixed-point set $\mathcal{B}^Z$ coincide. Every one-parameter subgroup in $T$ is $W$-conjugate to a dominant one-parameter subgroup, so without loss of generality, we may assume
\[ \langle \lambda, \alpha\rangle=0 \  \forall \alpha \in \Phi(L) \text{ and } \langle \lambda, \gamma \rangle>0 \  \forall \gamma \in \Phi(\mathfrak{u}_P), \]
where $\langle \lambda, \alpha\rangle$ is the natural pairing $Y(T) \times X(T) \longrightarrow \mathbb{Z}$ between the cocharacter and character group of $T$. Clearly $\mathcal{O}^Z=\mathcal{O}^{\lambda}$ for each $P$-orbit $\mathcal{O}$ and $X_w^Z=X_w^{\lambda}$ for each Schubert cell $X_w$.

For each Schubert cell $X_w=U^w\dot{w}B$, $\lambda$ acts on it by left multiplication. Hence, according to the natural isomorphisms $X_w \cong U^w \cong \mathfrak{u}^w$, $\lambda$ acts on $U^w$ by conjugation and on $\mathfrak{u}^w$ by the adjoint action. Let $w=yv$ be the decomposition as in Lemma~\ref{wl2}, we have
\[\mathfrak{u}^w=\bigoplus_{\alpha \in y(\Phi_v)}\mathfrak{g}_{\alpha} \oplus \bigoplus_{\alpha \in \Phi_y}\mathfrak{g}_{\alpha}=\dot{y} \cdot \mathfrak{u}^v \oplus \mathfrak{u}^y.\]
Since $\dot{y} \cdot \mathfrak{u}^v \subset \mathfrak{u}_P$ and $\mathfrak{u}^y \subset \mathfrak{l}$, $\lambda$ yields a $\mathbb{C}^{\times}$-action on $\mathfrak{u}^w$ which has strictly positive weights on $\dot{y} \cdot \mathfrak{u}^v$ and fixes $\mathfrak{u}^y$. Therefore, 
\[ X_w^{\lambda} \cong (\mathfrak{u}^w)^{\lambda}=\mathfrak{u}^y \cong U^y.  \]
Now
\begin{equation} \label{orbitfixedpoint}
\mathcal{O}_v^{\lambda}=(\coprod_{y \in W_L}X_{yv})^{\lambda} = \coprod_{y \in W_L}(U^{yv}\dot{y}\dot{v}B)^{\lambda} = \coprod_{y \in W_L}U^y\dot{y}\dot{v}B \cong \coprod_{y \in W_L}U^y \dot{y}B_L=\mathcal{B}(L),  
\end{equation}
where $\mathcal{B}(L)$ is the flag variety of $L$ and $B_L=L \cap \dot{v} \cdot B=L \cap \dot{v}B \dot{v}^{-1}$ is a Borel of $L$. Extrinsically, the isomorphism $\mathcal{O}_v^{\lambda} \cong \mathcal{B}(L)$ takes $u \dot{y} \dot{v}B$ to $u \dot{y}B_L$ ($u \in U^y$); intrinsically, the isomorphism takes any Borel $B_0 \in \mathcal{O}_v^{\lambda}$ to $B_0 \cap L \in \mathcal{B}(L)$. If we assemble all these $\mathcal{O}_v$ into $\mathcal{B}$, we get the following result.

\begin{proposition}
Each connected component of $\mathcal{B}^Z=\mathcal{B}^{\lambda}$ takes the form of $\mathcal{O}_v^{\lambda}$ for some $v \in W^L$. $\mathcal{O}_v^{\lambda}$ is isomorphic to the flag variety $\mathcal{B}(L)$ of $L$ and the isomorphism sends $B_0 \in \mathcal{B}^Z$ to $B_0 \cap L$ in $\mathcal{B}(L)$.
\end{proposition}

Moreover, we can show that each $\mathcal{O}_v$, considered as a variety with the $\mathbb{C}^{\times}$-action from the left multiplication by $\lambda$, satisfies all the requirements for $E$ as in Lemma~\ref{3rdtech}. To be precise:


\begin{proposition} \label{favb}
$\mathcal{O}_v$ is a connected smooth variety with a $\mathbb{C}^{\times}$-action from the left multiplication by $\lambda$. $\mathcal{O}_v$ can be covered by $\mathbb{C}^{\times}$-stable quasi-affine open subschemes. Moreover, for every point $x \in \mathcal{O}_v$, $\operatorname{lim}_{t \rightarrow 0} \lambda(t) \cdot x \in \mathcal{O}_v^{\lambda}$, where $\lambda(t) \cdot x$ denotes the $\mathbb{C}^{\times}$-action via the left multiplication by $\lambda$.
\end{proposition}

\begin{proof}
Recall that $\mathcal{O}_v$ is the $P$-orbit of $\dot{v}B$ in $G/B$, so it is a connected smooth variety. Since $\mathcal{O}_v=P \dot{v}B$ and $\lambda \subset P$, the left multiplication of $\lambda$ on $\mathcal{O}_v$ clearly stabilizes it. 

Since $\mathcal{O}_v=\amalg_{y \in W_L}X_{yv}$, each point $x \in \mathcal{O}_v$ lies in some Schubert cell $X_{yv}$. By the preceding analysis in this subsection, we know that $X_{yv} \cong \mathfrak{u}^{yv} = \dot{y} \cdot \mathfrak{u}^v \oplus \mathfrak{u}^y \cong \dot{y} \cdot \mathfrak{u}^v \times X_{yv}^{\lambda}$. The projection onto the second factor $\operatorname{pr}_2: X_{yv} \longrightarrow X_{yv}^{\lambda}$ gives a trivial vector bundle structure over $X_{yv}^{\lambda}$. Since $\dot{y} \cdot \mathfrak{u}^v \subset \mathfrak{u}_P$, the left multiplication by $\lambda$ on $X_{yv}$ yields a linear $\mathbb{C}^{\times}$-action on the fiber of $\operatorname{pr}_2$ with strictly positive weights. Therefore, $\operatorname{lim}_{t \rightarrow 0} \lambda(t) \cdot x$ exists and lies in $X_{yv}^{\lambda} \subset \mathcal{O}_v^{\lambda}$.

Now it only remains to prove that $\mathcal{O}_v$ can be covered by $\mathbb{C}^{\times}$-stable quasi-affine open subschemes. Each Schubert cell $X_{yv}$ is clearly $\mathbb{C}^{\times}$-stable and affine, but they are not open subschemes except for the highest dimensional one. However, we can write down an explicit $\mathbb{C}^{\times}$-stable open affine cover of $\mathcal{O}_v$ by using $W_L$-translates of the highest dimensional cell of $\mathcal{O}_v$. Next we elaborate on this claim.

Let $y_0 \in W_L$ be the unique element of maximal length. That is to say, $l(y_0)=\operatorname{dim}(L / L \cap B)=\operatorname{dim} \mathfrak{u}_L$. As a result, $\mathfrak{u}^{y_0}=\mathfrak{u}_L$ and $\dot{y}_0^{-1} \cdot \mathfrak{u}^{y_0}= \mathfrak{u}_L^{-}$. $X_{y_0 v}$ is the Schubert cell of the highest dimension in $\mathcal{O}_v$, hence an open affine subscheme. It is clearly $\mathbb{C}^{\times}$-stable. We claim that the set of $W_L$-translates of $X_{y_0 v}$, $\{\  \dot{y} \dot{y}_0^{-1} X_{y_0 v} \ |\ y \in W_L \}$, is a $\mathbb{C}^{\times}$-stable open affine cover of $\mathcal{O}_v$.

For $\mathbb{C}^{\times}$-stableness, consider the action of $\lambda$ on $\dot{y} \dot{y}_0^{-1} X_{y_0 v}$ by left multiplication. Since $y y_0^{-1} \in W_L$, for every $t \in \mathbb{C}^{\times}$, $\dot{y}_0 \dot{y}^{-1} \lambda(t)\dot{y} \dot{y}_0^{-1} \in T$. Moreover, $X_{y_0 v}=B \dot{y}_0 \dot{v}B$ is clearly stable under left multiplication by $T$. Therefore, $\lambda(t) \dot{y} \dot{y}_0^{-1} X_{y_0 v}=\dot{y} \dot{y}_0^{-1} (\dot{y}_0 \dot{y}^{-1} \lambda(t)\dot{y} \dot{y}_0^{-1}) X_{y_0 v}=\dot{y} \dot{y}_0^{-1} X_{y_0 v}$. Then $\dot{y} \dot{y}_0^{-1} X_{y_0 v}$ is $\mathbb{C}^{\times}$-stable.

For the covering part, it suffices to show that $\dot{y} \dot{y}_0^{-1} X_{y_0 v} \supset X_{yv}$ for every $y \in W_L$. Recall that $X_w=U^w \dot{w} B$, so we need to prove $\dot{y} \dot{y}_0^{-1} U^{y_0 v} \dot{y}_0 \dot{v}B \supset U^{y v} \dot{y} \dot{v} B$. The preceding inclusion relationship follows if we can show that $\dot{y} \dot{y}_0^{-1} U^{y_0 v} \dot{y}_0 \supset U^{y v} \dot{y}$. Multiplying both sides by $\dot{y}^{-1}$, it remains to show the inclusion relationship of two unipotent subgroups of $G$: $\dot{y}_0^{-1} U^{y_0 v} \dot{y}_0 \supset \dot{y}^{-1} U^{y v} \dot{y}$. As both are closed subgroups, the problem can be reduced to the level of Lie algebras: $\dot{y}_0^{-1} \cdot \mathfrak{u}^{y_0 v} \supset \dot{y}^{-1} \cdot \mathfrak{u}^{y v}$.

Now recall that $\mathfrak{u}^{yv}=\dot{y} \cdot \mathfrak{u}^v \oplus \mathfrak{u}^y$ for every $y \in W_L$. Therefore, $\dot{y}_0^{-1} \cdot \mathfrak{u}^{y_0 v} = \mathfrak{u}^{v} \oplus \dot{y}_0^{-1} \cdot \mathfrak{u}^{y_0}$ and $\dot{y}^{-1} \cdot \mathfrak{u}^{y v} = \mathfrak{u}^{v} \oplus \dot{y}^{-1} \cdot \mathfrak{u}^{y}$. Since $\dot{y}_0^{-1} \cdot \mathfrak{u}^{y_0} = \mathfrak{u}_L^{-} \supset \dot{y}^{-1} \cdot \mathfrak{u}^{y}$, the inclusion $\dot{y}_0^{-1} \cdot \mathfrak{u}^{y_0 v} \supset \dot{y}^{-1} \cdot \mathfrak{u}^{y v}$ is clearly true and we have finished the proof.
\end{proof}

\subsection{Associated parabolics} \label{assparab}
As stated in the previous subsection, we will intersect the Hessenberg ideal fiber $\pi_{I}^{-1}(N)$ with the $P$-orbit paving of $G/B$ for some parabolic $P$. In fact, this $P$ is always the associated parabolic of the nilpotent element $N$. 

Let $N \in \mathfrak{g}$ be a nilpotent element. By the Jacobson-Morozov theorem, there exists a homomorphism of algebraic groups $\varphi: \operatorname{SL}_2(\mathbb{C}) \longrightarrow G$ such that $d\varphi\left(\begin{smallmatrix}0&1\\0&0\end{smallmatrix}\right)=N$. Define a one-parameter subgroup $\lambda: \mathbb{C}^{\times} \longrightarrow G$ such that $\lambda(z)=\varphi\left(\begin{smallmatrix}z&0\\0&z^{-1}\end{smallmatrix}\right)$ for all $z \in \mathbb{C}^{\times}$. $\lambda$ decomposes $\mathfrak{g}$ into a direct sum of weight spaces
\[ \mathfrak{g}(i)=\{\ X \in \mathfrak{g}\ |\ \lambda(z) \cdot X=z^i X \ \forall z \in \mathbb{C}^{\times}\}.  \]
We know that $N \in \mathfrak{g}(2)$, and that $\mathfrak{g}=\bigoplus_{i \in \mathbb{Z}}\mathfrak{g}(i)$ where $\left[\mathfrak{g}(i), \mathfrak{g}(j)\right] \subset \mathfrak{g}(i+j)$ for all $i, j \in \mathbb{Z}$. Let $L$ and $P$ denote the connected algebraic subgroups of $G$ whose Lie algebras are $\mathfrak{l}=\mathfrak{g}(0)$ and $\mathfrak{p}=\bigoplus_{i \geq 0}\mathfrak{g}(i)$. It is known that:
\begin{itemize}
\item[(1)] $P$ is a parabolic subgroup depending only on $N$ (not on the choice of $\varphi$).
\item[(2)] $P=LU_P$ is a Levi decomposition, and its unipotent radical $U_P$ has Lie algebra $\mathfrak{u}_P=\bigoplus_{i > 0} \mathfrak{g}(i)$.
\item[(3)] The $P$-orbit of $N$ in $\mathfrak{u}_P^{\geq 2}=\bigoplus_{i \geq 2}\mathfrak{g}(i)$ is dense.
\item[(4)] The $L$-orbit of $N$ in $\mathfrak{g}(2)$ is dense. 
\item[(5)] If $N$ is distinguished, in the sense that it is not contained in any Levi subalgebra of a proper parabolic subalgebra of $\mathfrak{g}$, then $\mathfrak{g}(i)=0$ for all odd $i$ (see \cite{bala1976classes}).
\end{itemize}

In the rest of this paper, for each nilpotent element $N$, the $P$ and $\lambda$ thus obtained are referred to as the associated parabolic of $N$ and an associated one-parameter subgroup of $N$. The image of $\lambda$ in $G$ is usually denoted by $D$. Note that $P$ as a subgroup of $G$ is uniquely determined by $N$ (see \cite{carter1985finite}[p.~163, Proposition 5.7.1]) while $\lambda$ depends of the choice of $\varphi$. Different choices of $\varphi$ are conjugate by an element of $C_G(N)$, hence so are the $\lambda$'s.

It is worth pointing out that such an associated one-parameter subgroup $\lambda$ is the same as the one mentioned in subsection~\ref{apotfv}, below Lemma~\ref{porbit1}. To be precise, we need to show that the $\lambda$ associated to $N$ via the Jacobson-Morozov theorem satisfies the following two requirements:
\begin{itemize}
\item[(1)] $\langle \lambda, \alpha\rangle=0 \  \forall \alpha \in \Phi(L) \text{ and } \langle \lambda, \gamma \rangle>0 \  \forall \gamma \in \Phi(\mathfrak{u}_P)$.
\item[(2)] $\mathcal{B}^Z=\mathcal{B}^{\lambda}$ where $Z$ is the connected center of $L$.
\end{itemize}

For (1), since $\lambda$ decomposes $\mathfrak{g}$ into weight spaces and $\mathfrak{l}=\mathfrak{g}(0)$, $\mathfrak{u}_P=\bigoplus_{i > 0} \mathfrak{g}(i)$, the requirement is clearly satisfied.

For (2), note that $L=C_G(Z)=C_G(\lambda)$. By \cite{humphreys1981linear}[p.~141, section 22.4], both $\mathcal{B}^Z$ and $\mathcal{B}^{\lambda}$ are identical to the set $\{\ B_0 \in \mathcal{B} \ |\ L \cap B_0 \text{ is a Borel subgroup of } L \}$. Hence the requirement is satisfied.

As a consequence, all the results of subsection~\ref{apotfv} can be applied to the $\lambda$, $L$ and $P$ that are associated to $N$ via the Jacobson-Morozov theorem.

\subsection{Prehomogeneous vector spaces} \label{prehomogeneous}
Let $M$ be a connected algebraic group over $\mathbb{C}$ and $V$ be a finite-dimensional vector space over $\mathbb{C}$ with a rational $M$-action. $V$ is said to be prehomogeneous if $V$ contains a dense $M$-orbit $V^0$. Pick an element $v \in V^0$.

Given a closed subgroup $H$ in $M$ and an $H$-stable vector subspace $U$ of $V$, we construct a closed subvariety $X_U \subset M/H$ as follows: Set
\[ M_U=\{\ g \in M\ |\ g^{-1} \cdot v \in U\}.  \]
Then $M_U$ is stable under right multiplication by $H$ and we set $X_U=M_U/H$. Clearly,
\[ X_U=\{\ gH\ |\ g \in M \text{ and } g^{-1} \cdot v \in U\}.  \]

The following result can be found in \cite{de1988homology}[Lemma 2.2].

\begin{lemma} \label{prehsp}
\hspace{1cm} 
\begin{itemize}
\item[(1)] When $X_U$ is not empty, it is smooth and $\operatorname{dim}(X_U)=\operatorname{dim}(M/H)-\operatorname{dim}(V/U)$.
\item[(2)] The connected components of $X_U$ are isomorphic, and $M_v$ acts transitively on the set of them, where $M_v$ is the stabilizer of $v$.
\end{itemize}
\end{lemma}

\begin{remark}
We use the quintuple notation $(M, H, V, U, v)$ to denote all the information necessary to construct $X_U$. By abuse of notation, we also use the quintuple to denote the variety $X_U$ itself. Equality such as $Y \cong (M, H, V, U, v)$ means the variety $Y$ is isomorphic to the variety $X_U$ constructed from the quintuple. 
\end{remark}

Prehomogeneous vector space is the technical core of this paper. It is important for both the proof in section \ref{maintheorempf} and the explicit computation in section \ref{g2hif} and \ref{F4E6pf}. In particular, we can use it to describe small pieces of the Hessenberg ideal fiber and their respective fixed-point subsets. Next we elaborate on this statement.

For a Hessenberg ideal $I$, let $N$ and $N'$ be conjugate nilpotent elements in the image of $\pi_I$. Then $\pi_{I}^{-1}(N)$ and $\pi_{I}^{-1}(N')$ are isomorphic while the associated parabolics $P$ (of $N$) and $P'$ (of $N'$) are conjugate. Therefore, we may assume that the associated parabolic $P$ of the nilpotent element $N$ contains the pre-selected Borel subgroup $B$. This makes it very convenient to intersect $\pi_{I}^{-1}(N)=\{\  gB \  | \  g^{-1} \cdot N \in I \}$ with various $P$-orbits on $G/B$.

Let $L$ be the Levi factor of $P$ and decompose the Weyl group $W=W_LW^L$. Let $\lambda$ be an associated one-parameter subgroup of $N$. By Lemma~\ref{porbit1}, $\{\ \mathcal{O}_v\ |\ v \in W^L \}$ is the set of all $P$-orbits on $G/B$, and they can be ordered into a paving of the flag variety. Since $\pi_{I}^{-1}(N)$ is a closed subvariety of $G/B$, $\{\ \pi_{I}^{-1}(N) \cap \mathcal{O}_v\ |\ v \in W^L\}$ is a paving of $\pi_{I}^{-1}(N)$. In fact, each piece of the paving $\pi_{I}^{-1}(N) \cap \mathcal{O}_{v}$ is a variety constructed from some quintuple. So is its $\lambda$-fixed-point subset $\pi_{I}^{-1}(N) \cap \mathcal{O}_{v}^{\lambda}$.

\begin{lemma} \label{dist1}
$\pi_{I}^{-1}(N) \cap \mathcal{O}_v=(P, P\cap \dot{v}\cdot B, \mathfrak{u}_P^{\geq 2}, \mathfrak{u}_P^{\geq 2} \cap \dot{v} \cdot I,N)$. 
\end{lemma}

\begin{proof}
First we examine the validity of the quintuple. By subsection \ref{assparab} (3), the $P$-orbit of $N$ in $\mathfrak{u}_P^{\geq 2}$ is dense. Then all we need to check is that $\mathfrak{u}_P^{\geq 2} \cap \dot{v} \cdot I$ is $(P\cap \dot{v}\cdot B)$-stable. Since the adjoint action of $\dot{v}$ permutes the root spaces of $\mathfrak{g}$, $\mathfrak{u}_P^{\geq2} \cap \dot{v} \cdot I$ still has a root space decomposition. To prove $\mathfrak{u}_P^{\geq 2}\cap \dot{v} \cdot I$ is $(P\cap \dot{v}\cdot B)$-stable, it is enough to show the following:

Let $\gamma$ be a root of $\mathfrak{u}_P^{\geq 2}\cap \dot{v} \cdot I$ and $\delta$ be a root of $\mathfrak{p} \cap \dot{v} \cdot \mathfrak{b}$. If $\gamma +\delta$ is still a root, then it must be a root of $\mathfrak{u}_P^{\geq 2} \cap \dot{v} \cdot I$.

Now we prove the claim above. Let $\gamma=v(\alpha)$ so that $\alpha$ is a root of $I$ and $v(\alpha)$ is a root of $\mathfrak{u}_P^{\geq 2}$. Let $\delta=v(\beta)$ so that $\beta$ is a root of $\mathfrak{b}$ and $v(\beta)$ is a root of $\mathfrak{p}$. If $\gamma+\delta=v(\alpha)+v(\beta)=v(\alpha+\beta)$ is a root, then so is $\alpha+\beta$. Since $\beta$ is a root of $\mathfrak{b}$ hence positive, and $I$ is an ideal, $\alpha+\beta$ is a root of $I$ as well. Recall that $\lambda$ is the associated one-parameter subgroup of $N$ which we have chosen to begin with. Then $\langle \lambda, \gamma+\delta\rangle=\langle \lambda, \gamma\rangle+\langle \lambda, \delta\rangle \geq 2$, because $\gamma$ is a root of $\mathfrak{u}_P^{\geq 2}$ and $\delta$ is a root of $\mathfrak{p}$. Therefore $\gamma+\delta=v(\alpha+\beta)$ is a root of $\mathfrak{u}_P^{\geq 2} \cap \dot{v} \cdot I$.

Next we prove the equality. We know that $\mathcal{O}_v=P \dot{v}B$. The stabilizer of $\dot{v}B$ in $P$ is $H\stackrel{\textup{\tiny def}}{=} P \cap \dot{v} \cdot B$. For any $p \in P$, the whole coset $pH$, when acting on the point $\dot{v}B$, gives the same result $p\dot{v}B$. If $p\dot{v}B \in \pi_{I}^{-1}(N)$, by definition, $\dot{v}^{-1} \cdot p^{-1} \cdot N \in I \iff p^{-1} \cdot N \in \dot{v} \cdot I$. Since $p^{-1} \in P$ and $N \in \mathfrak{g}(2)$, we must have $p^{-1} \cdot N \in \mathfrak{u}_P^{\geq 2}$. Therefore, $p^{-1} \cdot N \in \dot{v} \cdot I \iff p^{-1} \cdot N \in \mathfrak{u}_P^{\geq 2} \cap \dot{v} \cdot I$. In summary, $\pi_{I}^{-1}(N) \cap \mathcal{O}_v=\{\ pH\ |\ p^{-1} \cdot N \in  \mathfrak{u}_P^{\ge 2} \cap \dot{v} \cdot I \}$ and this is exactly $(P, P\cap \dot{v}\cdot B, \mathfrak{u}_P^{\geq 2}, \mathfrak{u}_P^{\geq 2}\cap \dot{v} \cdot I , N)$.
\end{proof}

Recall that $\lambda$ acts on $G/B$ by left multiplication. Since $N \in \mathfrak{g}(2)$, $\lambda$ stabilizes $\pi_{I}^{-1}(N)$. Taking the $\lambda$-fixed-point subset of $\pi_{I}^{-1}(N) \cap \mathcal{O}_v$, we get $(\pi_{I}^{-1}(N) \cap \mathcal{O}_v)^{\lambda}=\pi_{I}^{-1}(N) \cap \mathcal{O}_v^{\lambda}$.

\begin{lemma} \label{dist2}
$\pi_{I}^{-1}(N) \cap \mathcal{O}_v^{\lambda} \cong (L, L \cap \dot{v} \cdot B, \mathfrak{g}(2), \mathfrak{g}(2) \cap \dot{v} \cdot I , N)$.
\end{lemma}

\begin{proof}
Similar to the previous lemma, we begin by checking the validity of the quintuple. By subsection \ref{assparab} (4), the $L$-orbit of $N$ in $\mathfrak{g}(2)$ is dense. Then all we need to check is that $\mathfrak{g}(2) \cap \dot{v} \cdot I$ is $(L \cap \dot{v} \cdot B)$-stable. Still because of the root space decomposition of $\mathfrak{g}(2)\cap \dot{v} \cdot I$ and $\mathfrak{l} \cap \dot{v} \cdot \mathfrak{b}$, it is enough to show that: 

Let $\gamma$ be a root of $\mathfrak{g}(2) \cap \dot{v} \cdot I$ and $\beta$ be a root of $\mathfrak{l} \cap \dot{v} \cdot \mathfrak{b}$. If $\gamma+\beta$ is still a root, it must be a root of $\mathfrak{g}(2)\cap \dot{v} \cdot I$. 

The proof is almost verbatim the same as in the previous lemma.

Next we prove the equality. By Equation~\ref{orbitfixedpoint}, we know that 
\[\mathcal{O}_v^{\lambda}=\coprod_{y \in W_L} U^y \dot{y}\dot{v}B \cong \coprod_{y \in W_L}U^y \dot{y}(L \cap \dot{v} \cdot B) = \mathcal{B}(L) \cong L/L\cap \dot{v} \cdot B,\] 
and that the isomorphism in the middle takes $u \dot{y}\dot{v}B \in \mathcal{O}_v^{\lambda}$ to $u \dot{y} (L \cap \dot{v} \cdot B) \in \mathcal{B}(L)$ ($u \in U^y$). If a point $u \dot{y}\dot{v}B \in \mathcal{O}_v^{\lambda}$ is also in $\pi_{I}^{-1}(N)$, by definition, $\dot{v}^{-1} \cdot \dot{y}^{-1} \cdot u^{-1} \cdot N \in I \iff \dot{y}^{-1} \cdot u^{-1} \cdot N \in \dot{v} \cdot I$. Since $\dot{y} \in W_L$ and $u \in U^y \subset L$ and $N \in \mathfrak{g}(2)$, we must have $ \dot{y}^{-1} \cdot u^{-1} \cdot N \in \mathfrak{g}(2)$. Therefore, $\dot{y}^{-1} \cdot u^{-1} \cdot N \in \dot{v} \cdot I \iff \dot{y}^{-1} \cdot u^{-1} \cdot N \in \mathfrak{g}(2)\cap \dot{v} \cdot I$. In summary, $\pi_{I}^{-1}(N) \cap \mathcal{O}_v^{\lambda}=\{\ u \dot{y}\dot{v}B \ |\ \dot{y}^{-1} \cdot u^{-1} \cdot N \in \mathfrak{g}(2)\cap \dot{v} \cdot I \}$. Mapped isomorphically into $\mathcal{B}(L)$, the previous set becomes $\{\ u \dot{y}(L \cap \dot{v} \cdot B)\ |\ \dot{y}^{-1} \cdot u^{-1} \cdot N \in \mathfrak{g}(2)\cap \dot{v} \cdot I \}$, which is exactly $(L, L \cap \dot{v} \cdot B, \mathfrak{g}(2), \mathfrak{g}(2)\cap \dot{v} \cdot I, N)$ by definition. 
\end{proof}

Combining the two lemmas above with Lemma~\ref{3rdtech}, we can make an important step towards the proof of the main theorem. The following proposition is hinted at in the beginning of subsection \ref{roadmap}.

\begin{proposition} \label{reductionstep}
If $\pi_{I}^{-1}(N) \cap \mathcal{O}_{v}^{\lambda}$ is paved by affines, so is $\pi_{I}^{-1}(N) \cap \mathcal{O}_{v}$.
\end{proposition}

\begin{proof}
By Proposition~\ref{favb}, $\mathcal{O}_v$ satisfies all the requirements for $E$ as in Lemma~\ref{3rdtech}. By Lemma~\ref{prehsp}, Lemma~\ref{dist1} and Lemma~\ref{dist2}, $\pi_{I}^{-1}(N) \cap \mathcal{O}_v$ satisfies all the requirements for $Z$ as in Lemma~\ref{3rdtech}. Now apply Lemma~\ref{3rdtech} with $E=\mathcal{O}_v$, $E^{\mathbb{C}^{\times}}=\mathcal{O}_v^{\lambda}$, $Z=\pi_{I}^{-1}(N) \cap \mathcal{O}_v$ and $Z^{\mathbb{C}^{\times}}=\pi_{I}^{-1}(N) \cap \mathcal{O}_v^{\lambda}$. We deduce that $\pi_{I}^{-1}(N) \cap \mathcal{O}_v$ is paved by affines if $\pi_{I}^{-1}(N) \cap \mathcal{O}_{v}^{\lambda}$ is.
\end{proof}

From the proof of Lemma~\ref{3rdtech}, it is not hard to deduce the following corollary, which is useful in the computation in section \ref{g2hif}.

\begin{corollary} \label{dimcor}
Let $q: \pi_{I}^{-1}(N) \cap \mathcal{O}_{v} \longrightarrow \pi_{I}^{-1}(N) \cap \mathcal{O}_{v}^{\lambda}$ be the restriction of $p: \mathcal{O}_v \longrightarrow \mathcal{O}_v^{\lambda}$ and $r=\operatorname{dim}(\pi_{I}^{-1}(N) \cap \mathcal{O}_{v})-\operatorname{dim}(\pi_{I}^{-1}(N) \cap \mathcal{O}_{v}^{\lambda})$ be the relative dimension. We then have the following results:
\begin{itemize}
\item[(1)] If $W \cong \mathbb{A}^l$ is an affine cell of $\pi_{I}^{-1}(N) \cap \mathcal{O}_{v}^{\lambda}$, then $q^{-1}(W)$ admits the structure of a rank $r$ vector bundle over $W$. In particular, $q^{-1}(W) \cong \mathbb{A}^{l+r}$ and it is an affine cell of $\pi_{I}^{-1}(N) \cap \mathcal{O}_{v}$.
\item[(2)] If $r=0$, $\pi_{I}^{-1}(N) \cap \mathcal{O}_{v}$ and $\pi_{I}^{-1}(N) \cap \mathcal{O}_{v}^{\lambda}$ are the same subset of $\pi_{I}^{-1}(N)$.
\end{itemize}
\end{corollary}

The following is the fulcrum of the proof of the main theorem.

\begin{theorem} \label{dclpthm}
Let $G$ be a connected reductive algebraic group over $\mathbb{C}$ whose Lie algebra has no simple component of type $E_7$ or $E_8$. Let $N \in \mathfrak{g}$ be a distinguished nilpotent element and $P$ be the associated parabolic of $N$. We may assume that $P$ contains the pre-selected Borel subgroup $B$. Following the notation in subsection \ref{assparab}, let $P=LU_P$ be the Levi decomposition of $P$ so that $\mathfrak{l}=\mathfrak{g}(0)$ and $\mathfrak{u}_P=\bigoplus_{i > 0}\mathfrak{g}(i)$. $L \cap B$ is a Borel subgroup of $L$. Then for any $(L \cap B)$-stable subspace $U \subset \mathfrak{g}(2)$, the variety $X_U=(L, L \cap B, \mathfrak{g}(2), U, N)$ is paved by affines whenever it is not empty.
\end{theorem}

\begin{proof}
Note that $N \in \mathfrak{g}(2)$ and the $L$-orbit of $N$ is dense in $\mathfrak{g}(2)$, so the $L$-module $\mathfrak{g}(2)$ is indeed prehomogeneous and the quintuple $(L, L \cap B, \mathfrak{g}(2), U, N)$ is valid.

Let $D$ be the image in $G$ of an associated one-parameter subgroup $\lambda$ of $N$. By \cite{de1988homology}[section 3.7], since $N$ is distinguished, there exists a $P$-orbit $\mathcal{O}$ on $\mathcal{B}$ such that $X_U \cong \mathcal{B}_{N, \mathcal{O}}^D$, where $\mathcal{B}_{N, \mathcal{O}}^D$ is the $D$-fixed-point set of the intersection $\mathcal{B}_N \cap \mathcal{O}$. Therefore, as long as we can prove that $\mathcal{B}_{N, \mathcal{O}}^D$ is paved by affines, we are done.

By \cite{de1988homology}[section 3.6], we know that $\mathcal{B}_N^D=\coprod_{\mathcal{O}}\mathcal{B}_{N, \mathcal{O}}^D$ and each piece is smooth projective and they do not meet each other. Therefore, if $\mathcal{B}_N^D$ is paved by affines, so is each piece $\mathcal{B}_{N, \mathcal{O}}^D$.

Next we turn to look at $\mathcal{B}_{N}^D$. In this proof, for an arbitrary reductive Lie algebra $\mathfrak{g}'$, let $\mathcal{B}(\mathfrak{g}')$ denote the flag variety of a connected reductive group whose Lie algebra is $\mathfrak{g}'$. Let $\mathcal{B}_{N'}(\mathfrak{g}')$ denote the Springer fiber for a nilpotent element $N' \in \mathfrak{g}'$. That is, $\mathcal{B}_{N'}(\mathfrak{g}')=\{\ \mathfrak{b}' \in \mathcal{B}(\mathfrak{g}') \ |\ N' \in \mathfrak{b}', \; \mathfrak{b}' \text{ is a Borel subalgebra of } \mathfrak{g}'\}$. Now come back to the group $G$ and Lie algebra $\mathfrak{g}$ we started with, and let $\mathfrak{g} \cong \mathfrak{z} \oplus (\bigoplus_{i=1}^m \mathfrak{g}_i)$ be the decomposition of $\mathfrak{g}$ into a direct sum of its center $\mathfrak{z}$ and simple components $\mathfrak{g}_1, \mathfrak{g}_2, \ldots, \mathfrak{g}_m$. Let $N_i$ be the projection of $N$ onto $\mathfrak{g}_i$ for $i=1, 2, \ldots, m$. Each $N_i$ is a nilpotent element of $\mathfrak{g}_i$ and $N$ is distinguished in $\mathfrak{g}$ if and only if each $N_i$ is distinguished in $\mathfrak{g}_i$. It is not hard to show that $\mathcal{B}_N=\mathcal{B}_N(\mathfrak{g}) \cong \prod_{i=1}^m \mathcal{B}_{N_i}(\mathfrak{g}_i)$ (see \cite{spaltenstein1982classes}[Chapter II, section 1.1]). Because each component $\mathfrak{g}_i$ is an ideal of $\mathfrak{g}$, the adjoint action of $D$ stabilizes $\mathfrak{g}_i$ and induces an action on $\mathcal{B}(\mathfrak{g}_i)$. Since $D$ stabilizes $\mathbb{C}N$, it also stabilizes each $\mathbb{C}N_i$. Then the adjoint action of $D$ on $\mathfrak{g}$ induces actions on $\mathcal{B}_N, \mathcal{B}_{N_1}(\mathfrak{g}_1), \mathcal{B}_{N_2}(\mathfrak{g}_2), \ldots, \mathcal{B}_{N_m}(\mathfrak{g}_m)$. Therefore, the isomorphism $\mathcal{B}_N \cong \prod_{i=1}^m \mathcal{B}_{N_i}(\mathfrak{g}_i)$ is $D$-equivariant and we have $\mathcal{B}_N^D \cong \prod_{i=1}^m \mathcal{B}_{N_i}^D(\mathfrak{g}_i)$.

Now it suffices to show that each $\mathcal{B}_{N_i}^D(\mathfrak{g}_i)$ is paved by affines. Since $\mathfrak{g}$ is assumed to have no simple component of type $E_7$ or $E_8$, each $\mathfrak{g}_i$ is either classical or of type $G_2$, $F_4$ or $E_6$. In the classical case, let $s \in D$ be a semisimple element such that $\mathcal{B}_{N_i}^D(\mathfrak{g}_i)=\mathcal{B}_{N_i}^s(\mathfrak{g}_i)$. By \cite{de1988homology}[Theorem 3.9], $\mathcal{B}_{N_i}^s(\mathfrak{g}_i)$ is paved by affines and so is $\mathcal{B}_{N_i}^D(\mathfrak{g}_i)$. In the other three cases, note that $\mathcal{B}_{N_i}^D(\mathfrak{g}_i)=\coprod_{\mathcal{O}} (\mathcal{B}_{N_i}^D(\mathfrak{g}_i) \cap \mathcal{O})$, and  for each $P_i$-orbit $\mathcal{O}$, $\mathcal{B}_{N_i}^D(\mathfrak{g}_i) \cap \mathcal{O}$ is constructed from some quintuple $(L_i, L_i \cap B_i, \mathfrak{g}_i(2), U_i, N_i)$, where $L_i$ is the Levi factor of the associated parabolic $P_i$ of $N_i$ (in $\mathfrak{g}_i$). Therefore, it suffices to prove that $(L_i, L_i \cap B_i, \mathfrak{g}_i(2), U_i, N_i)$ is paved by affines for a simple Lie algebra $\mathfrak{g}_i$ of type $G_2$, $F_4$ or $E_6$. For type $G_2$, the semisimple rank of $L_i$ is at most 1 (see section \ref{g2hif}), hence each nonempty $(L_i, L_i \cap B_i, \mathfrak{g}_i(2), U_i, N_i)$ is either a finite set of points or $\mathbb{P}^1$, both of which are paved by affines. For type $F_4$, only two of its four distinguished nilpotent orbits need to be carefully inspected. One of them has already been done in \cite{de1988homology}[section 4.2]. For the other orbit, a nonempty $(L_i, L_i \cap B_i, \mathfrak{g}_i(2), U_i, N_i)$ (with $N_i$ from this nilpotent orbit) is one of the following: $\mathbb{P}^1 \times \mathbb{P}^1$, $\mathbb{P}^1$ (or disjoint union thereof), a finite set of points. For type $E_6$, only one orbit needs inspection, for which a nonempty $(L_i, L_i \cap B_i, \mathfrak{g}_i(2), U_i, N_i)$ is one of: $\mathbb{P}^1 \times \mathbb{P}^1 \times \mathbb{P}^1$, a smooth rational surface, $\mathbb{P}^1$ (or disjoint union thereof), a finite set of points. Then knowledge of these very special varieties concludes the proof. The details for the $F_4$ and $E_6$ cases are given in section \ref{F4E6pf}.
\end{proof}

\subsection{Undistinguished nilpotent elements} \label{undist}
The proof of the main theorem when $N$ is distinguished can be done by combining various results from subsection \ref{prehomogeneous}. When $N$ is undistinguished, the classification of nilpotent orbits in \cite{bala1976classes,bala1976classesii} provides us a good way of reducing to the distinguished case. More specifically, we need the following result.

\begin{proposition*}[{\cite{carter1985finite}[p.~172, Proposition 5.9.4]}]
There exists a mininal Levi subalgebra $\mathfrak{m}$ of $\mathfrak{g}$ containing $N$. $N$ is a distinguished nilpotent element of $\mathfrak{m}$.
\end{proposition*}

A Levi subalgebra $\mathfrak{m}$ of $\mathfrak{g}$ is the Lie algebra of the Levi factor of a parabolic subgroup of $G$. Miminal Levi subalgebra is minimal with respect to inclusion. 

Next we describe how to find a minimal Levi subalgebra (and its corresponding Levi subgroup) that contains $N$. The following results are taken from \cite{carter1985finite}[p.~156, Proposition 5.5.9; p.~172, Proposition 5.9.4].

For any nilpotent element $N$, let $\lambda$ be an associated one-parameter subgroup (see subsection \ref{assparab}) and $D$ be the image of $\lambda$ in $G$. Let $F=C_G^{\circ}(N)$ be the connected centralizer of $N$. Let $R$ be the unipotent radical of $F$ and $C=C_F(D)$. We know that: 
\begin{itemize}
\item[(1)] $F=RC$ and $R \cap C=1$.
\item[(2)] $C$ is a connected reductive group.
\end{itemize}
Let $S$ be a maximal torus of $C$ and let $\mathfrak{s}$ be the Lie algebra of $S$. Set $\mathfrak{m}=C_{\mathfrak{g}}(\mathfrak{s})$ and $M=C_G(S)$. Then $\mathfrak{m}$ is a minimal Levi subalgebra that contains $N$. $M$ is the corresponding Levi subgroup of $\mathfrak{m}$ and $N$ is a distinguished nilpotent element of $\mathfrak{m}$. In particular, when $N$ is distinguished, $S$ is the connected center of $G$ and the minimal Levi subalgebra that contains $N$ is $\mathfrak{g}$ itself.

Note that $S \subset C_F(D)$, so $S$ and $D$ commute with each other. $D$ acts on $N$ with weight 2 while $S$ centralizes $N$. Therefore $S$ and $D$ have at most finite intersection. Choose a maximal torus $T$ that contains both $S$ and $D$ and pick a Borel $B \supset T$ so that the associated parabolic $P$ of $N$ is standard. Since $B \supset T \supset S$, $M \cap B$ is a Borel subgroup of $M$ (\cite{humphreys1981linear}[section 22.4]). Let $P=LU_P$ be the Levi decomposition and let $\mu: \mathbb{C}^{\times} \longrightarrow S$ be a one-parameter subgroup so that the $\mu$-fixed points and $S$-fixed points on $\mathcal{B}$ coincide. 

The groups $\lambda$, $D$, $S$, $M$, $T$, $B$, $P$, $L$ and $\mu$ described above will be used in the proof of the main theorem when $N$ is an undistinguished nilpotent element.

\section{Proof of the Main Theorem} \label{maintheorempf}
In this section we prove the main theorem using various results from section \ref{preliminaries}. 

\begin{theorem} \label{mymainthm}
Let $G$ be a connected reductive algebraic group over $\mathbb{C}$ whose Lie algebra has no simple component of type $E_7$ or $E_8$. For any Hessenberg ideal $I \subset \mathfrak{u}$ and any nilpotent element $N \in \mathfrak{g}$, the Hessenberg ideal fiber $\pi_{I}^{-1}(N)$ is paved by affines whenever it is not empty.
\end{theorem}

\begin{proof}
For any nilpotent element $N$, let $P$ be the associated parabolic and $\lambda$ be an associated one-parameter subgroup of $N$. $P=LU_P$ is the Levi decomposition and $W=W_LW^L$. Because $\{\ \pi_{I}^{-1}(N) \cap \mathcal{O}_{v}\ |\ v \in W^L\}$ is a paving of $\pi_{I}^{-1}(N)$, by Lemma~\ref{1st}, it is enough to show that each nonempty piece $\pi_{I}^{-1}(N) \cap \mathcal{O}_{v}$ is paved by affines. By Proposition~\ref{reductionstep}, it suffices to show that each nonempty $\pi_{I}^{-1}(N) \cap \mathcal{O}_{v}^{\lambda}$ is paved by affines. We accomplish this task in two different ways, depending on whether $N$ is distinguished or undistinguished.

For any nilpotent $N$, by Lemma~\ref{dist2}, $\pi_{I}^{-1}(N) \cap \mathcal{O}_v^{\lambda} \cong (L, L \cap \dot{v} \cdot B, \mathfrak{g}(2), \mathfrak{g}(2) \cap \dot{v} \cdot I, N)$. Because $v \in W^L$, $L \cap \dot{v} \cdot B=L \cap B$ (see Lemma~\ref{dimlem}). Then $\pi_{I}^{-1}(N) \cap \mathcal{O}_v^{\lambda} \cong (L, L \cap B, \mathfrak{g}(2),  \mathfrak{g}(2)\cap \dot{v} \cdot I, N)$.

If $N$ is distinguished, by Theorem~\ref{dclpthm}, $\pi_{I}^{-1}(N) \cap \mathcal{O}_v^{\lambda} \cong (L, L \cap B, \mathfrak{g}(2),  \mathfrak{g}(2)\cap \dot{v} \cdot I, N)$ is paved by affines, and we are done.

If $N$ is undistinguished, take all the groups $\lambda$, $D$, $S$, $M$, $T$, $B$, $P$, $L$ and $\mu$ associated to $N$ as in subsection \ref{undist}. For simplicity of notation, we drop the subscript $v$ and use $\mathcal{O}$ to denote any $P$-orbit on $G/B$. Now we show that $\pi_{I}^{-1}(N) \cap \mathcal{O}^{\lambda}$ is paved by affines by making use of the $S$-action on it.

Since $S$ and $D$ act on the flag variety $G/B$ by left multiplication, the two actions commute with each other. Because $S \subset C_G^{\circ }(N)$, it stabilizes $\pi_{I}^{-1}(N)$; because $S$ lies in $T$ and commutes with $D$, it stabilizes $\mathcal{O}^{\lambda}$. Therefore, $S$ stabilizes $\pi_{I}^{-1}(N) \cap \mathcal{O}^{\lambda}$ and $(\pi_{I}^{-1}(N) \cap \mathcal{O}^{\lambda})^S=(\pi_{I}^{-1}(N) \cap \mathcal{O}^{\lambda})^{\mu}$. Lemma~\ref{dist2} and Lemma~\ref{prehsp} imply that $\pi_{I}^{-1}(N) \cap \mathcal{O}^{\lambda}$ is smooth projective. Then we can apply the Lemma~\ref{3rdtechaffinepaving} to $\pi_{I}^{-1}(N) \cap \mathcal{O}^{\lambda}$ with the $\mu$-action. In particular, we deduce that:

The fixed-point set $(\pi_{I}^{-1}(N) \cap \mathcal{O}^{\lambda})^{\mu}$ is the disjoint union of its connected components, each of which is smooth projective. Moreover, if every connected component of $(\pi_{I}^{-1}(N) \cap \mathcal{O}^{\lambda})^{\mu}$ is paved by affines, the same is true of $\pi_{I}^{-1}(N) \cap \mathcal{O}^{\lambda}$. 

Let $\mathscr{P}(D, S)$ be the set of connnected components of $(\pi_{I}^{-1}(N) \cap \mathcal{O}^{\lambda})^{\mu}$ when $\mathcal{O}$ ranges over all $P$-orbits on $G/B$. We want to show that every variety in $\mathscr{P}(D,S)$ is paved by affines. The key lies in viewing the set $\mathscr{P}(D, S)$ in a different way. 

Consider the fixed-point variety $\pi_{I}^{-1}(N)^{D, S}$ by both $D$ and $S$, and let $\mathscr{R}$ be the set of connected components of $\pi_{I}^{-1}(N)^{D, S}$. For each $\pi_{I}^{-1}(N) \cap \mathcal{O}^{\lambda}$, it is a smooth projective closed subvariety of $G/B$ . The intersection of two different pieces $(\pi_{I}^{-1}(N) \cap \mathcal{O}_{v}^{\lambda}) \bigcap (\pi_{I}^{-1}(N) \cap \mathcal{O}_{u}^{\lambda})$ is empty, because the two $P$-orbits $\mathcal{O}_v$ and $\mathcal{O}_u$ do not meet each other. Therefore, the elements of $\mathscr{P}(D, S)$ have to be exactly all connnect components of $\pi_{I}^{-1}(N)^{D, S}$. That is, $\mathscr{P}(D, S)=\mathscr{R}$.

Now we consider $\pi_{I}^{-1}(N)^{D, S}$ in a different manner. 

Let $DS$ be the product group of $D$ and $S$. It is a toral subgroup of $T$. Then $\pi_{I}^{-1}(N)^{D, S}=\pi_{I}^{-1}(N)^{DS}$. Since $L \cap M = C_G(D) \cap C_G(S)=C_G(DS)$, $L \cap M$ is a Levi subgroup of $G$. Let $Q$ be the subgroup of $G$ generated by $L \cap M$ and $B$, then $Q$ is the standard parabolic of which $L \cap M$ is a Levi factor. On the other hand, $L \cap M= C_G(D) \cap M=C_M(D)$. Because $D \subset M$, $L \cap M$ is a Levi subgroup of $M$. By the definition $M=C_G(S)$ and \cite{collingwood1993nilpotent}[Lemma 3.4.4], it is easy to show that the homomorphism $\varphi: \operatorname{SL}_2(\mathbb{C}) \longrightarrow G$ associated to $N$ by the Jacobson-Morozov theorem factors through the subgroup $M$ of $G$. Therefore, $L \cap M=C_M(D)$ is exactly the Levi subgroup of the unique parabolic subgroup of $M$ associated to $N \in \mathfrak{m}$. Then we know that $\mathfrak{l} \cap \mathfrak{m}=\mathfrak{g}_M(0)$, $N \in \mathfrak{g}_M(2)$ and that the $(L \cap M)$-orbit of $N$ in $\mathfrak{g}_M(2)$ is dense. Here $\mathfrak{m}=\bigoplus_{i \in \mathbb{Z}}\mathfrak{g}_M(i)$ is the weight space decomposition of $\mathfrak{m}$ with respect to $D$.

Since $L \cap M$ is the Levi factor of the standard parabolic subgroup $Q$ of $G$, we can decompose the Weyl group $W=W_{L \cap M}W^{L \cap M}$. For each $v \in W^{L \cap M}$, let $\mathscr{O}_v$ be the corresponding $Q$-orbit on $G/B$. Then
\[ \pi_{I}^{-1}(N) \cap \mathscr{O}_v^{DS}=\pi_{I}^{-1}(N) \cap (\coprod_{y \in W_{L \cap M}}U^{yv} \dot{y} \dot{v}B)^{DS}=\pi_{I}^{-1}(N) \cap (\coprod_{y \in W_{L \cap M}}U^{y} \dot{y} \dot{v}B).  \]
For any $u \dot{y}\dot{v}B \in U^y \dot{y}\dot{v}B$, it lies in $\pi_{I}^{-1}(N)$ whenever $\dot{v}^{-1} \cdot \dot{y}^{-1} \cdot u^{-1} \cdot N \in I \iff \dot{y}^{-1} \cdot u^{-1} \cdot N \in \dot{v} \cdot I \iff \dot{y}^{-1} \cdot u^{-1} \cdot N \in \mathfrak{g}_M(2) \cap \dot{v} \cdot I$. The last equivalent condition is due to the fact that $u \dot{y} \in L \cap M=C_M(D)$ and $N \in \mathfrak{g}_M(2)$. Let $B_{L \cap M}=(L \cap M) \cap \dot{v} \cdot B$. It is a Borel subgroup of $L \cap M$. The natural isomorphism from $(\coprod_{y \in W_{L \cap M}}U^{y} \dot{y} \dot{v}B)$ to the flag variety $\mathcal{B}(L \cap M)$ takes $u \dot{y}\dot{v}B$ to $u \dot{y} B_{L \cap M}$. Under this isomorphism, $\pi_{I}^{-1}(N) \cap \mathscr{O}_v^{DS}$ can be identified with $\{\ u \dot{y} B_{L \cap M}\ |\ \dot{y}^{-1} \cdot u^{-1} \cdot N \in \mathfrak{g}_M(2) \cap \dot{v} \cdot I \}$, which is exactly the quintuple $(L \cap M, B_{L \cap M}, \mathfrak{g}_M(2), \mathfrak{g}_M(2) \cap \dot{v} \cdot I,N)$. Because $v \in W^{L \cap M}$, we have $B_{L \cap M}=(L \cap M) \cap (M \cap B)$ (see Lemma~\ref{dimlem}). Because $M=C_G(S)$ is a Levi subgroup of $G$, its Dynkin diagram is obtained from that of $G$ by removing certain nodes. Therefore, $M$ is also a connected reductive algebraic group whose Lie algebra has no simple component of type $E_7$ or $E_8$. Now apply Theorem~\ref{dclpthm} to the group $M$, its Borel subgroup $M \cap B$, the distinguished nilpotent element $N \in \mathfrak{m}$ and the quintuple $(L \cap M, B_{L \cap M}, \mathfrak{g}_M(2), \mathfrak{g}_M(2) \cap \dot{v} \cdot I,N)$, and we deduce that $\pi_{I}^{-1}(N) \cap \mathscr{O}_v^{DS} \cong (L \cap M, B_{L \cap M}, \mathfrak{g}_M(2), \mathfrak{g}_M(2)\cap \dot{v} \cdot I,N)$ is paved by affines. By Lemma~\ref{dist2} and Lemma~\ref{prehsp}, $\pi_{I}^{-1}(N) \cap \mathscr{O}_v^{DS}$ is smooth projective. Therefore it is a finite disjoint union of connected components from $\mathscr{R}=\mathscr{P}(D,S)$. Since $\pi_{I}^{-1}(N) \cap \mathscr{O}_v^{DS}$ is paved by affines, so is every one of its connected components. Because the collection of $\pi_{I}^{-1}(N) \cap \mathscr{O}_v^{DS}$ for different orbits $\mathscr{O}_v$ cover $\pi_{I}^{-1}(N)^{DS}$, every connected component from $\mathscr{R}=\mathscr{P}(D, S)$ belongs to some $\pi_{I}^{-1}(N) \cap \mathscr{O}_v^{DS}$, hence has to be paved by affines as well. Then we have finished the proof.
\end{proof}

\section{Type \texorpdfstring{$G_2$}{}} \label{g2hif}
Throughout this section, let $G$ be a connected algebraic group over $\mathbb{C}$ of type $G_2$. For every Hessenberg ideal fiber $\pi_{I}^{-1}(N)$, we explicitly describe each cell of $\pi_{I}^{-1}(N)$ as an affine subspace of some Schubert cell of $G/B$. Various geometric properties of Hessenberg ideal fibers can be deduced from these explicit cell structures. In particular, Theorem~\ref{aninterestinghif} shows that the irreducible components of $\pi_{I}^{-1}(N)$ are not always of the same dimension. 

\subsection{Some structures of $G_2$} \label{strofg2}
First we collect some well-known results about $G_2$ that are relevent to our computation.

Fix a Borel subgroup $B$ of $G$ and choose a maximal torus $T \subset B$. Let $\alpha$ and $\beta$ be the short and long simple roots respectively. The root system of $G_2$ is shown in Figure~\ref{g2rsall}. The labeled arrows correspond to all the positive roots. We see that all the roots of $G_2$ have only two different lengths, and we call them short roots and long roots respectively. 

\begin{figure}[h]
\centering 
\includegraphics[scale=.4]{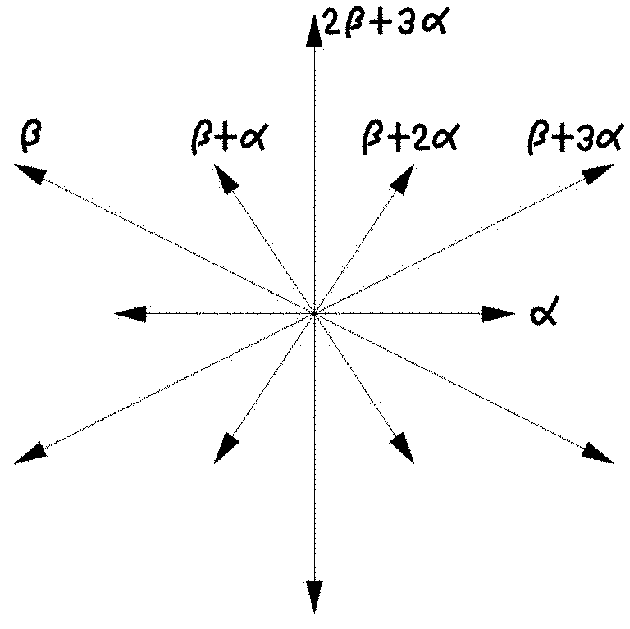}
\caption{Root system of $G_2$}
\label{g2rsall} 
\end{figure}

There is a natural partial order on the set of positive roots: $\gamma \le \delta$ if $\delta-\gamma$ is a linear combination of $\alpha$ and $\beta$ with nonnegative coefficients.

Define $I_{\gamma}=\bigoplus_{\delta \ge \gamma} \mathfrak{g}_{\delta}$ for every $\gamma \in \Phi^{+}$, $I_{\alpha, \beta}=\mathfrak{u}$ and $I_{\emptyset}=(0)$. It is clear that these are all the Hessenberg ideals of $G_2$. There is a natural partial order on the set of Hessenberg ideals by inclusion.

Let $E \cong \mathbb{R}^2$ be the real vector space that the root system of $G_2$ spans. We view $E$ as the plane that contains all the arrows in Figure~\ref{g2rsall}. Let $s=s_{\alpha}$ and $t=s_{\beta}$ be the reflections of $E$ associated to the simple roots $\alpha$ and $\beta$ respectively. Note that $s$ is the reflection about the line through the arrow $2\beta+3\alpha$ and $t$ is the reflection about the line through the arrow $\beta+2\alpha$. Let $r=st$ and it is a rotation of $E$ in the counterclockwise direction for 60 degrees. Let $W$ be the Weyl group of $G_2$ and we have the following presentation of $W$:
\[ W= \langle s, \, t | \, s^2=t^2=(st)^6=e \rangle, \]
where $e$ is the unit element of $W$. Clearly, $W \cong D_6$.

The rest of this subsection consists of several results that can be easily proved by knowledge of the Chevalley groups. They are true for all connected simple algebraic groups over $\mathbb{C}$.

For any root $\gamma \in \Phi$, let $X_{\gamma}$ be the 1-dimensional unipotent subgroup of $G$ whose Lie algebra is the root space $\mathfrak{g}_{\gamma}$. Let $x_{\gamma}: \mathbb{C} \longrightarrow X_{\gamma}$ be the group isomorphism defined via the exponential map and the choice of a Chevalley basis element in $\mathfrak{g}_{\gamma}$. For each positive root $\gamma$, choose a nonzero vector $E_{\gamma} \in \mathfrak{g}_{\gamma}$. Let $E_{-\gamma}$ be the unique vector in $\mathfrak{g}_{-\gamma}$ so that $\left\{E_{\gamma}, \left[E_{\gamma}, E_{-\gamma}\right], E_{-\gamma}\right\}$ is an $\mathfrak{sl}_2$-triple. Let $H_{\gamma}= \left[E_{\gamma}, E_{-\gamma}\right]$.

\begin{lemma} \label{flowlem}
For any two roots $\gamma$ and $\delta$, there exist nonzero complex numbers $c_1, c_2, \ldots, c_q$ depending only on $\gamma$, $\delta$ and $G$ such that 
\[ x_{\gamma}(z) \cdot E_{\delta}=E_{\delta}+\sum_{n=1}^q c_n z^n E_{\delta+n\gamma} \text{  for any } z \in \mathbb{C}.  \]
$x_{\gamma}(z) \cdot E_{\delta}$ is the adjoint action and  $\delta+q\gamma$ is the last root in the $\gamma$-string that goes through $\delta$.
\end{lemma}

\begin{lemma} \label{commlem}
For any two roots $\gamma$ and $\delta$ such that $\gamma+\delta$ is not in $\Phi \cup \left\{0\right\}$, 
\[ x_{\gamma}(z)x_{\delta}(z')=x_{\delta}(z')x_{\gamma}(z) \text{  for any } z, \, z' \in \mathbb{C}.  \]
That is, the two groups $X_{\gamma}$ and $X_{\delta}$ commute with each other.
\end{lemma}

\begin{lemma} \label{permlem}
For any $w \in W$ and $\gamma \in \Phi$, there exists a nonzero complex number $c$ such that
\[ \dot{w}x_{\gamma}(z)\dot{w}^{-1}=x_{w(\gamma)}(cz) \text{  for any } z \in \mathbb{C}. \]
\end{lemma}

\begin{lemma} \label{prodlem}
For any $w \in W$, the following equality of sets is also an isomorphism of algebraic varieties.
\[ U^w=\prod_{\gamma \in \Phi_w}X_{\gamma},  \]
in which factors on the right hand side are multiplied with respect to a fixed order of roots in $\Phi_w$. Moreover, if $\gamma_1+\gamma_2$ is not in $\Phi \cup \left\{0\right\} $ for any two roots $\gamma_1, \, \gamma_2 \in \Phi_w$, factors can be interchanged freely without changeing their product.
\end{lemma}

\subsection{Outline of the algorithm} \label{algorithm}
In this subsection, we outline the algorithm of the computation for the cell structures of all Hessenberg ideal fibers. The idea comes directly from the proof of Theorem~\ref{mymainthm}. If $N$ is the trivial nilpotent orbit $\left\{0\right\}$, then $\pi_{I}^{-1}(0)=G/B$, which is paved by the Schubert cells. We therefore ignore this case and only compute $\pi_{I}^{-1}(N)$ for $N$ from a nontrivial nilpotent orbit. 

Note that in the presentation of $G_2$ in subsection \ref{strofg2}, we have chosen a Borel subgroup $B$ and a maximal torus $T$ to begin with. Therefore, for each nontrivial nilpotent orbit, we now have to first pick a representative $N$ so that its associated parabolic $P$ is standard. Let $\lambda$ be an associated one-parameter subgroup of $N$, $P=LU_P$ the Levi decomposition, and $W=W_L W^L$ as in Lemma~\ref{wl1}. For each $v \in W^L$, let $\mathcal{O}_v$ be the $P$-orbit corresponding to $v$. For any Hessenberg ideal $I$, by Lemma~\ref{dist1} and Lemma~\ref{dist2}, we have 
\[ \pi_{I}^{-1}(N) \cap \mathcal{O}_{v}=(P, P \cap \dot{v} \cdot B, \mathfrak{u}_P^{\ge 2}, \mathfrak{u}_P^{\ge 2}\cap \dot{v} \cdot I, N),  \]
\[ \pi_{I}^{-1}(N) \cap \mathcal{O}_{v}^{\lambda} \cong (L, L \cap \dot{v} \cdot B, \mathfrak{g}(2), \mathfrak{g}(2) \cap \dot{v} \cdot I, N).  \]
Clearly $\pi_{I}^{-1}(N) \cap \mathcal{O}_{v}$ and $\pi_{I}^{-1}(N) \cap \mathcal{O}_{v}^{\lambda}$ are empty or nonempty at the same time. In the case of type $G_2$, for every nontrivial nilpotent orbit, $L$ has semisimple rank at most 1 and $\mathfrak{g}(2)$ is always of small dimension. Therefore it is easy to check whether $\pi_{I}^{-1}(N) \cap \mathcal{O}_{v}^{\lambda}$ is empty or not. To simplify notation, set $P(N, I, v)= \pi_{I}^{-1}(N) \cap \mathcal{O}_v$ and $L(N, I, v)=\pi_{I}^{-1}(N) \cap \mathcal{O}_{v}^{\lambda}$.

For any two Hessenberg ideals $J \subset I$, it is obvious from definition that $L(N, J, v) \subset L(N, I, v)$ and $P(N, J, v) \subset P(N, I, v)$. This simple observation is very useful in telling which $L(N, I, v)$ is nonempty. In fact, in the case of type $G_2$, when $\operatorname{dim}(I/J)=1$, it happens quite often that the cells of $P(N, J, v)$ are still cells of $P(N, I, v)$ (there are exceptions). 

According to \cite{bala1976classesii}[p.~6], $G_2$ has 4 nontrivial nilpotent orbits, denoted by $A_1$, $\tilde{A_1}$, $G_2(a_1)$ and $G_2$. They are ordered in increasing dimensions. Our algorithm is the following:

\begin{itemize}
\item[(1)] Starting with the orbit $A_1$, apply steps (2) to (5). Then repeat the same process for $\tilde{A_1}$, $G_2(a_1)$ and $G_2$ in that order. 
\item[(2)] For each nontrivial nilpotent orbit, choose a representative $N$ so that its associated parabolic $P$ is standard.
\item[(3)] Fixing the $N$, compute $\pi_{I}^{-1}(N)$ for all nonzero Hessenberg ideals ranging from the smallest $I_{2\beta+3\alpha}$ to the biggest $I_{\alpha, \beta}$ by the next two steps.
\item[(4)] For each Hessenberg ideal $I$, find all $v \in W^L$ so that $L(N, I, v) \neq \emptyset$.
\item[(5)] For each $v$ obtained from step (4), compute the cell structure of $P(N, I, v)$ explicitly. 
\end{itemize}
Note that when $I$ is the zero Hessenberg ideal $I_{\emptyset}$, $\pi_{I}^{-1}(0) \cong G/B$ is the only nonempty Hessenberg ideal fiber, so we omit it from the algorithm.

The following lemma will be useful in our computation. It is true for any connected reductive algebraic group over $\mathbb{C}$.

\begin{lemma} \label{dimlem}
Let $P$ be a standard parabolic of $G$, and $P=LU_P$ be its Levi decomposition. $W=W_LW^L$. For any $v \in W^L$, we have the following:
\begin{itemize}
\item[(1)] $\mathfrak{p} \cap \dot{v} \cdot \mathfrak{b}=\mathfrak{t} \oplus (\bigoplus_{\gamma \in \Phi^+ \setminus \Phi_v}\mathfrak{g}_{\gamma}$). 
\item[(2)] $L \cap \dot{v} \cdot B= L \cap B$.
\item[(3)] $\operatorname{dim}(P/P \cap \dot{v} \cdot B)=|\Phi^-(L)|+|\Phi_v|$.
\end{itemize}
\end{lemma}

\begin{proof}
Straightforward from the definition of $W^L$ and $\Phi_v$.
\end{proof}

\subsection{Computation}
As mentioned in algorithm step (1), we split the computation into four cases, one for each nontrivial nilpotent orbit.

\subsubsection{The case of $A_1$}
First we choose a nice representative $N$ that satisfies the requirement in step (2), and describe its associated parabolic $P$, associated one-parameter subgroup $\lambda$, the Levi subgroup $L$ of $P$, $\mathfrak{g}(2)$, $\mathfrak{u}_P^{\ge 2}$ and the decomposition $W=W_LW^L$.

According to \cite{bala1976classesii}, ``$A_1$'' represents the undistinguished nilpotent orbit of $G_2$ every element of which lies in a minimal Levi subalgebra of type $A_1$. For a suitable representative of the orbit $A_1$, the root system of its minimal Levi subalgebra consists of a pair of opposite long roots of $G_2$. $N=E_{2\beta+3\alpha}$ is a representative satisfying the requirement of step (2). We use $N$ to denote $E_{2\beta+3\alpha}$ in this case for the sake of simplicity.

Next we justify the claim briefly. $N=E_{2\beta+3\alpha}$ is a regular nilpotent element of the Levi subalgebra $\mathfrak{m}=\mathfrak{t} \oplus \mathfrak{g}_{-2\beta-3\alpha} \oplus \mathfrak{g}_{2\beta+3\alpha}$, so it belongs to the orbit $A_1$. The more subtle part is to show that the unique associated parabolic of $N$ is standard. For this, it suffices to find an $\mathfrak{sl}_2$-triple $\left\{N, H, Y\right\}$ so that $H \in \mathfrak{t}$ and $\alpha(H),\, \beta(H) \ge 0$. Clearly, $\left\{E_{2\beta+3\alpha}, H_{2\beta+3\alpha}, E_{-2\beta-3\alpha}\right\}$ is such an $\mathfrak{sl}_2$-triple. In this case, $\alpha(H_{2\beta+3\alpha})=0$ and $\beta(H_{2\beta+3\alpha})=1$ (these two numbers are always the same as the weights of $\alpha$ and $\beta$ in the weighted Dynkin diagram of the nilpotent orbit).

Now we can easily see that:
\begin{itemize}
\item[(1)] $P=\langle B, X_{-\alpha}\rangle$ where $\langle B, X_{-\alpha}\rangle$ denotes the subgroup of $G$ generated by $B$ and $X_{-\alpha}$. 
\item[(2)] $\langle \lambda, \alpha\rangle=0$ and $\langle \lambda, \beta\rangle=1$.
\item[(3)] $L=\langle X_{-\alpha}, T, X_{\alpha}\rangle$.
\item[(4)] $\mathfrak{g}(2)=\mathfrak{u}_P^{\ge 2}=\mathbb{C} E_{2\beta+3\alpha}$.
\item[(5)] $W_L=\left\{e, s\right\}$ and $W^L=\left\{e, t, ts, sr^2, sr^3, r^4 \right\}$. (In this case, $v \in W^L \iff v^{-1}(\alpha) \in \Phi^{+}$.)
\end{itemize}
Next we enumerate the Hessenberg ideal $I$ from the smallest $I_{2\beta+3\alpha}$ to the biggest $I_{\alpha, \beta}$, and compute $\pi_{I}^{-1}(N)$ by steps (4) and (5). $N$ is understood to be $E_{2\beta+3\alpha}$ all the time. We know that
\begin{equation*} 
\begin{aligned}
P(N, I, v)&=(P, P \cap \dot{v} \cdot B, \mathbb{C}E_{2\beta+3\alpha}, \mathbb{C}E_{2\beta+3\alpha} \cap \dot{v} \cdot I,N), \\
L(N, I, v)&\cong(L, L \cap B, \mathbb{C}E_{2\beta+3\alpha}, \mathbb{C}E_{2\beta+3\alpha} \cap \dot{v} \cdot I, N).  
\end{aligned}
\end{equation*}
(Note that $L \cap \dot{v} \cdot B =L \cap B$ for any $v \in W^L$ by Lemma~\ref{dimlem}.) Since $\mathfrak{g}(2)=\mathbb{C} E_{2\beta+3\alpha}$, 
$L(N, I, v) \neq \emptyset \iff \mathbb{C}E_{2\beta+3\alpha} \cap \dot{v} \cdot I=\mathbb{C}E_{2\beta+3\alpha} \iff v$ sends some root of $I$ to $2\beta+3\alpha$. When this happens, $L(N, I, v)=\mathcal{O}_v^{\lambda}$ and $P(N, I, v)=\mathcal{O}_v=P\dot{v}B$. That is, both are equal to their biggest possibility.

\noindent $\bullet$ $I=I_{2\beta+3\alpha}$

Since $I$ has only one root $2\beta+3\alpha$, for $L(N, I, v)$ to be nonempty, $v$ has to send $2\beta+3\alpha$ to itself. Then the only possibility is $v=e$. Therefore, $\pi_{I}^{-1}(N)=P(N, I, e)=P\dot{e}B=X_s \amalg X_e=\overline{X_s} \cong \mathbb{P}^1$. $\pi_{I}^{-1}(N)$ is the Schubert variety $\overline{X_s}$. It has one 0-cell and one 1-cell.

\noindent $\bullet$ $I=I_{\beta+3\alpha}$

Now that $I$ has one more long root $\beta+3\alpha$ than the previous ideal $I_{2\beta+3\alpha}$, $v$ has to send either of them to $2\beta+3\alpha$. Then $v=e$ or $t$. $P(N, I, e)=P\dot{e}B=\overline{X_s}$. $P(N, I, t)=P\dot{t}B=X_t \amalg X_{st}$. Then $\pi_{I}^{-1}(N)=\overline{X_s} \amalg X_t \amalg X_{st}= \overline{X_{st}}$. It is a Schubert variety of dimension 2.

\noindent $\bullet$ $I=I_{\beta+2\alpha }$

$I$ has one more short root $\beta+2\alpha$ than $I_{\beta+3\alpha}$, and the action of any $w \in W$ on $\Phi$ preserves the lengths of roots. Therefore, $v$ can still only be $e$ or $t$. Then $\pi_{I}^{-1}(N)=P(N, I, e) \amalg P(N, I, t)=P\dot{e}B \amalg P\dot{t}B=\overline{X_{st}}$. It is the same as the previous one.

\noindent $\bullet$ $I=I_{\beta+\alpha }$

$I$ has one more short root $\beta+\alpha$ than $I_{\beta+2\alpha}$, by a similar argument as above, $\pi_{I}^{-1}(N)$ is still $\overline{X_{st}}$.

\noindent $\bullet$ $I=I_{\alpha }$

One more short root $\alpha$, so $\pi_{I}^{-1}(N)=\overline{X_{st}}$.

\noindent $\bullet$ $I=I_{\beta }$

The previous ideal (by the partial order of inclusion) of $I$ is $I_{\beta+\alpha}$. $I$ has one more long root $\beta$ and it leads to a new possibility for $v$. Now $v=e$, $t$ or $ts$. $\pi_{I}^{-1}(N)=P(N, I, e) \amalg P(N, I, t) \amalg P(N, I, ts)=P\dot{e}B \amalg P\dot{t}B \amalg P\dot{t}\dot{s}B=\overline{X_{sts}}$. It is a Schubert variety of dimension 3.

\noindent $\bullet$ $I=I_{\alpha, \beta }$

$I$ has one more short root $\alpha$ than $I_{\beta}$, so $\pi_{I}^{-1}(N)=\pi_{I_{\beta}}^{-1}(N)=\overline{X_{sts}}$. This one is a Springer fiber.

We have finished the case of $A_1$.

\subsubsection{The case of $\tilde{A_1}$}
The notation ``$\tilde{A_1}$'' represents the undistinguished nilpotent orbit of $G_2$, every element of which lies in a minimal Levi subalgebra of type $A_1$ as well, but this time, for a suitable representative of $\tilde{A_1}$, the root system of its minimal Levi subalgebra consists of a pair of opposite short roots of $G_2$. $N=E_{\beta+2\alpha}$ is a representative satisfying the requirement of step (2). We use $N$ to denote $E_{\beta+2\alpha}$ in this case.

First we justify the choice of $N=E_{\beta+2\alpha}$. $N$ is a regular nilpotent element of the Levi subalgebra $\mathfrak{m}=\mathfrak{t} \oplus \mathfrak{g}_{-\beta-2\alpha} \oplus \mathfrak{g}_{\beta+2\alpha}$, hence belongs to $\tilde{A_1}$. $\left\{N, H_{\beta+2\alpha}, E_{-\beta-2\alpha}\right\}$ is an $\mathfrak{sl}_2$-triple such that $H_{\beta+2\alpha} \in \mathfrak{t}$, $\alpha(H_{\beta+2\alpha})=1$ and $\beta(H_{\beta+2\alpha})=0$. Then the associated parabolic of $N$ is standard.

We know that:
\begin{itemize}
\item[(1)] $P=\langle B, X_{-\beta}\rangle$. 
\item[(2)] $\langle \lambda, \alpha\rangle=1$ and $\langle \lambda, \beta\rangle=0$.
\item[(3)] $L=\langle X_{-\beta}, T, X_{\beta}\rangle$.
\item[(4)] $\mathfrak{g}(2)=\mathbb{C} E_{\beta+2\alpha}$ and $\mathfrak{u}_P^{\ge 2} =\operatorname{span}_{\mathbb{C}} \left\{E_{\beta+2\alpha}, E_{\beta+3\alpha}, E_{2\beta+3\alpha}\right\}$.
\item[(5)] $W_L=\left\{e, t\right\}$ and $W^L=\left\{e, s, st, r^2, tr^3, tr^4 \right\}$. (In this case, $v \in W^L \iff v^{-1}(\beta) \in \Phi^{+}$.)
\end{itemize}
Since $\mathfrak{g}(2)=\mathbb{C}E_{\beta+2\alpha}$ is still one dimensional, $L(N, I, v) \neq \emptyset \iff v$ sends some root of $I$ to $\beta+2\alpha$. When this happens, $L(N, I, v)=\mathcal{O}_v^{\lambda}$, but $P(N, I, v)$ could be strictly smaller that $\mathcal{O}_v=P\dot{v}B$. In this situation, it is usually helpful to compare $P(N, I, v)$ with $P(N, J, v)$ where $J$ is the previous Hessenberg ideal (meaning a subspace of codimensin 1) of $I$. Quite often, they are equal to each other.

\noindent $\bullet$ $I=I_{2\beta+3\alpha}$ or $I_{\beta+3\alpha}$

Since these two ideals only have long roots, $\mathfrak{g}(2) \cap \dot{v} \cdot I=0$ for any $v \in W^L$. Then $\pi_{I}^{-1}(N)=\emptyset$ for both ideals.

\noindent $\bullet$ $I=I_{\beta+2\alpha }$

Since $I$ has only one short root $\beta+2\alpha$, $v$ has to fix $\beta+2\alpha$. Then it could only be $e$. Because $ \mathfrak{u}_P^{\ge 2} \cap \dot{e} \cdot I=\mathfrak{u}_P^{\ge 2}$, $P(N, I, e)=P\dot{e}B=BW_LB=X_e \amalg X_t=\overline{X_t} \cong \mathbb{P}^1$. It is a Schubert variety of dimension 1.

\noindent $\bullet$ $I=I_{\beta+\alpha }$

$I$ has one more short root $\beta+\alpha$ than $I_{\beta+2\alpha}$, so $v=e$ or $s$. When $v=e$, since $\mathfrak{u}_P^{\ge 2} \cap \dot{e} \cdot I=\mathfrak{u}_P^{\ge 2} \cap \dot{e} \cdot I_{\beta+2\alpha}$ , $P(N, I, e)=P(N, I_{\beta+2\alpha}, e)=\overline{X_t}$. As for $v=s$, $L(N, I, s)=\mathcal{O}_s^{\lambda} \cong L/L \cap B$. Because $\mathfrak{u}_P^{\ge 2} \cap \dot{s} \cdot I=\mathbb{C}E_{\beta+2\alpha} \oplus \mathbb{C}E_{2\beta+3\alpha}$, by Lemma~\ref{prehsp} and Lemma~\ref{dimlem}, 
\[\operatorname{dim}P(N, I, s)=|\Phi^-(L)|+|\Phi_s|-\operatorname{dim}(\mathfrak{u}_P^{\ge 2}/\mathfrak{u}_P^{\ge 2} \cap \dot{s} \cdot I)=1+1-1=\operatorname{dim}L(N, I, s).\]

By Corollary~\ref{dimcor}, 
\[ P(N, I, s)=L(N, I, s)=\mathcal{O}_s^{\lambda}=U^e \dot{s}B \amalg U^t \dot{t}\dot{s}B=\left\{\dot{s}B\right\} \amalg X_{\beta}\dot{t}\dot{s}B=\overline{X_{\beta}\dot{t}\dot{s}B} \cong \mathbb{P}^1. \] Then $\pi_{I}^{-1}(N)=\overline{X_t} \amalg \overline{X_{\beta}\dot{t}\dot{s}B} \cong \mathbb{P}^1 \amalg \mathbb{P}^1$. This is our first example which is not a Schubert variety.

\noindent $\bullet$ $I=I_{\beta }$

$I$ has one more long root $\beta$ than $I_{\beta+\alpha}$, so $v=e$ or $s$. $P(N, I, e)=P(N, I_{\beta+\alpha}, e)=\overline{X_t}$. Since $\mathfrak{u}_P^{\ge 2} \cap \dot{s} \cdot I=\mathfrak{u}_P^{\ge 2}$ this time, $P(N, I, s)=P\dot{s}B=X_{ts} \amalg X_s$. Then $\pi_{I}^{-1}(N)=\overline{X_t} \amalg X_{ts} \amalg X_s=\overline{X_{ts}}$. It is a Schubert variety of dimension 2. Note that $\pi_{I_{\beta+\alpha}}^{-1}(N)$ is naturally included in $\pi_{I_{\beta}}^{-1}(N)$, but the 1-cell $X_{\beta}\dot{t}\dot{s}B$ of the former is no longer a cell of the latter.

\noindent $\bullet$ $I=I_{\alpha }$

$I$ has one more short root $\alpha$ than $I_{\beta+\alpha}$, so $v$ has one more possibility and $v=e, s$ or $r$. $P(N, I, e)=P(N, I_{\beta+\alpha}, e)=\overline{X_t}$ and $P(N, I, s)=P(N, I_{\beta+\alpha}, s)=\overline{X_{\beta}\dot{t}\dot{s}B}$.

As for $v=r$, $\mathfrak{u}_P^{\ge 2} \cap \dot{r} \cdot I=\mathbb{C}E_{\beta+2\alpha} \oplus \mathbb{C}E_{2\beta+3\alpha}$. By Lemma~\ref{dimlem} and Lemma~\ref{prehsp}, $\operatorname{dim}P(N, I, r)=|\Phi^-(L)|+|\Phi_r|-\operatorname{dim}(\mathfrak{u}_P^{\ge 2}/ \mathfrak{u}_P^{\ge 2} \cap \dot{r} \cdot I)=2$. By Corollary~\ref{dimcor}, $P(N, I, r)$ has one 1-cell and one 2-cell, each being a respective rank 1 vector bundle over the 0-cell and 1-cell of $L(N, I, r)=\mathcal{O}_r^{\lambda} \cong \mathbb{P}^1$. Next we describe the cells of $P(N, I, r)$ precisely.

Recall that $P(N, I, r)=\pi_{I}^{-1}(N) \cap \mathcal{O}_r$ and $\mathcal{O}_r=P\dot{r}B=U^r\dot{r}B \amalg U^{tr}\dot{t}\dot{r}B$. By Lemma~\ref{prodlem}, $U^r=X_{\alpha}X_{\beta+3\alpha}$ and $U^{tr}=X_{\beta+\alpha}X_{\beta}X_{2\beta+3\alpha}$. 
Now we are able to inspect $\pi_{I}^{-1}(N) \cap U^r\dot{r}B$ and $\pi_{I}^{-1}(N) \cap U^{tr}\dot{t}\dot{r}B$. 

For $u\dot{r}B \in U^r\dot{r}B$ to be in $\pi_{I}^{-1}(N)$, we need $u^{-1} \cdot N \in \mathfrak{u}_P^{\ge 2} \cap \dot{r} \cdot I$. Here $u \in U^r$ takes the form $u=x_{\alpha}(z)x_{\beta+3\alpha}(z')$ for some $z, z' \in \mathbb{C}$. By Lemma~\ref{flowlem}, there exists a nonzero constant $c$ such that 
\[ u^{-1} \cdot N=x_{\beta+3\alpha}(-z') \cdot x_{\alpha}(-z) \cdot E_{\beta+2\alpha}=E_{\beta+2\alpha}-czE_{\beta+3\alpha}  \]
for any $z, z' \in \mathbb{C}$. Then $E_{\beta+2\alpha}-czE_{\beta+3\alpha} \in \mathfrak{u}_P^{\ge 2} \cap \dot{r} \cdot I=\mathbb{C}E_{\beta+2\alpha} \oplus \mathbb{C}E_{2\beta+3\alpha}$ if and only if $z=0$. Therefore, $\pi_{I}^{-1}(N) \cap U^r\dot{r}B=X_{\beta+3\alpha}\dot{r}B \subset X_{r}$.

For $u\dot{t}\dot{r}B \in U^{tr}\dot{t}\dot{r}B$ to be in $\pi_{I}^{-1}(N)$, we need $\dot{t}^{-1} \cdot u^{-1} \cdot N \in \mathfrak{u}_P^{\ge 2} \cap \dot{r} \cdot I$. Here $u \in U^{tr}$ takes the form $u=x_{\beta+\alpha}(z)x_{\beta}(z')x_{2\beta+3\alpha}(z'')$ for some $z, z', z'' \in \mathbb{C}$. By Lemma~\ref{flowlem}, there exists a nonzero constant $c'$ such that 
\[ \dot{t}^{-1} \cdot u^{-1} \cdot N=\dot{t} \cdot x_{2\beta+3\alpha}(-z'') \cdot x_{\beta}(-z') \cdot x_{\beta+\alpha}(-z) \cdot E_{\beta+2\alpha}=\dot{t} \cdot E_{\beta+2\alpha}-c'z(\dot{t} \cdot E_{2\beta+3\alpha})  \]
for any $z, z', z'' \in \mathbb{C}$. Then $\dot{t} \cdot E_{\beta+2\alpha}-c'z(\dot{t} \cdot E_{2\beta+3\alpha}) \in \mathfrak{u}_P^{\ge 2} \cap \dot{r} \cdot I$ if and only if $z=0$. Therefore, $\pi_{I}^{-1}(N) \cap U^{tr}\dot{t}\dot{r}B=X_{\beta}X_{2\beta+3\alpha}\dot{t}\dot{r}B \subset X_{tr}$.

In addition, $L(N, I, r)=\mathcal{O}_r^{\lambda}=\left\{\dot{r}B\right\} \amalg X_{\beta}\dot{t}\dot{r}B$. Clearly, $X_{\beta+3\alpha}\dot{r}B$ contains and lies above $\left\{\dot{r}B\right\}$ as a rank 1 vector bundle. The same is true with $X_{\beta}X_{2\beta+3\alpha}\dot{t}\dot{r}B$ and $X_{\beta}\dot{t}\dot{r}B$. This is compatible with Corollary~\ref{dimcor}. 

In summary, $\pi_{I}^{-1}(N)=\overline{X_{t}} \amalg \overline{X_{\beta}\dot{t}\dot{s}B} \amalg X_{\beta+3\alpha}\dot{r}B \amalg X_{\beta}X_{2\beta+3\alpha}\dot{t}\dot{r}B$. It has two 0-cells, three 1-cells and one 2-cell. Further geometric properties of this $\pi_{I}^{-1}(N)$ will be discussed in later part of this section (Theorem~\ref{aninterestinghif}).

\noindent $\bullet$ $I=I_{\alpha, \beta }$

$I$ has one more long root $\beta$ than $I_{\alpha}$, so $v$ can still only be $e$, $s$ or $r$. Since $\mathfrak{u}_P^{\ge 2} \cap \dot{e} \cdot I=\mathfrak{u}_P^{\ge 2} \cap \dot{s} \cdot I=\mathfrak{u}_P^{\ge 2}$, $P(N, I, e)=P\dot{e}B=\overline{X_{t}}$ and $P(N, I, s)=P\dot{s}B=X_s \amalg X_{ts}$. Moreover, $P(N, I, r)=P(N, I_{\alpha}, r)=X_{\beta+3\alpha}\dot{r}B \amalg X_{\beta}X_{2\beta+3\alpha}\dot{t}\dot{r}B$. Therefore, $\pi_{I}^{-1}(N)=\overline{X_{ts}} \amalg X_{\beta+3\alpha}\dot{r}B \amalg X_{\beta}X_{2\beta+3\alpha}\dot{t}\dot{r}B$. This is a Springer fiber of dimension 2. 

We have finished the case of $\tilde{A_1}$.

\subsubsection{The case of $G_2(a_1)$}
The notation ``$G_2(a_1)$'' represents the distinguished nilpotent orbit of $G_2$ whose associated parabolic has semisimple rank 1. $N=E_{\beta+\alpha}+E_{\beta+3\alpha}$ is a representative that satisfies the requirement of step (2). We use $N$ to denote $E_{\beta+\alpha}+E_{\beta+3\alpha}$ in this case.

Now we justify the choice of $N=E_{\beta+\alpha}+E_{\beta+3\alpha}$. Set $H=\left[E_{\beta+\alpha}+E_{\beta+3\alpha}, E_{-\beta-\alpha}+E_{-\beta-3\alpha}\right]$. Simple computation shows that $\left\{N, H, E_{-\beta-\alpha}+E_{-\beta-3\alpha}\right\}$ is an $\mathfrak{sl}_2$-triple such that $H \in \mathfrak{t}$, $\alpha(H)=0$ and $\beta(H)=2$. Then the associated parabolic of $N$ is standard. In fact, we know that $P=\langle B, X_{-\alpha}\rangle$ and $\operatorname{dim}\mathfrak{g}(0)=\operatorname{dim}\mathfrak{g}(2)=4$, so $N$ is a distinguished nilpotent element whose associated parabolic is not the Borel subgroup $B$. Therefore $N$ has to be in $G_2(a_1)$.

We know that:
\begin{itemize}
\item[(1)] $P=\langle B, X_{-\alpha}\rangle$. 
\item[(2)] $\langle \lambda, \alpha\rangle=0$ and $\langle \lambda, \beta\rangle=2$.
\item[(3)] $L=\langle X_{-\alpha}, T, X_{\alpha}\rangle$.
\item[(4)] $\mathfrak{g}(2)=\operatorname{span}_{\mathbb{C}}\left\{E_{\beta}, E_{\beta+\alpha}, E_{\beta+2\alpha}, E_{\beta+3\alpha}\right\}$ and $\mathfrak{u}_P^{\ge 2}=\operatorname{span}_{\mathbb{C}}\left\{E_{\beta}, E_{\beta+\alpha}, E_{\beta+2\alpha}, E_{\beta+3\alpha}, E_{2\beta+3\alpha}\right\}$.
\item[(5)] $W_L=\left\{e, s\right\}$ and $W^L=\left\{e, t, ts, sr^2, sr^3, r^4 \right\}$. 
\end{itemize}

Now that $\mathfrak{g}(2)$ is of dimension 4, telling whether $L(N, I, v) \neq \emptyset$ is harder. We therefore propose the following lemma.

\begin{lemma} \label{emptylem}
For any $v \in W^L$, $L(N, I, v) \neq \emptyset$ if and only if $ \mathfrak{g}(2) \cap \dot{v} \cdot I$ contains at least one of the three spaces below:
\[ \mathbb{C}E_{\beta+\alpha} \oplus \mathbb{C}E_{\beta+3\alpha} \text{, } \mathbb{C}E_{\beta} \oplus \mathbb{C}E_{\beta+2\alpha} \text{ or } \mathbb{C}E_{\beta+\alpha} \oplus \mathbb{C}E_{\beta+2\alpha}.  \]
\end{lemma}

\begin{proof}
Let $B_L$ denote $L \cap B$ in this lemma. For any $v \in W^L$, 
\[ L(N, I, v) \cong (L, B_L, \mathfrak{g}(2), \mathfrak{g}(2) \cap \dot{v} \cdot I, N) \subset L/B_L .\]
Recall the definition of the quintuple
\begin{equation} \label{twopoint}
(L, B_L, \mathfrak{g}(2), \mathfrak{g}(2) \cap \dot{v} \cdot I, N)=\left\{lB_L \in L/B_L \, |\, l^{-1} \cdot N \in \mathfrak{g}(2) \cap \dot{v} \cdot I \right\}.
\end{equation}
By Lemma~\ref{flowlem}, there exist nonzero constants $c$ and $c'$ such that 
\begin{equation} \label{alphaflow}
x_{\alpha}(z) \cdot N=x_{\alpha}(z) \cdot (E_{\beta+\alpha}+E_{\beta+3\alpha})=E_{\beta+\alpha}+czE_{\beta+2\alpha}+(1+c'z^2)E_{\beta+3\alpha} 
\end{equation}
for any $z \in \mathbb{C}$. Let $z_1$ and $z_2$ be the two distinct solutions of $1+c'z^2=0$. We partition $L/B_L$ into four subsets:
\[ L/B_L=\left\{\dot{e}B_L\right\} \amalg \left\{\dot{s}B_L\right\} \amalg \left\{x_{\alpha}(z_1)\dot{s}B_L, \, x_{\alpha}(z_2)\dot{s}B_L\right\} \amalg \left\{ x_{\alpha}(z)\dot{s}B_L \,|\, z \neq 0, z_1, z_2\right\}.  \]
By this partition and Equation~\ref{twopoint}, we can easily deduce the lemma. Further details are left out to mitigate distraction from the main course of computation.
\end{proof}

Next we give an algorithm (referred to as the ``$v$-algorithm'') to find all the $v$'s so that $L(N, I, v) \neq \emptyset$. 

Firstly, $W_Lr^0, W_Lr^1, \ldots, W_Lr^5$ are all the six $W_L$-cosets of $W$, only that $r^i$ may not be in $W^L$. We view the set of roots of $I$ as a configuration of arrows. For each $i=0, 1, \ldots, 5$, check if the aforementioned configuration, after a rotation by $r^i$, covers one of the three sets of arrows (roots):
\[ \left\{\beta+\alpha, \beta+3\alpha\right\} \text{, } \left\{\beta, \beta+2\alpha\right\} \text{ and } \left\{\beta+\alpha, \beta+2\alpha\right\}.  \]
If so, both $\mathfrak{g}(2) \cap \dot{r}^i \cdot I$ and $\mathfrak{g}(2) \cap \dot{s}\dot{r}^i \cdot I$ satisfy the condition of Lemma~\ref{emptylem}. One of $r^i$ and $sr^i$ has to be in $W^L$ and it serves as a $v$ so that $L(N, I, v) \neq \emptyset$.

\noindent $\bullet$ $I=I_{2\beta+3\alpha }$, $I_{\beta+3\alpha}$ or $I_{\beta+2\alpha}$

Using the $v$-algorithm, we deduce that $\pi_{I}^{-1}(N)=\emptyset$ for all three ideals. The same fact can be more easily established by comparing the dimension of $G \times^B I$ with the dimension of the nilpotent orbit $G_2(a_1)$, but the algorithm is necessary when we work with bigger ideals.

\noindent $\bullet$ $I=I_{\beta+\alpha }$

Using the $v$-algorithm, we see that $v=e$ is the only element so that $L(N, I, v) \neq \emptyset$. In this situation, 
\begin{equation*} 
\begin{aligned}
\mathfrak{g}(2) \cap \dot{e} \cdot I=&\operatorname{span}_{\mathbb{C}}\left\{E_{\beta+\alpha}, E_{\beta+2\alpha}, E_{\beta+3\alpha}\right\},  \\
 \mathfrak{u}_P^{\ge 2} \cap \dot{e} \cdot I=&\operatorname{span}_{\mathbb{C}}\left\{E_{\beta+\alpha}, E_{\beta+2\alpha}, E_{\beta+3\alpha}, E_{2\beta+3\alpha}\right\}.
\end{aligned}
\end{equation*}
Then, by Lemma~\ref{prehsp} and Lemma~\ref{dimlem}, $\operatorname{dim}P(N, I, e)=\operatorname{dim}L(N, I, e)=0$. By Corollary~\ref{dimcor}, $P(N, I, e)$ and $L(N, I, e)$ are the same finite set. Next we compute the set $L(N, I, e)$. 

To do this, we use the same setup as in the proof of Lemma~\ref{emptylem}. In particular, note that $L(N, I, e) \cong (L, B_L, \mathfrak{g}(2), \mathfrak{g}(2) \cap \dot{e} \cdot I, N) \subset L/B_L$ and that $L/B_L$ can be partitioned into the following four subsets.
\[ L/B_L=\left\{\dot{e}B_L\right\} \amalg \left\{\dot{s}B_L\right\} \amalg \left\{x_{\alpha}(z_1)\dot{s}B_L, \, x_{\alpha}(z_2)\dot{s}B_L\right\} \amalg \left\{ x_{\alpha}(z)\dot{s}B_L \,|\, z \neq 0, z_1, z_2\right\}.  \]
We then find the intersection of each subset with $L(N, I, e)$. 

Clearly, $\dot{e}B_L \in L(N, I, e)$.

For the other three subsets, note that they all consist of points in the form $x_{\alpha}(z)\dot{s}B_L$ for some $z \in \mathbb{C}$. By definition, $x_{\alpha}(z)\dot{s}B_L \in L(N, I, e) \iff \dot{s}^{-1} \cdot x_{\alpha}(z)^{-1} \cdot N \in \mathfrak{g}(2) \cap \dot{e} \cdot I \iff$ \linebreak 
$x_{\alpha}(-z) \cdot N \in \dot{s} \cdot ( \mathfrak{g}(2) \cap \dot{e} \cdot I)=\operatorname{span}_{\mathbb{C}}\left\{E_{\beta}, E_{\beta+\alpha}, E_{\beta+2\alpha}\right\}$. By Equation~\ref{alphaflow}, 
\[ x_{\alpha}(-z) \cdot N=E_{\beta+\alpha}-czE_{\beta+2\alpha}+(1+c'z^2)E_{\beta+3\alpha}.  \]
Then $x_{\alpha}(-z) \cdot N \in \operatorname{span}_{\mathbb{C}}\left\{E_{\beta}, E_{\beta+\alpha}, E_{\beta+2\alpha}\right\}$ if and only if $z=z_1$ or $z_2$, the two solutions of $1+c'z^2=0$.

In summary, $L(N, I, e) \cong (L, B_L, \mathfrak{g}(2), \mathfrak{g}(2) \cap \dot{e} \cdot I, N)=\left\{\dot{e}B_L, x_{\alpha}(z_1)\dot{s}B_L, \, x_{\alpha}(z_2)\dot{s}B_L\right\}$, hence 
\[ \pi_{I}^{-1}(N)=P(N, I, e)=L(N, I, e)=\left\{\dot{e}B, x_{\alpha}(z_1)\dot{s}B, x_{\alpha}(z_2)\dot{s}B\right\}. \]
It is the variety of 3 distinct points.

\noindent $\bullet$ $I=I_{\beta }$

$v=e$ is still the only possibility, but now $\mathfrak{g}(2) \cap \dot{e} \cdot I=\mathfrak{g}(2)$ and $\mathfrak{u}_P^{\ge 2} \cap \dot{e} \cdot I=\mathfrak{u}_P^{\ge 2}$. Therefore, $\pi_{I}^{-1}(N)=P(N, I, e)=P\dot{e}B=\overline{X_{s}}$. It is a Schubert variety. Note that the two 0-cells $\left\{x_{\alpha}(z_1)\dot{s}B\right\}$ and $\left\{x_{\alpha}(z_2)\dot{s}B\right\}$ of $\pi_{I_{\beta+\alpha}}^{-1}(N)$ are no longer cells of $\pi_{I}^{-1}(N)$.

\noindent $\bullet$ $I=I_{\alpha }$

Using the $v$-algorithm, we now have $v=e$ or $t$. 

When $v=e$, $P(N, I, e)=P(N, I_{\beta+\alpha}, e)=\left\{\dot{e}B, x_{\alpha}(z_1)\dot{s}B, x_{\alpha}(z_2)\dot{s}B\right\}$.

When $v=t$, $\mathfrak{u}_P^{\ge 2} \cap \dot{t} \cdot I=\operatorname{span}_{\mathbb{C}}\left\{E_{\beta+\alpha}, E_{\beta+2\alpha}, E_{\beta+3\alpha}, E_{2\beta+3\alpha}\right\}$. By Lemma~\ref{prehsp} and Lemma~\ref{dimlem}, $\operatorname{dim}P(N, I, t)=1$, so $P(N, I, t)$ is strictly smaller than $\mathcal{O}_t$. Recall that $P(N, I, t)=\pi_{I}^{-1}(N) \cap \mathcal{O}_{t}$ and $\mathcal{O}_t=P\dot{t}B=U^t\dot{t}B \amalg U^{st}\dot{s}\dot{t}B$. By Lemma~\ref{prodlem}, $U^t=X_{\beta}$ and $U^{st}=X_{\alpha}X_{\beta+3\alpha}$. Next we inspect $\pi_{I}^{-1}(N) \cap U^t\dot{t}B$ and $\pi_{I}^{-1}(N) \cap U^{st}\dot{s}\dot{t}B$.

For $u\dot{t}B \in U^t\dot{t}B$ to be in $\pi_{I}^{-1}(N)$, we need $u^{-1} \cdot N \in \mathfrak{u}_P^{\ge 2} \cap \dot{t} \cdot I$. Here $u \in U^t$ takes the form $u=x_{\beta}(z)$ for some $z \in \mathbb{C}$. By Lemma~\ref{flowlem}, there exists a nonzero constant $c$ such that 
\[ u^{-1} \cdot N=x_{\beta}(-z) \cdot (E_{\beta+\alpha}+E_{\beta+3\alpha})=E_{\beta+\alpha}+E_{\beta+3\alpha}-czE_{2\beta+3\alpha}  \]
for any $z \in \mathbb{C}$. Clearly, $E_{\beta+\alpha}+E_{\beta+3\alpha}-czE_{2\beta+3\alpha} \in \mathfrak{u}_P^{\ge 2} \cap \dot{t} \cdot I$ for any $z \in \mathbb{C}$. Therefore, $\pi_{I}^{-1}(N) \cap U^t\dot{t}B=X_{\beta}\dot{t}B=X_t$.

For $u\dot{s}\dot{t}B \in U^{st}\dot{s}\dot{t}B$ to be in $\pi_{I}^{-1}(N)$, we need $\dot{s}^{-1} \cdot u^{-1} \cdot N \in \mathfrak{u}_P^{\ge 2} \cap \dot{t} \cdot I$. Here $u \in U^{st}$ takes the form $u=x_{\alpha}(z)x_{\beta+3\alpha}(z')$ for some $z, z' \in \mathbb{C}$. The same nonzero constants $c$ and $c'$ from Equation~\ref{alphaflow} guarantee that 
\begin{equation*} 
\begin{aligned}
\dot{s}^{-1} \cdot u^{-1} \cdot N=&\dot{s} \cdot x_{\beta+3\alpha}(-z') \cdot x_{\alpha}(-z) \cdot (E_{\beta+\alpha}+E_{\beta+3\alpha}) \\
=&\dot{s} \cdot E_{\beta+\alpha}-cz(\dot{s} \cdot E_{\beta+2\alpha})+(1+c'z^2)(\dot{s} \cdot E_{\beta+3\alpha}) 
\end{aligned}
\end{equation*}
for any $z, z' \in \mathbb{C}$. Then $\dot{s} \cdot E_{\beta+\alpha}-cz(\dot{s} \cdot E_{\beta+2\alpha})+(1+c'z^2)(\dot{s} \cdot E_{\beta+3\alpha}) \in \mathfrak{u}_P^{\ge 2} \cap \dot{t} \cdot I$ if and only if $1+c'z^2=0$. That is, $z=z_1$ or $z_2$, the two solutions of $1+c'z^2=0$ and $z'$ could be any complex number. Therefore, $\pi_{I}^{-1}(N) \cap U^{st}\dot{s}\dot{t}B=x_{\alpha}(z_1)X_{\beta+3\alpha}\dot{s}\dot{t}B \amalg x_{\alpha}(z_2)X_{\beta+3\alpha}\dot{s}\dot{t}B \subset X_{st}$.

In summary, 
\[ \pi_{I}^{-1}(N)=\left\{\dot{e}B, x_{\alpha}(z_1)\dot{s}B, x_{\alpha}(z_2)\dot{s}B\right\} \amalg X_{\beta}\dot{t}B \amalg x_{\alpha}(z_1)X_{\beta+3\alpha}\dot{s}\dot{t}B \amalg x_{\alpha}(z_2)X_{\beta+3\alpha}\dot{s}\dot{t}B. \]

There is a simple description of $\pi_{I}^{-1}(N)$ as an algebraic variety. As will be shown in Lemma~\ref{closurelem2} and Lemma~\ref{closurelem4},
\begin{equation*} 
\begin{aligned}
&\lim_{z \rightarrow \infty} x_{\beta}(z)\dot{t}B =\dot{e}B, \\
&\lim_{z \rightarrow \infty} x_{\alpha}(z_1)x_{\beta+3\alpha}(z)\dot{s}\dot{t}B =x_{\alpha}(z_1)\dot{s}B, \\
&\lim_{z \rightarrow \infty} x_{\alpha}(z_2)x_{\beta+3\alpha}(z)\dot{s}\dot{t}B =x_{\alpha}(z_2)\dot{s}B,
\end{aligned}
\end{equation*}
where limits are taken in the flag variety $G/B$. Therefore, 
\[ \pi_{I}^{-1}(N)=\overline{X_{\beta}\dot{t}B} \amalg \overline{x_{\alpha}(z_1)X_{\beta+3\alpha}\dot{s}\dot{t}B} \amalg \overline{x_{\alpha}(z_2)X_{\beta+3\alpha}\dot{s}\dot{t}B} \cong \mathbb{P}^1 \amalg \mathbb{P}^1 \amalg \mathbb{P}^1. \]

\noindent $\bullet$ $I=I_{\alpha, \beta }$

Still, $v=e$ or $t$.

When $v=e$, since $ \mathfrak{u}_P^{\ge 2} \cap \dot{e} \cdot I=\mathfrak{u}_P^{\ge 2}$, $P(N, I, e)=P\dot{e}B=\overline{X_{s}}$.

When $v=t$, since $\mathfrak{u}_P^{\ge 2} \cap \dot{t} \cdot I= \mathfrak{u}_P^{\ge 2} \cap \dot{t} \cdot I_{\alpha}$, 
\[ P(N, I, t)=P(N, I_{\alpha}, t)=X_{\beta}\dot{t}B \amalg x_{\alpha}(z_1)X_{\beta+3\alpha}\dot{s}\dot{t}B \amalg x_{\alpha}(z_2)X_{\beta+3\alpha}\dot{s}\dot{t}B. \]

Putting these two parts together, we see that $\pi_{I}^{-1}(N)$ has 4 irreducible components, each isomorphic to $\mathbb{P}^1$. $\pi_{I}^{-1}(N)$ is obtained by attaching 3 copies of $\mathbb{P}^1$ to the 3 distinct points $\left\{\dot{e}B, x_{\alpha}(z_1)\dot{s}B, x_{\alpha}(z_2)\dot{s}B\right\}$ of $\overline{X_{s}}$ at their respective points at infinity. Explicitly,
\[ \pi_{I}^{-1}(N)=\overline{X_{s}} \amalg X_{\beta}\dot{t}B \amalg x_{\alpha}(z_1)X_{\beta+3\alpha}\dot{s}\dot{t}B \amalg x_{\alpha}(z_2)X_{\beta+3\alpha}\dot{s}\dot{t}B.  \]

We have finished the case of $G_2(a_1)$.

\begin{remark} \label{subregularSF}
In \cite{collingwood1993nilpotent}, $G_2(a_1)$ is called the subregular nilpotent orbit of $G_2$. In general, the subregular nilpotent orbit is the unique nilpotent orbit of codimension 2 in the nilpotent cone $\mathcal{N} \subset \mathfrak{g}$. Let $N$ be an element of the subregular orbit. It is known that the Springer fiber $\mathcal{B}_N$ is a union of $\mathbb{P}^1$'s whose configuration we now describe. We can form the dual graph of $\mathcal{B}_N$ so that its vertices are the irreducible components of $\mathcal{B}_N$ and two vertices are joined by an edge if the two corresponding components intersect (they always intersect at a single point). 

When $G$ is of type $G_2$ and $N$ is from the subregular orbit $G_2(a_1)$, the dual graph of $\mathcal{B}_N$ is the Dynkin diagram of $D_4$ (see \cite{steinberg1974conjugacy}[section 3.10]). This description of $\mathcal{B}_N$ matches exactly our result for $\pi_{I_{\alpha, \beta}}^{-1}(G_2(a_1))$. Let $A(N)=C_G(N)/C_G^{\circ}(N)$. Since $C_G(N)$ acts naturally on $\pi_{I}^{-1}(N) \subset G/B$ by left multiplication, $A(N)$ permutes the irreducible components of $\pi_{I}^{-1}(N)$. If $G$ is in addition the adjoint form of $G_2$, $A(N) \cong S_3$ (see \cite{carter1985finite}[p.~427]). Then $A(N)$ acts on $\pi_{I_{\alpha, \beta}}^{-1}(G_2(a_1))$ by naturally permuting three components and fixing the last one to which the other three are attached. Consequently, $A(N)$ acts by the natural permutation action on $\pi_{I_{\alpha}}^{-1}(G_2(a_1)) \cong \mathbb{P}^1 \amalg \mathbb{P}^1 \amalg \mathbb{P}^1$ and $\pi_{I_{\beta + \alpha}}^{-1}(G_2(a_1)) \cong \{3 \text{ distinct points}\}$ and trivially on $\pi_{I_{\beta}}^{-1}(G_2(a_1)) \cong \mathbb{P}^1$. This action will be used in our computation of the dot actions for type $G_2$ in section \ref{dotactiong2}.
\end{remark}

\subsubsection{The case of $G_2$}
The notation ``$G_2$'' represents the regular nilpotent orbit of the group $G_2$. It is well-known that the Springer fiber $\mathcal{B}_N$ is a single point when $N$ is regular nilpotent. By comparing the dimension of $G \times^B I$ with that of the regular nilpotent orbit, it is easy to see that the Springer fiber $\mathcal{B}_N$ is the only nonempty Hessenberg ideal fiber, so there is nothing to compute in this case. To show the scope of our algorithm, we give another proof of the two statements just made along the lines of this section.

Firstly, our representative of the regular nilpotent orbit is $N=E_{\alpha}+E_{\beta}$. Set $H=6H_{\alpha}+10H_{\beta}$. Simple computation shows that $\left\{N, H, 6E_{-\alpha}+10E_{-\beta}\right\}$ is an $\mathfrak{sl}_2$-triple such that $H \in \mathfrak{t}$, $\alpha(H)=2$ and $\beta(H)=2$. Therefore, the associated parabolic of $N$ is the Borel subgroup $B$ and we know the following:
\begin{itemize}
\item[(1)] $P=B$. 
\item[(2)] $\langle \lambda, \alpha\rangle=2$ and $\langle \lambda, \beta\rangle=2$.
\item[(3)] $L=T$.
\item[(4)] $\mathfrak{g}(2)=\mathbb{C}E_{\alpha} \oplus \mathbb{C}E_{\beta}$ and $\mathfrak{u}_P^{\ge 2}=\mathfrak{u}$.
\item[(5)] $W_L=\left\{e\right\}$ and $W^L=W$. 
\end{itemize}
For any $v \in W^L=W$ and any Hessenberg ideal $I$, $L(N, I, v) \cong (T, T, \mathfrak{g}(2),  \mathfrak{g}(2) \cap \dot{v} \cdot I, N)$. Then $L(N, I, v) \neq \emptyset \iff N \in \mathfrak{g}(2) \cap \dot{v} \cdot I$. Because $N=E_{\alpha}+E_{\beta}$ and $\mathfrak{g}(2) \cap \dot{v} \cdot I$ is $T$-stable, $N \in \mathfrak{g}(2) \cap \dot{v} \cdot I \iff \mathfrak{g}(2) \subset \mathfrak{g}(2) \cap \dot{v} \cdot I \iff \mathfrak{g}(2) \subset \dot{v} \cdot I$. The last condition is possible only when $v=e$ and $I=\mathfrak{u}$, so the Springer fiber $\mathcal{B}_N$ is the only nonempty Hessenberg ideal fiber. In that case, $P(N, \mathfrak{u}, e)=(B, B, \mathfrak{u}, \mathfrak{u}, N)$ and it is a single point.

We have now finished the computation of all Hessenberg ideal fibers when $G$ is of type $G_2$. The results can be summarized by Table~\ref{tab:listhifg2}. Note that the ways in which $\pi_{I_{\alpha}}^{-1}(\tilde{A}_1)$ and $\pi_{I_{\alpha, \beta}}^{-1}(\tilde{A}_1)$ are presented allude to the closure relationships of their cells. These relationships are proved in subsection \ref{subsectionaninterestinghif}. 

\begin{table}[h]
\def\arraystretch{1.4}
\caption{Hessenberg ideal fibers for type $G_2$} \label{tab:listhifg2}
\centering
\begin{tabular}{|c|c|c|c|c|c|}
\hline
                     & $\{0\}$ & $A_1$                & $\tilde{A}_1$                                                                   & $G_2(a_1)$                                                           & $G_2$   \\ \hline
$I_{\emptyset}$      & $G/B$   &                      &                                                                                 &                                                                      &         \\ \hline
$I_{2\beta+3\alpha}$ & $G/B$   & $\overline{X_s}$     &                                                                                 &                                                                      &         \\ \hline
$I_{\beta+3\alpha}$  & $G/B$   & $\overline{X_{st}}$  &                                                                                 &                                                                      &         \\ \hline
$I_{\beta+2\alpha}$  & $G/B$   & $\overline{X_{st}}$  & $\overline{X_{t}}$                                                              &                                                                      &         \\ \hline
$I_{\beta+\alpha}$   & $G/B$   & $\overline{X_{st}}$  & $\mathbb{P}^1 \amalg \mathbb{P}^1$                                              & 3 distinct points                                                    &         \\ \hline
$I_{\alpha}$         & $G/B$   & $\overline{X_{st}}$  & $\overline{X_{t}} \amalg \overline{X_{\beta}X_{2\beta+3\alpha}\dot{t}\dot{r}B}$ & $\mathbb{P}^1 \amalg \mathbb{P}^1 \amalg \mathbb{P}^1$               &         \\ \hline
$I_{\beta}$          & $G/B$   & $\overline{X_{sts}}$ & $\overline{X_{ts}}$                                                             & $\overline{X_s}$                                                     &         \\ \hline
$I_{\alpha, \beta}$  & $G/B$   & $\overline{X_{sts}}$ & $\overline{X_{ts}} \cup \overline{X_{\beta}X_{2\beta+3\alpha}\dot{t}\dot{r}B}$  & $\mathbb{P}^1 \cup \mathbb{P}^1 \cup \mathbb{P}^1 \cup \mathbb{P}^1$ & 1 point \\ \hline
\end{tabular}
\end{table}

\subsection{An interesting Hessenberg ideal fiber} \label{subsectionaninterestinghif}
Now we study one of the Hessenberg ideal fibers computed in the last subsection---$\pi_{I}^{-1}(N)$ where $I=I_{\alpha}$ and $N=E_{\beta+2\alpha}$. Since $N$ is a representative of $\tilde{A_1}$, we denote this fiber by $\pi_{I_{\alpha}}^{-1}(\tilde{A_1})$ in the rest of this subsection. 

N. Spaltenstein proved in \cite{spaltenstein1982classes} that any Springer fiber $\mathcal{B}_N$ of a reducitve linear algebraic group over $\mathbb{C}$ is connected and that its irreducible components are of the same dimension. A natural question is whether the same is true for Hessenberg ideal fibers. The answer is no, because $\pi_{I_{\alpha}}^{-1}(\tilde{A_1})$ is a counterexample. In fact, we can prove the following.

\begin{theorem} \label{aninterestinghif}
$\pi_{I_{\alpha}}^{-1}(\tilde{A_1})$ has two connected components, each of which is irreducible as well. They are of dimension 1 and 2 respectively. 
\end{theorem}

To prove the theorem, we need to study the closure relationships between different cells of the Hessenberg ideal fiber $\pi_{I_{\alpha}}^{-1}(\tilde{A_1})$. These relationships can be deduced from the following 4 lemmas. It is well-known that, on algebraic varieties over $\mathbb{C}$, the closure of a locally closed subvariety in the Zariski topology is the same as in the ordinary (analytic) topology. Therefore, we mostly work with the ordinary topology in the rest of this subsection. All closures and limits are taken on the flag variety $G/B$, where $G$ is assumed to be a connected algebraic group of type $G_2$ over $\mathbb{C}$. 

From the last subsection we know that 
\[ \pi_{I_{\alpha}}^{-1}(\tilde{A_1})=\overline{X_{t}} \amalg \overline{X_{\beta}\dot{t}\dot{s}B} \amalg X_{\beta+3\alpha}\dot{r}B \amalg X_{\beta}X_{2\beta+3\alpha}\dot{t}\dot{r}B, \]
\[ \pi_{I_{\alpha, \beta}}^{-1}(\tilde{A_1})=\overline{X_{ts}} \amalg X_{\beta+3\alpha}\dot{r}B \amalg X_{\beta}X_{2\beta+3\alpha}\dot{t}\dot{r}B.  \]

\begin{lemma} \label{closurelem1}
$\overline{X_{\beta+3\alpha}\dot{r}B} \subset \overline{X_{\beta}X_{2\beta+3\alpha}\dot{t}\dot{r}B}$.
\end{lemma}

\begin{proof}
$\pi_{I_{\alpha, \beta}}^{-1}(\tilde{A_1})$ is a Springer fiber. By Spaltenstein's result, its irreducible components are of the same dimension. We know that 
\[ \pi_{I_{\alpha, \beta}}^{-1}(\tilde{A_1})=\overline{\pi_{I_{\alpha, \beta}}^{-1}(\tilde{A_1}) }=\overline{X_{ts}} \cup \overline{X_{\beta+3\alpha}\dot{r}B} \cup \overline{X_{\beta}X_{2\beta+3\alpha}\dot{t}\dot{r}B}.  \]
Since $\operatorname{dim}\overline{X_{\beta+3\alpha}\dot{r}B}=1$ and $\operatorname{dim}\overline{X_{ts}}=\operatorname{dim}\overline{X_{\beta}X_{2\beta+3\alpha}\dot{t}\dot{r}B}=2$, the irreducible components of $\pi_{I_{\alpha, \beta}}^{-1}(\tilde{A_1})$ have to be $\overline{X_{ts}}$ and $\overline{X_{\beta}X_{2\beta+3\alpha}\dot{t}\dot{r}B}$. Because $\overline{X_{\beta+3\alpha}\dot{r}B} \cong \mathbb{P}^1$ hence irreducible, it has to lie in one of the irreducible components. Note that $X_{\beta+3\alpha}\dot{r}B$ lies in the Schubert cell $X_{r}=X_{st}$, so has no intersection with $\overline{X_{ts}}$. Therefore, $\overline{X_{\beta+3\alpha}\dot{r}B} \subset \overline{X_{\beta}X_{2\beta+3\alpha}\dot{t}\dot{r}B}$.
\end{proof}

$N=E_{\beta+2\alpha}$ is the representative of $\tilde{A_1}$ used in the computation of $\pi_{I_{\alpha}}^{-1}(\tilde{A_1})$. Recall that the Levi subgroup of the associated parabolic of $N$ is $L=\langle X_{-\beta}, T, X_{\beta}\rangle$, and that $W_L=\left\{e, t\right\}$ and $W=W_LW^L$. 

\begin{lemma} \label{closurelem2}
For any $v \in W^L$, 
\[ \lim_{z \rightarrow \infty}x_{\beta}(z)\dot{t}\dot{v}B=\dot{v}B.  \]
\end{lemma}

\begin{proof}
Let $\lambda$ be an associated one-parameter subgroup of $N$. Recall that
\[ \mathcal{O}_{v}^{\lambda}=U^e\dot{v}B \amalg U^t \dot{t}\dot{v}B=\left\{\dot{v}B\right\} \amalg X_{\beta}\dot{t}\dot{v}B \cong \mathbb{P}^1.  \]
$X_{\beta}\dot{t}\dot{v}B$ is the 1-cell of $\mathbb{P}^1$ and $\dot{v}B$ is the point at infinity. Then the limit is clearly true.
\end{proof}

\begin{lemma} \label{closurelem3}
For any $z \in \mathbb{C}$,
\[ \lim_{z' \rightarrow \infty} x_{\beta}(z)x_{2\beta+3\alpha}(z')\dot{t}\dot{r}B=x_{\beta}(z)\dot{t}\dot{s}B.  \]
As a result, $\overline{X_{\beta}\dot{t}\dot{s}B} \subset \overline{X_{\beta}X_{2\beta+3\alpha}\dot{t}\dot{r}B}$.
\end{lemma}

\begin{proof}
Since $tr=t(st)=(ts)t$, we have
\[ x_{\beta}(z)x_{2\beta+3\alpha}(z')\dot{t}\dot{r}B=x_{\beta}(z)\dot{t}\dot{s}(\dot{s}\dot{t}x_{2\beta+3\alpha}(z')\dot{t}\dot{s})\dot{t}B. \]
By Lemma~\ref{permlem}, there exists a nonzero constant $c$ such that $\dot{s}\dot{t}x_{2\beta+3\alpha}(z')\dot{t}\dot{s}=x_{\beta}(cz')$ for any $z' \in \mathbb{C}$. Since $e \in W^L$, by Lemma~\ref{closurelem2},\[ \lim_{z' \rightarrow \infty}x_{\beta}(cz')\dot{t}B=\dot{e}B. \]
Therefore, for any $z \in \mathbb{C}$,
\[ \lim_{z' \rightarrow \infty} x_{\beta}(z)x_{2\beta+3\alpha}(z')\dot{t}\dot{r}B= x_{\beta}(z)\dot{t}\dot{s}(\lim_{z' \rightarrow \infty}x_{\beta}(cz')\dot{t}B)=x_{\beta}(z)\dot{t}\dot{s}B. \]

$\overline{X_{\beta}\dot{t}\dot{s}B} \subset \overline{X_{\beta}X_{2\beta+3\alpha}\dot{t}\dot{r}B}$ follows easily from the limit above.
\end{proof}

\begin{lemma} \label{closurelem4}
\[ \lim_{z \rightarrow \infty} x_{\beta}(z)\dot{t}\dot{s}B=\lim_{z \rightarrow \infty}x_{\beta+3\alpha}(z)\dot{r}B=\dot{s}B.  \]
\end{lemma}

\begin{proof}
Since $s \in W^L$, by Lemma~\ref{closurelem2},
\[ \lim_{z \rightarrow \infty}x_{\beta}(z)\dot{t}\dot{s}B=\dot{s}B.  \]

By Lemma~\ref{permlem}, there exists a nonzero constant $c$ such that for any $z \in \mathbb{C}$,
\[ x_{\beta+3\alpha}(z)\dot{r}B=\dot{s}(\dot{s}x_{\beta+3\alpha}(z)\dot{s})\dot{t}B=\dot{s}(x_{\beta}(cz)\dot{t}B).  \]
Therefore, by Lemma~\ref{closurelem2},
\[ \lim_{z \rightarrow \infty} x_{\beta+3\alpha}(z)\dot{r}B=\dot{s}(\lim_{z \rightarrow \infty}x_{\beta}(cz)\dot{t}B)=\dot{s}B.\]
\end{proof}

Now we can prove Theorem~\ref{aninterestinghif}.

\begin{proof}
The explicit description
\[ \pi_{I_{\alpha}}^{-1}(\tilde{A_1})=\overline{X_{t}} \amalg \overline{X_{\beta}\dot{t}\dot{s}B} \amalg X_{\beta+3\alpha}\dot{r}B \amalg X_{\beta}X_{2\beta+3\alpha}\dot{t}\dot{r}B  \]
shows that $\pi_{I_{\alpha}}^{-1}(\tilde{A_1})$ has two 0-cells, three 1-cells and one 2-cell. Since $\pi_{I_{\alpha}}^{-1}(\tilde{A_1})$ is projective, it is compact in the ordinary topology. Therefore the cell structure gives us the singular homology of $\pi_{I_{\alpha}}^{-1}(\tilde{A_1})$. In particular, $H_0(\pi_{I_{\alpha}}^{-1}(\tilde{A_1}); \mathbb{Z})=\mathbb{Z} \oplus \mathbb{Z}$, which means $\pi_{I_{\alpha}}^{-1}(\tilde{A_1})$ has two connected components. 

Lemma~\ref{closurelem1} and Lemma~\ref{closurelem3} imply that both $\overline{X_{\beta}\dot{t}\dot{s}B}$ and $\overline{X_{\beta+3\alpha}\dot{r}B}$ lie in $\overline{X_{\beta}X_{2\beta+3\alpha}\dot{t}\dot{r}B}$. Therefore, $\pi_{I_{\alpha}}^{-1}(\tilde{A_1})=\overline{X_{t}} \cup \overline{X_{\beta}X_{2\beta+3\alpha}\dot{t}\dot{r}B}$. Note that $\overline{X_{\beta}X_{2\beta+3\alpha}\dot{t}\dot{r}B}$ is irreducible and connected in the Zariski topology, because it is the Zariski closure of a 2-cell. Therefore, it is also connected in the ordinary topology (\cite{hartshorne1997algebraic}[Appendix B]). Then we must have $\overline{X_{t}} \cap \overline{X_{\beta}X_{2\beta+3\alpha}\dot{t}\dot{r}B}=\emptyset$. Otherwise $\pi_{I_{\alpha}}^{-1}(\tilde{A_1})$ is connected in the ordinary topology, contradictory to $H_0(\pi_{I_{\alpha}}^{-1}(\tilde{A_1}); \mathbb{Z})=\mathbb{Z} \oplus \mathbb{Z}$. In summary, $\overline{X_{t}}$ and $\overline{X_{\beta}X_{2\beta+3\alpha}\dot{t}\dot{r}B}$ are both connected and irreducible and they are of dimension 1 and 2 respectively. We have now proved the theorem.
\end{proof}

\section{Dot Actions for Type \texorpdfstring{$G_2$}{}} \label{dotactiong2}
As an application of Hessenberg ideal fibers, we use the results from section \ref{g2hif} to classify Tymoczko's dot actions for type $G_2$. This section is joint work with the author's advisor Dr. Patrick Brosnan.

\subsection{Background and motivation}
Let $G$ be a connected reductive algebraic group over $\mathbb{C}$ and $B$ a Borel subgroup of $G$. Let $\mathfrak{g}$ and $\mathfrak{b}$ denote their respective Lie algebras. Recall that a Hessenberg subspace is an $\operatorname{ad}(\mathfrak{b})$-stable subspace $M$ of $\mathfrak{g}$ containing $\mathfrak{b}$. Let $y \in \mathfrak{g}$ be a regular semisimple element. The Hessenberg variety $\operatorname{Hess}(M, y)$ is defined to be the fiber over $y$ of the following map
\begin{equation} \label{piM}
\begin{array}{llll}
\pi_M: & G \times^B M & \longrightarrow & \mathfrak{g} \\
 & (g, x) &  \longmapsto & g \cdot x
\end{array}
\end{equation}
where $g \cdot x$ denotes the adjoint action $\operatorname{Ad}(g)(x)$. According to \cite{de1992hessenberg}, $\operatorname{Hess}(M,y)$ is nonsingular for any regular semisimple $y \in \mathfrak{g}$, and, if $T=C_G(y)$ denotes the maximal torus in $G$ centralizing $y$, we have $\operatorname{Hess}(M,y)^T=(G/B)^T$. By \cite{bialynicki1973some}, it follows that $\operatorname{Hess}(M,y)$ is cellular with cells in one-one correspondence with elements of the Weyl group $W$ of $G$. Tymoczko \cite{tymoczko2007permutation} used these facts to define the dot actions of $W$ on both the $T$-equivariant and the ordinary cohomology groups with coefficient $\mathbb{C}$ of the (regular semisimple) Hessenberg variety $\operatorname{Hess}(M, y)$. The point is that $W$ does not act on the underlying Hessenberg variety in any interesting way.

The Shareshian-Wachs conjecture expresses Tymoczko's dot action in type A ($G=\mathbf{GL}_n$) in terms of a symmetric function attached to colorings of a certain graph $\Gamma_f$. When $G=\mathbf{GL}_n$, the Hessenberg subspace $M$ is uniquely determined by a non-decreasing function $f: \left\{1, \ldots, n\right\} \longrightarrow \left\{1, \ldots, n\right\}$ called a Hessenberg function (we require $f(i ) \ge i$ for all $i$). Then $\Gamma_f$ is the graph with vertex set $V=\left\{1, \ldots, n\right\}$ and edge set $E=\left\{ \left\{i, j\right\} \subset V \ | \  i < j \le f(i) \right\}$. Let $\Lambda$ denote the ring of symmetric functions, a subring of the power series ring in infinitely many variables $\mathbb{C}\left[ [x_1, x_2, \ldots]  \right]$. The ring $\Lambda=\bigoplus_{n \ge 0 } \Lambda_n$ is graded in an obvious way and the Frobenius character $\operatorname{ch}: \operatorname{Rep}S_n \otimes_{\mathbb{Z}}\mathbb{C} \longrightarrow \Lambda_n$ is an isomorphism from the representation ring (viewed as an abelian group) $\operatorname{Rep}S_n$ to $\Lambda_n$. Modifying a construction of Stanley, Shareshian and Wachs \cite{shareshian2016chromatic} defined a polynomial $X_{\Gamma_f}(t) \in \Lambda[t]$ given by
\begin{equation} \label{chrqf}
X_{\Gamma_f}(t)=\sum_{\kappa: V \rightarrow \mathbb{Z}_{+}} t^{\operatorname{asc}(\kappa)}x_{\kappa(1)}x_{\kappa(2)}\cdots x_{\kappa(n)}  
\end{equation}
where $\kappa$ runs over all colorings of $\Gamma_f$ and $\operatorname{asc}(\kappa)=|\left\{(i, j ) \in E \ | \  i <j, \kappa(i) < \kappa(j)\right\}|$. The Shareshian-Wachs conjecture is then 
\begin{equation} \label{ShWaconj}
\omega X_{\Gamma_f}(t)=\sum_{i \ge 0} \operatorname{ch}(S_n, H^{2i}(\operatorname{Hess}(M, y)))t^i
\end{equation}
where $\omega$ is the involution on $\Lambda$ corresponding to tensoring with the sign representation and $y \in \mathfrak{g}$ is regular semisimple.

This conjecture has already been proved by Brosnan and Chow in \cite{brosnan2018unit}. Therefore, Equation~\ref{ShWaconj} gives us a formula for Tymoczko's dot action. Moreover, it implies that the Stanley-Stembridge conjecture, which states that $X_{\Gamma_f}(1)$ is a positive linear combination of elementary symmetric functions, has the following representation theoretic formulation: when $G=\mathbf{GL}_n$, the dot action of $S_n$ on $H^*(\operatorname{Hess}(M, y))$ is a direct sum of representations of the form $\operatorname{Ind}_{S_{\lambda}}^{S_n} \mathbf{1}$, where $y \in \mathfrak{g}$ is regular semisimple and $S_{\lambda}$ denotes the Young subgroup of $S_n$ for the partition $\lambda$ of $n$. The Stanley-Stembridge conjecture has not been proved yet, which means the dot action on $H^*(\operatorname{Hess}(M, y))$ still needs to be further investigated. Naturally, if a representation theoretic proof of the Stanley-Stembridge conjecture is desired, we should find for it a representation theoretic formulation in all types via the dot action. As one step in the attempt to generalize the Shareshian-Wachs and the Stanley-Stembridge conjectures, Brosnan and the author classified all the dot actions for type $G_2$. In particular, Theorem~\ref{thmBX} Table~\ref{tab:listdotg2} shows that not every $H^*(\operatorname{Hess}(M, y))$ is a direct sum of induced representations of the form $\operatorname{Ind}_{W_J}^W \mathbf{1}$, where $W_J$ is the Weyl group of a Levi subalgebra of $\mathfrak{g}$. Therefore, at least the ``naive'' generalization of the Stanley-Stembridge conjecture for type $A$ is not true for type $G_2$.

\subsection{Computational techniques} \label{comptech}
In this subsection, we present the main ideas of the computation and gather all the relevant techniques.

Let $G$ denote a connected reductive algebraic group over $\mathbb{C}$ of any type. For a Hessenberg subspace $M$, consider the map $\pi_M: G \times^B M \longrightarrow \mathfrak{g}$ again. Let $\mathfrak{g}^{rs}$ denote the Zariski open dense subset of $\mathfrak{g}$ consisting of regular semisimple elements and $\underline{\mathbb{C}}$ denote the constant sheaf over $G \times^B M$. Results from \cite{de1992hessenberg} show that the direct push-forward complex $R\pi_{M*}\underline{\mathbb{C}}|_{\mathfrak{g}^{rs}}$, considered as an object of the constructible bounded derived category $D_c^b(\mathfrak{g}^{rs})$, is equivalent to a local system over $\mathfrak{g}^{rs}$. The local system corresponds to a monodromy action of $\pi_1(\mathfrak{g}^{rs},y)$ on $H^{*}(\operatorname{Hess}(M, y))$ after picking a base-point $y \in \mathfrak{g}^{rs}$. The following theorem was first stated and proved for type $A$ by Brosnan and Chow, and generalized to all other types by B\u{a}libanu and Crooks.

\begin{theorem}[{\cite{brosnan2018unit}[Theorem 73]}, {\cite{bualibanu2022perverse}[Corollary 1.14]}] \label{BrChmonodromy}
The monodromy action of $\pi_1(\mathfrak{g}^{rs}, y)$ on $H^*(\operatorname{Hess}(M, y))$ factors through the Weyl group $W$ and coincides with Tymoczko's dot action.
\end{theorem}

From now on, we will not distinguish between the monodromy action above and Tymoczko's original definition via moment graph, and will refer to both of them simply as the dot action of $W$ on $H^*(\operatorname{Hess}(M, y))$. In particular, the theorem above implies that the dot action on $H^*(\operatorname{Hess}(M,y))$ does not depend on the choice of the regular semisimple element $y$.

Let $d$ denote the complex dimension of $G \times^B M$. Since $G \times^B M$ is a vector bundle over $G/B$, it is nonsingular and we can therefore apply the decomposition theorem of Beilinson, Bernstein and Deligne to $R\pi_{M*}\underline{\mathbb{C}}[d]$. Brosnan has the following conjecture.

\begin{conjecture}[Brosnan] \label{Brconj} 
$R\pi_{M*}\underline{\mathbb{C}}[d]$ is a direct sum of shifted intermediate extensions of irreducible local systems on $\mathfrak{g}^{rs}$. That is, we have the following decomposition,
\begin{equation} \label{Brdecomp}
R\pi_{M*}\underline{\mathbb{C}}[d]=\bigoplus \operatorname{IC}(\overline{\mathfrak{g}^{rs}}, \mathcal{V}_i)[a_i]
\end{equation}
where each $\mathcal{V}_i$ is an irreducible local system on $\mathfrak{g}^{rs}$ and $\operatorname{IC}(\overline{\mathfrak{g}^{rs}}, \mathcal{V}_i)$ is the intermediate extension of $\mathcal{V}_i[\operatorname{dim \mathfrak{g}}]$ from $\mathfrak{g}^{rs}$ to $\mathfrak{g}=\overline{\mathfrak{g}^{rs}}$. Moreover, for any $y \in \mathfrak{g}^{rs}$, the monodromy action of $\pi_1(\mathfrak{g}^{rs}, y)$ on the stalk $\mathcal{V}_{i, y}$ factors through the Weyl group $W$.
\end{conjecture}

We are going to prove the conjecture independently for type $G_2$ later in this subsection, which will be sufficient for our computation of dot actions. 

\begin{remark}
This conjecture has been proved for type $A$ by B\u{a}libanu and Crooks \cite{bualibanu2022perverse}, and for all types by Precup and Sommers \cite{precup2022perverse}. Vilonen and Xue proved a similar but weaker statement in \cite{vilonen2021note}.
\end{remark}

When the conjecture above is true, taking cohomology of both sides of Equation~\ref{Brdecomp} at $y \in \mathfrak{g}^{rs}$ and ignoring shifting, we get the decomposition of the dot action of $W$ on $H^*(\operatorname{Hess}(M, y))$ as a direct sum of irreducible $W$ representations: $H^*(\operatorname{Hess}(M,y)) \cong \bigoplus \mathcal{V}_{i, y}$. Therefore, Tymoczko's dot action is determined by the decomposition of the complex $R\pi_{M*}\underline{\mathbb{C}}[d]$.

On the other hand, Hessenberg subspace is the natural ``dual'' notion of Hessenberg ideal. Recall that a Hessenberg ideal is an $\operatorname{ad}(\mathfrak{b})$-stable subspace $I$ of $\mathfrak{u}$, where $\mathfrak{u}$ is the Lie algebra of the unipotent radical $U$ of $B$. There is a one-one correspondence between Hessenberg subspaces and Hessenberg ideals denoted by $M \mapsto I=M^{\perp}$. To understand the correspondence, choose a maximal torus $T$ in $B$. Let $\Phi$ denote the set of roots and $\mathfrak{g}_{\alpha}$ denote the root space corresponding to $\alpha \in \Phi$ as before. It is easy to see that $M$ and $I$ have root space decompositions for certain subsets $\Phi(M)$ and $\Phi(I)$ of $\Phi$:
\[ M \cong \mathfrak{t} \oplus \bigoplus_{\alpha \in \Phi(M)} \mathfrak{g}_{\alpha} \, \text{ and } \, I \cong \bigoplus_{\alpha \in \Phi(I)} \mathfrak{g}_{\alpha},  \]
where $\mathfrak{t}$ is the Lie algebra of the maximal torus $T$. Then $I=M^{\perp}$ if and only if $\Phi(I)=-(\Phi \setminus \Phi(M))$. In the rest of this section, unless otherwise specified, $M$ and $I$ are always assumed to satisfy the relation $I=M^{\perp}$ as explained above.

Recall the following map defined in the very beginning of section \ref{introduction}.
\begin{equation*} 
\begin{array}{llll}
\pi_I: & G \times^B I & \longrightarrow & \mathfrak{g} \\
 & (g, x) &  \longmapsto & g \cdot x
\end{array}
\end{equation*}
Let $N \in \mathcal{N}$ be an element in the image of $\pi_I$. The fiber of $\pi_I$ over $N$, $\pi_{I}^{-1}(N)$, is a Hessenberg ideal fiber. 
Let $G \times \mathbb{G}_m$ act on $\mathfrak{g}$ with $G$ acting via the adjoint action and $\mathbb{G}_m$ acting by scaling. By fixing a Killing form, we get an autoequivalence $F$, the Fourier-Sato transform, from the category $\operatorname{Perv}_{G \times \mathbb{G}_m} (\mathfrak{g})$ of $G \times \mathbb{G}_m$-equivariant perverse sheaves on $\mathfrak{g}$ to itself. Let $d^{\vee}$ denote the complex dimension of $G \times^B I$. We know that $F(R\pi_{M*}\underline{\mathbb{C}}[d])=R\pi_{I*}\underline{\mathbb{C}}[d^{\vee}]$ and $F$ maps a simple summand of $R\pi_{M*}\underline{\mathbb{C}}[d]$ to a simple summand of $R\pi_{I*}\underline{\mathbb{C}}[d^{\vee}]$ (see \cite{bualibanu2022perverse}[section 2.2]). $R\pi_{I*}\underline{\mathbb{C}}[d^{\vee}]$ is supported on the nilpotent cone $\mathcal{N} \subset \mathfrak{g}$, and the $G \times \mathbb{G}_m$-orbits in $\mathcal{N}$ and the equivariant perverse sheaves on these nilpotent orbits are very well-understood (as they are the main subject of Springer theory). Since $G \times^B I$ is nonsingular, we can apply the decomposition theorem to $R\pi_{I*}\underline{\mathbb{C}}[d^{\vee}]$ and get 
\begin{equation} \label{Brdecomp2}
R\pi_{I*}\underline{\mathbb{C}}[d^{\vee}]=\bigoplus_{(N, \mathcal{L})} \operatorname{IC}(\overline{C(N)}, \mathcal{L})[b_{N, \mathcal{L}}]
\end{equation}
where $N \in \mathcal{N}$ is a nilpotent element, $C(N)$ is the conjugacy class (nilpotent orbit) of $N$, $\mathcal{L}$ is an irreducible $G$-equivariant local system on $C(N)$, $\operatorname{IC}(\overline{C(N)}, \mathcal{L})$ is the intermediate extension of $\mathcal{L}[\operatorname{dim}C(N)]$ from $C(N)$ to $\overline{C(N)}$ and $b_{N, \mathcal{L}}$ is an integer for shifting. The pairs $(N, \mathcal{L})$ appearing in the sum are determined by $\mathfrak{g}$ and $I$. Applying the Fourier-Sato transform $F$ (which is an autoequivalence) to both sides of Equation~\ref{Brdecomp2} and comparing it to Equation~\ref{Brdecomp}, we see that for each summand $\operatorname{IC}(\overline{C(N)}, \mathcal{L})$ of $R\pi_{I*}\underline{\mathbb{C}}[d^{\vee}]$, $F(\operatorname{IC}(\overline{C(N)}, \mathcal{L}))=\operatorname{IC}(\overline{\mathfrak{g}^{rs}}, \mathcal{V}_i)$, where the right hand side is a certain summand of $R\pi_{M*}\underline{\mathbb{C}}[d]$. Moreover, the correspondence 
\[ \mathcal{V}_{i, y} \mapsto \operatorname{IC}(\overline{C(N)}, \mathcal{L}) \text{ such that } F(\operatorname{IC}(\overline{C(N)}, \mathcal{L}))=\operatorname{IC}(\overline{\mathfrak{g}^{rs}}, \mathcal{V}_i)  \]
is exactly the Springer correspondence that sends the trivial representation of $W$ to $\operatorname{IC}(\left\{0\right\}, \underline{\mathbb{C}})$. Now we have reduced the problem of computing the dot action of $W$ on $H^*(\operatorname{Hess}(M, y))$ to the decomposition of $R\pi_{I*}\underline{\mathbb{C}}[d^{\vee}]$ into simple summands. That is, if we know the decomposition of $R\pi_{I*}\underline{\mathbb{C}}[d^{\vee}]$, applying the Springer correspondence (whose $G_2$ case will be explicitly given shortly), we get the decomposition of the dot action on $H^*(\operatorname{Hess}(M, y))$ as a direct sum of irreducible $W$ representations. Note that if a nilpotent element $N \in \mathcal{N}$ is in the image of $\pi_I$, $H^*(R\pi_{I*}\underline{\mathbb{C}}|_N) = H^*(\pi_{I}^{-1}(N))$. This is where the knowledge of Hessenberg ideal fibers, in particular the results from section \ref{g2hif}, enters the computation of dot actions.

We briefly recall the theory of Springer correspondence. For simplicity, assume that $G$ is a connected simple algebraic group over $\mathbb{C}$ of the adjoint form. Since both Hessenberg varieties and Hessenberg ideal fibers are subvarieties of $G/B$ defined via subspaces of the Lie algebra $\mathfrak{g}$, different choices of the group $G$ does not affect them as long as $G$ is of the same type. The theory of Springer correspondence states that:

\begin{theorem*}[{\cite{borho1983partial}}, p.~48, section 2.2]
Let $W$ be the Weyl group of $G$, $\operatorname{Irr}(W)$ be the set of isomorphic classes of irreducible $W$ representations and $\{(N, \psi)\}/G$ be the $G$-conjugacy classes of pairs $(N, \psi)$ where $N$ is a nilpotent element and $\psi$ is an irreducible character of $A(N)=C_G(N)/C_G^{\circ}(N)$. Then there is a meaningful injection $\operatorname{Irr}(W) \longrightarrow \{(N, \psi)\}/G$.
\end{theorem*}

Since Springer's original paper \cite{springer1976trigonometric}, several different constructions of the Springer correspondence have arisen. They result in two different versions of the map $\operatorname{Irr}(W) \longrightarrow \{(N, \psi)\}/G$, which differ from each other by tensoring with the sign representation of $W$. The one constructed via the Fourier-Sato transform, which is used here, is characterized by sending the trivial representation of $W$ to the pair $(0, \mathbf{1})$, where $0 \in \mathfrak{g}$ is the nilpotent element and $\mathbf{1}$ is the trivial representation of the trivial group. For any nilpotent element $N$, the set of irreducible $G$-equivariant local systems over the conjugacy class $C(N)$ is classified by the set of irreducible characters of $A(N)$. Therefore, there is a bijection between the set $\{(N, \psi)\}/G$ and the set of simple $G \times \mathbb{G}_m$-equivariant $\operatorname{IC}$-sheaves 
\[ \{\operatorname{IC}(\overline{C(N)}, \mathcal{L}) \ | \ N \text{ is nilpotent}, \mathcal{L} \text{ is a } G\text{-equivariant} \text{ irreducible local system over } C(N)\}. \]
In particular, the pair $(0, \mathbf{1})$ corresponds to $\operatorname{IC}(\{0\}, \underline{\mathbb{C}})$ under this bijection.

Next we explicitly describe the Springer correspondence for type $G_2$. In the rest of this section, let $G$ denote the adjoint form of type $G_2$. We assume the same presentation of $G$ as in subsection \ref{strofg2}. The character table of the Weyl group $W$ of $G_2$ is given in Table~\ref{tab:chartableg2}, which is taken from \cite{carter1985finite}[p.~412] with new character names added in the first column. The nilpotent orbits and their respective dimensions are recalled in Table~\ref{tab:g2nilporbits}. Their closure relationship is simple: the closure of any nilpotent orbit is the union of itself and all the other lower dimensional orbits. The Springer correspondence for type $G_2$ is given in Table~\ref{tab:sprcorg2}, which is obtained from \cite{carter1985finite}[p.~427] after tensoring with the sign representation $\epsilon$ and adding two additional columns for the $\operatorname{IC}$-sheaves. The original table from \cite{carter1985finite}[p.~427] gives the version of Springer correspondence that sends the sign representation $\epsilon$ of $W$ to the pair $(0, \mathbf{1})$, and that is why we tensor its last column with the sign representation in order to get Table~\ref{tab:sprcorg2}.

\begin{table}
\caption{Character table of $W(G_2)$} \label{tab:chartableg2}
\centering
\begin{tabular}{|c|c|c|c|c|c|c|}
\hline
 & $1$ & $s$ & $t$  & $st$ & $(st)^2$ & $(st)^3$ \\ \hline
 $\mathbf{1}=\phi_{1, 0}$ & 1  & 1  & 1  & 1  & 1  & 1  \\ \hline
 $\epsilon_1=\phi_{1, 3}'$ & 1  & 1  & -1  & -1  & 1  & -1  \\ \hline
 $\epsilon_2=\phi_{1, 3}''$ & 1  & -1 & 1 & -1  & 1 & -1 \\ \hline
 $\epsilon=\phi_{1, 6}$   & 1  & -1 & -1 & 1 & 1  & 1  \\ \hline
 $\chi_1=\phi_{2, 1}$ & 2 & 0  & 0 & 1 & -1  & -2 \\ \hline
 $\chi_2=\phi_{2,2}$ & 2 & 0  & 0 & -1 & -1 & 2 \\ \hline
\end{tabular}
\end{table}

\begin{table}
\def\arraystretch{1.4}
\caption{Nilpotent orbits for type $G_2$} \label{tab:g2nilporbits}
\centering
\begin{tabular}{|c|c|c|c|c|c|}
\hline
orbit     & $\{0\}$ & $A_1$  & $\tilde{A}_1$  & $G_2(a_1)$  & $G_2$  \\ \hline
dimension & 0  & 6  & 8  & 10  & 12  \\ \hline
\end{tabular}
\end{table}

\begin{table}
\def\arraystretch{1.4}
\caption{Springer correspondence for type $G_2$} \label{tab:sprcorg2}
\centering
\begin{tabular}{|c|c|c|c|c|c|}
\hline
\begin{tabular}[c]{@{}c@{}}nilpotent\\ orbit \end{tabular} &      $A(N)$              & \begin{tabular}[c]{@{}c@{}}character\\ of $A(N)$ \end{tabular} & \begin{tabular}[c]{@{}c@{}}character\\ of $W$ \end{tabular} & \begin{tabular}[c]{@{}c@{}}$\operatorname{IC}$-sheaf \\ over $\mathcal{N}$ \end{tabular} & \begin{tabular}[c]{@{}c@{}}$\operatorname{IC}$-sheaf \\ over $\mathfrak{g}$ \end{tabular} \\ \hline
\{0\} &  1  &  $\mathbf{1}$ &  $\mathbf{1}$ & $\operatorname{IC}(\{0\}, \underline{\mathbb{C}})$  &   $\operatorname{IC}(\overline{\mathfrak{g}^{rs}}, \mathbf{1})$  \\ \hline
$A_1$ &  1  &  $\mathbf{1}$ &  $\epsilon_1$ & $\operatorname{IC}(A_1, \underline{\mathbb{C}})$    &   $\operatorname{IC}(\overline{\mathfrak{g}^{rs}}, \epsilon_1)$  \\ \hline
$\tilde{A}_1$  & 1 & $\mathbf{1}$     &  $\chi_2$   &  $\operatorname{IC}(\tilde{A}_1, \underline{\mathbb{C}})$  & $ \operatorname{IC}(\overline{\mathfrak{g}^{rs}}, \chi_2)$  \\ \hline
\multirow{3}{*}{$G_2(a_1)$} & \multirow{3}{*}{$S_3$} & $\psi_3$ & $\chi_1$  & $\operatorname{IC}(G_2(a_1), \psi_3)$   &      $\operatorname{IC}(\overline{\mathfrak{g}^{rs}}, \chi_1)$ \\ \cline{3-6} 
&       & $\psi_{21}$  & $\epsilon_2$   & $\operatorname{IC}(G_2(a_1), \psi_{21})$    &    $\operatorname{IC}(\overline{\mathfrak{g}^{rs}}, \epsilon_2)$          \\ \cline{3-6} 
&       & $\psi_{111}$ &                &        $\operatorname{IC}(G_2(a_1), \psi_{111})$                                     &                                                 \\ \hline
$G_2$ & 1   & $\mathbf{1}$  &  $\epsilon$   & $\operatorname{IC}(G_2, \underline{\mathbb{C}})$    &   $\operatorname{IC}(\overline{\mathfrak{g}^{rs}}, \epsilon)$    \\ \hline
\end{tabular}
\end{table}

A few words on the reading of Table~\ref{tab:sprcorg2}. From every row of the table, we get a pair $(N, \psi)$ from the 1st and 3rd column. The character $\chi$ of the irreducible $W$ representation corresponding to the pair is in the 4th column. The $\operatorname{IC}$-sheaf supported over $\mathcal{N}$ in the 5th column corresponds to the pair $(N, \psi)$ as previously explained. The $\operatorname{IC}$-sheaf supported over $\mathfrak{g}=\overline{\mathfrak{g}^{rs}}$ in the 6th column corresponds to $\chi$ after picking a base point (see Conjecture~\ref{Brconj}). In particular, the Fourier-Sato transform $F$ always takes the 5th column to the 6th column and vice versa. $A(N)$ is the trivial group except for the orbit $G_2(a_1)$, for which the irreducible characters $\psi_{3}$, $\psi_{21}$ and $\psi_{111}$ are respectively indexed by the partions $(3)$, $(2, 1)$ and $(1, 1, 1)$ of $3$. $\psi_{3}$ is the trivial representation, $\psi_{111}$ is the sign representation and $\psi_3 + \psi_{21}$ is the natural 3-dimensional permutation representation of $S_3$. By abuse of notation, we also use $\psi_{3}$, $\psi_{21}$ and $\psi_{111}$ to denote their corresponding $G$-equivariant irreducible local systems over the nilpotent orbit $G_2(a_1)$ in the 5th column. The two blank boxes in the table means that the pair $(G_2(a_1), \psi_{111})$ does not correspond to any irreducible $W$ representation under the Springer correspondence, and that $F(\operatorname{IC}(G_2(a_1), \psi_{111}))$ is an $\operatorname{IC}$-sheaf supported on a proper closed subset of $\mathfrak{g}$, that is to say, it is not the intermediate extension of an irreducible local system on $\mathfrak{g}^{rs}$.

Besides the techniques of Fourier-Sato transform and Springer correspondence summarized above, we need the following results for our computation. Except for Lemma~\ref{Brconjg2}, they are true for connected reductive algebraic group $G$ in general.

\begin{lemma} \label{ANactioncompatible}
Let $I$ be a Hessenberg ideal and $N$ be an element from the maximal nilpotent orbit contained in the image of $\pi_I: G \times^B I \longrightarrow  \mathfrak{g}$. Restricting both sides of Equation~\ref{Brdecomp2} to the point $N$, taking cohomology of both sides and ignoring the shiftings, we get the isomorphism $H^*(\pi_{I}^{-1}(N)) \cong \bigoplus \mathcal{L}_N$, where the direct sum is taken over those $\mathcal{L}$'s supported on the very orbit of $C(N)$ (not on a smaller orbit than $C(N)$). Then the isomorphsim is $A(N)$-equivariant with respect to the natural $A(N)$ actions on both sides. 
\end{lemma}

\begin{proof}
If shiftings are ignored, the isomorphism of vector spaces $H^*(\pi_{I}^{-1}(N)) \cong \bigoplus \mathcal{L}_N$ follows directly from the isomorphism of complexes in Equation~\ref{Brdecomp2}. By the definition of $\pi_{I}^{-1}(N)$ in Equation~\ref{hif1}, it is clear that $C_G(N)$ acts on $\pi_{I}^{-1}(N) \subset G/B$ by left multiplication. The actions of $A(N)=C_G(N)/C_G^{\circ}(N)$ on both sides of the isomorphism $H^*(\pi_{I}^{-1}(N)) \cong \bigoplus \mathcal{L}_N$ are induced by actions of $C_G(N)$ on the underlying varieties $\pi_{I}^{-1}(N)$ and $C(N)$, hence they respect the isomorphism. In addition, in the case of type $G_2$, the only nontrivial case is when $N$ comes from the orbit $G_2(a_1)$, where the actions of $A(N) \cong S_3$ on the $H^*(\pi_{I}^{-1}(N))$'s are explicitly described in Remark~\ref{subregularSF}.
\end{proof}

\begin{lemma}[Xue] \label{Brconjg2}
Conjecture~\ref{Brconj} is true when $G$ is the adjoint form of type $G_2$.
\end{lemma}

\begin{proof}
Since $F(R\pi_{M*}\underline{\mathbb{C}}[d])=R\pi_{I*}\underline{\mathbb{C}}[d^{\vee}]$, Conjecture~\ref{Brconj} is equivalent to the claim that every summand $\operatorname{IC}(\overline{C(N)}, \mathcal{L})$ in the decomposition $R\pi_{I*}\underline{\mathbb{C}}[d^{\vee}]=\bigoplus \operatorname{IC}(\overline{C(N)}, \mathcal{L})[b_{N, \mathcal{L}}]
$ comes from some irreducible $W$ representation via the Springer correspondence (shiftings ignored). In the case of type $G_2$, according to Table~\ref{tab:sprcorg2}, it means that $\operatorname{IC}(G_2(a_1), \psi_{111})$ is not a summand of any $R\pi_{I*}\underline{\mathbb{C}}[d^{\vee}]$.

For $\operatorname{IC}(G_2(a_1), \psi_{111})$ to be a summand of $R\pi_{I*}\underline{\mathbb{C}}[d^{\vee}]$ at all, the image of $\pi_I$ has to contain the orbit $G_2(a_1)$ in the first place. According to Table~\ref{tab:listhifg2}, such an $I$ can only be $I_{\beta+\alpha}$, $I_{\beta}$, $I_{\alpha}$ or $I_{\alpha, \beta}$.

When $I=I_{\alpha, \beta}=\mathfrak{u}$, $\pi_I: G \times^B I \longrightarrow \mathcal{N}$ is the Springer resolution of the nilpotent cone $\mathcal{N}$, and the decomposition of $R\pi_{I*}\underline{\mathbb{C}}[d^{\vee}]$ is well-known: every summand of it does come from some irreducible $W$ representation via the Springer correspondence. In fact, Borho and MacPherson constructed Springer correspondence via this decomposition (see \cite{borho1983partial}).

When $I=I_{\beta+\alpha}$,  $I_{\beta}$ or $I_{\alpha}$, $G_2(a_1)$ is the maximal nilpotent orbit contained in the image of $\pi_I$. Let $N$ be an element of $G_2(a_1)$. From Remark~\ref{subregularSF}, we deduce that the actions of $A(N) \cong S_3$ on the $H^i(\pi_{I}^{-1}(N))$'s ($i=0, 2$) are either the trivial representation or the 3-dimensional permutation representation. In the light of Lemma~\ref{ANactioncompatible}, $\operatorname{IC}(G_2(a_1), \psi_{111})$ can never be a summand of $R\pi_{I*}\underline{\mathbb{C}}[d^{\vee}]$, because $\psi_{111}$, being the sign representation, can never be a subrepresentation of either the trivial or the permutation representation. 

We have now finished the proof of the lemma.
\end{proof}

\begin{proposition}[Brosnan] \label{Brletterlem}
\hspace{1cm} 
\begin{itemize}
\item[(1)] Let $M$ be a Hessenberg subspace, $y \in \mathfrak{g}^{rs}$ a regular semisimple element, $n=\operatorname{dim}_{\mathbb{C}}\operatorname{Hess}(M,y)$ and $\eta \in H^2(G/B)$ be the first Chern class of a hyperplane line bundle over $G/B$ that is invariant under the dot action by $W$. Then the cup product map $\eta \wedge: H^i(\operatorname{Hess}(M, y)) \longrightarrow H^{i+2}(\operatorname{Hess}(M, y))$ respects the dot action. As a result, $\eta^i \wedge: H^{n-i}(\operatorname{Hess}(M, y)) \longrightarrow  H^{n+i}(\operatorname{Hess}(M,y))$ is an isomorphism of dot actions. In short, the dot action is compatible with the Hard Lefschetz theorem.
\item[(2)] The action of $W$ on $H^0(\operatorname{Hess}(M, y))$ is a permutation representation for any Hessenberg subspace $M$ and regular semisimple element $y$. As a result, if $\operatorname{dim}_{\mathbb{C}}H^0(\operatorname{Hess}(M,y))=1$, it is the trivial representation.
\item[(3)] Let $N$ and $M$ be two Hessenberg subspaces with $N \subset M$. If $\operatorname{Hess}(N, y)$ is connected, the restriction map $H^2(\operatorname{Hess}(M, y)) \longrightarrow  H^2(\operatorname{Hess}(N, y))$ is injective and respects the dot action.
\end{itemize}
\end{proposition}

\begin{proof}
Let $\eta_0$ be the first Chern class of an arbitrary hyperplane line bundle over $G/B$. Since the Weyl group $W$ is finite, we can define $\eta=\sum_{\sigma \in W} \sigma \cdot \eta_0$, which is clearly a $W$-invariant first Chern class. From Tymoczko's definition of the dot action in terms of moment graphs and equivariant cohomology (see \cite{tymoczko2007permutation}), it is clear that the dot action respects cup product. Therefore, claim (1) follows from the usual Hard Lefschetz theorem.

Claim (2) and (3) also follow easily from Tymoczko's definition.
\end{proof}

\begin{theorem}[{\cite{brosnan2018unit}}, Theorem 76] \label{genlichpin}
Let $G$ be a connected reductive algebraic group over $\mathbb{C}$ and $x_J \in \mathfrak{g}$ be a regular element. Take the Jordan decomposition $x_J=s_J+n_J$ where $s_J$ is semisimple and $n_J$ is nilpotent. Then the centralizer $C_{\mathfrak{g}}(s_J)$ is a Levi subalgebra of $\mathfrak{g}$ and $n_J$ is a regular nilpotent element of $C_{\mathfrak{g}}(s_J)$. Let $W_J$ denote the Weyl group of $C_{\mathfrak{g}}(s_J)$. Let $M$ be a Hessenberg subspace of $\mathfrak{g}$ and $y \in \mathfrak{g}$ be a regular semisimple element. We have
\[ \operatorname{dim}_{\mathbb{C}}H^i(\operatorname{Hess}(M,x_J))=\operatorname{dim}_{\mathbb{C}}H^i(\operatorname{Hess}(M, y))^{W_J}  \]
for all $i$.
\end{theorem}

Theorem~\ref{genlichpin} is proved in \cite{brosnan2018unit} only for type $A$, but it can be generalized to other types without difficulty. Note that $\operatorname{Hess}(M,x_J)$ is the Hessenberg variety associated to a regular element $x_J$, so it is different from our main object of study---regular semisimple Hessenberg variety $\operatorname{Hess}(M, y)$. In general, $\operatorname{Hess}(M,x_J)$ is not smooth, but Precup shows that it is still paved by affines and gives formula for $H^*(\operatorname{Hess}(M,x_J))$ in \cite{precup2013affine}.

\subsection{Computation} \label{dotactioncomputation}
In this subsection, we fix $G$ to be the adjoint form of type $G_2$ and compute the dot action of $W$ on $H^*(\operatorname{Hess}(M, y))$ for any Hessenberg subspace $M$. To be specific, we decompose the character of each $H^{2i}(\operatorname{Hess}(M, y))$ as a direct sum of irreducible characters of $W$ (listed in Table~\ref{tab:chartableg2}). Note that $H^*(\operatorname{Hess}(M, y))$ are supported on the even degrees because it is paved by affines. Besides the same setups for $G_2$ as in subsection \ref{strofg2}, we introduce the following notational conventions. 

Let $M \subset \mathfrak{g}$ denote a Hessenberg subspace in general and $I \subset \mathfrak{u}$ denote the Hessenberg ideal ``dual'' to $M$. That is, $I=M^{\perp}$ and
$M \cong \mathfrak{t} \oplus \bigoplus_{\alpha \in \Phi(M)} \mathfrak{g}_{\alpha}$, $I \cong \bigoplus_{\alpha \in \Phi(I)} \mathfrak{g}_{\alpha}$, where $\Phi(I)=-(\Phi \setminus \Phi(M))$. Let $y \in \mathfrak{g}$ be a regular semisimple element. We define the Poincar\'e polynomial of $\operatorname{Hess}(M, y)$ to be 
\[ P_{M}(q)=\sum_{i=0}^n \chi_{H^{2i}(\operatorname{Hess}(M,y))}q^i, \ \operatorname{deg}(q)=2,  \]
where $n=\operatorname{dim}_{\mathbb{C}}\operatorname{Hess}(M, y)$ and the coefficient $\chi_{H^{2i}(\operatorname{Hess}(M,y))}$ is the character of $H^{2i}(\operatorname{Hess}(M,y))$ as a $W$ representation.

We can easily enumerate all the Hessenberg ideals. Using the same notation as in subsection \ref{strofg2}, they are $I_{\emptyset}$, $I_{2\beta+3\alpha}$, $I_{\beta+3\alpha}$, $I_{\beta+2\alpha}$, $I_{\beta+\alpha}$, $I_{\alpha}$, $I_{\beta}$ and $I_{\alpha, \beta}$. 

To state the plan for computation in more details, we need the following simple lemma.

\begin{lemma} \label{dotcomplem}
Let $G$ be the adjoint form of type $G_2$.
\begin{itemize}
\item[(1)] $\operatorname{dim}_{\mathbb{C}}\mathfrak{g}=14$, $\operatorname{dim}_{\mathbb{C}}\mathfrak{b}=8$ and $\operatorname{dim}_{\mathbb{C}}\mathfrak{u}=6$.
\item[(2)] $d^{\vee}=\operatorname{dim}_{\mathbb{C}}(G \times^B I)=6+ |\Phi(I)|$ and $d=\operatorname{dim}_{\mathbb{C}}(G \times^B M)=20- |\Phi(I)|$, where $|\Phi(I)|$ is the order of $\Phi(I)$.
\item[(3)] $n=\operatorname{dim}_{\mathbb{C}}\operatorname{Hess}(M, y)=6-|\Phi(I)|$.
\item[(4)] Let $\chi$ be an irreducible character of $W$. If $\operatorname{IC}(\overline{C(N)}, \mathcal{L})=F(\operatorname{IC}(\overline{\mathfrak{g}^{rs}}, \chi))$ and $\operatorname{IC}(\overline{C(N)}, \mathcal{L})[b]$ is a direct summand of $R\pi_{I*}\underline{\mathbb{C}}[d^{\vee}]$, then $\chi q^{(d-b-14)/2}$ is a summand of $P_M(q)$.
\item[(5)] $\sum_{i=0}^n \operatorname{dim}_{\mathbb{C}}H^{2i}(\operatorname{Hess}(M, y))=12$.
\end{itemize}
\end{lemma}
\begin{proof}
(1) and (2) are straightforward. 

For (3), by \cite{de1992hessenberg}[Theorem 6], $\operatorname{dim}_{\mathbb{C}}\operatorname{Hess}(M, y)=\operatorname{dim}_{\mathbb{C}}M/\mathfrak{b}=\operatorname{dim}_{\mathbb{C}}\mathfrak{g}-\operatorname{dim}_{\mathbb{C}}I-\operatorname{dim}_{\mathbb{C}}\mathfrak{b}=14-|\Phi(I)|-8=6-|\Phi(I)|$.

For (4), since $R\pi_{M*}\underline{\mathbb{C}}[d]=F(R\pi_{I*}\underline{\mathbb{C}}[d^{\vee}])$, if $\operatorname{IC}(\overline{C(N)}, \mathcal{L})[b]$ is a direct summand of $R\pi_{I*}\underline{\mathbb{C}}[d^{\vee}]$, $\operatorname{IC}(\overline{\mathfrak{g}^{rs}}, \chi)[b]=F(\operatorname{IC}(\overline{C(N)}, \mathcal{L})[b])$ is a direct summand of $R\pi_{M*}\underline{\mathbb{C}}[d]$. Hence $\operatorname{IC}(\overline{\mathfrak{g}^{rs}}, \chi)[b-d]$ is a direct summand of $R\pi_{M*}\underline{\mathbb{C}}$. Restricting both to the point $y$ and taking cohomology, we deduce that $H^{2i+b-d}(\operatorname{IC}(\overline{\mathfrak{g}^{rs}}, \chi)|_y)$ is a $W$ subrepresentation of $H^{2i}(\operatorname{Hess}(M, y))$ for any $i \in \mathbb{Z}$. By the definition of $\operatorname{IC}$-sheaves, $\operatorname{IC}(\overline{\mathfrak{g}^{rs}}, \chi)|_{\mathfrak{g}^{rs}} \cong \chi[\operatorname{dim}_{\mathbb{C}}\mathfrak{g}] = \chi[14]$, so $H^{2i+b-d}(\operatorname{IC}(\overline{\mathfrak{g}^{rs}}, \chi)|_y)$ is nonzero only when $2i+b-d=-14$. In this case, $i=(d-b-14)/2$ and $\chi \cong H^{-14}(\operatorname{IC}(\overline{\mathfrak{g}^{rs}}, \chi)|_y)$ is a subrepresentation of $H^{2i}(\operatorname{Hess}(M, y))$. Therefore $\chi q^{i}=\chi q^{(d-b-14)/2}$ is a summand of $P_M(q)$. By abuse of notation, we also used $\chi$ to denote its corresponding irreducible local system over $\mathfrak{g}^{rs}$ and irreducible $W$ representation.

For (5), according to \cite{de1992hessenberg}, $\operatorname{Hess}(M, y)$ is paved by affines with the set of cells in bijection with $W$. Hence $\sum_{i=0}^n \operatorname{dim}_{\mathbb{C}}H^{2i}(\operatorname{Hess}(M, y))=|W|=12$.
\end{proof}

Our plan for computation is the following:
\begin{itemize}
\item[(1)] Start from the smallest Hessenberg ideal $I=I_{\emptyset}$ and compute $P_M(q)$ in the increasing order of $\operatorname{dim}_{\mathbb{C}}I$. 
\item[(2)] The previous results, combined with Proposition~\ref{Brletterlem}, provide certain summands of the new $P_M(q)$ immediately.
\item[(3)] Let $N$ be an element from the maximal nilpotent orbit contained in the image of $\pi_I: G \times^B I \longrightarrow  \mathfrak{g}$. In the light of Lemma~\ref{ANactioncompatible} and Lemma~\ref{dotcomplem} (4), we obtain some new summands of $P_M(q)$ by inspecting $H^*(\pi_{I}^{-1}(N))$.
\item[(4)] If the total dimension of the summands above is already 12, according to Lemma~\ref{dotcomplem} (5), we have the complete $P_M(q)$. Otherwise we bring in Theorem~\ref{genlichpin} to figure out the rest of the summands.
\end{itemize}

\noindent $\bullet$ $I=I_{\emptyset }$

In this case, $M=\mathfrak{g}$ and the map $\pi_M: G \times^B M \longrightarrow \mathfrak{g}$ is isomorphic to the projection $\operatorname{pr}_2: G/B \times \mathfrak{g} \longrightarrow \mathfrak{g}$. Then $\operatorname{Hess}(M, y) \cong G/B$ and $R\pi_{M*}\underline{\mathbb{C}}[d]|_{\mathfrak{g}^{rs}}$ is a constant sheaf with fibers isomorphic to $H^*(G/B)$. Therefore, every $H^{2i}(\operatorname{Hess}(M, y))$ is the trivial representation of $W$ and we can figure out $P_M(q)$ from the cell structure of $G/B$. In fact, $P_M(q)=\mathbf{1}+ \mathbf{2}q + \mathbf{2}q^2 + \mathbf{2}q^3 + \mathbf{2}q^4 + \mathbf{2}q^5 + \mathbf{1}q^6$, where $\mathbf{1}$ denotes the trivial representation of $W$ and $\mathbf{2}$ denotes $\mathbf{1} \oplus \mathbf{1}$.

\noindent $\bullet$ $I=I_{2\beta+3\alpha }$

In this case, $|\Phi(I)|=1$. By Lemma~\ref{dotcomplem}, $d=19$, $d^{\vee}=7$, $n=\operatorname{dim}_{\mathbb{C}}\operatorname{Hess}(M, y)=5$. From \cite{de1992hessenberg}[Corollary 9], it follows easily that $\operatorname{Hess}(M, y)$ is connected. (In fact, $\operatorname{Hess}(M, y)$ is connected for $I=I_{\emptyset}$, $I_{2\beta+3\alpha}$, $I_{\beta+3\alpha}$, $I_{\beta+2\alpha}$, $I_{\beta+\alpha}$). By Proposition~\ref{Brletterlem} (2), $H^0(\operatorname{Hess}(M, y)) \cong \mathbf{1}$. Using the $H^2$ term of the previous case (when $I=I_{\emptyset}$), we deduce by Proposition~\ref{Brletterlem} (3) that $H^2(\operatorname{Hess}(M, y))$ has $\mathbf{2}$ as a subrepresentation. Now Proposition~\ref{Brletterlem} (1) implies that $P_M(q)$ must have $\mathbf{1}+\mathbf{2}q+\mathbf{2}q^2+\mathbf{2}q^3+\mathbf{2}q^4+\mathbf{1}q^5$ as summands. By Lemma~\ref{dotcomplem} (5), there are only 2 additional dimensions that remain to be figured out. 

We know from Table~\ref{tab:listhifg2} that $A_1$ is the maximal nilpotent orbit contained in the image of $\pi_I: G \times^B I \longrightarrow  \mathfrak{g}$. Let $N$ be an element from $A_1$. Then $\pi_{I}^{-1}(N)=\overline{X_{s}} \cong \mathbb{P}^1$. Hence we have $H^*(\pi_{I}^{-1}(N)) = \mathbb{C}[0] \oplus \mathbb{C}[-2]$ and $H^*(R\pi_{I*}\underline{\mathbb{C}}[7]|_N) = \mathbb{C}[7] \oplus \mathbb{C}[5]$. Since $A_1$ is the maximal orbit contained in the image of $\pi_I$, $R\pi_{I*}\underline{\mathbb{C}}[7]$ is the direct sum of shifted copies of $\operatorname{IC}(\{0\}, \underline{\mathbb{C}})$ and $\operatorname{IC}(A_1, \underline{\mathbb{C}})$. We know that $\operatorname{IC}(A_1, \underline{\mathbb{C}})|_N \cong \mathbb{C}[6]$ and $\operatorname{IC}(\{0\}, \underline{\mathbb{C}})|_N \cong 0$. Because the decomposition of $R\pi_{I*}\underline{\mathbb{C}}[7]$ satisfies the Hard Lefschetz theorem (see \cite{deligne1982faisceaux}[Th\'eor\`eme 6.2.10]), $\operatorname{IC}(A_1, \underline{\mathbb{C}})[1]$ and $\operatorname{IC}(A_1, \underline{\mathbb{C}})[-1]$ must be the only summands of $R\pi_{I*}\underline{\mathbb{C}}[7]$ supported on $A_1$. All the other summands are shifted copies of $\operatorname{IC}(\{0\}, \underline{\mathbb{C}})$. By Lemma~\ref{dotcomplem} (4) and Table~\ref{tab:sprcorg2}, $\operatorname{IC}(A_1, \underline{\mathbb{C}})[1]$ and $\operatorname{IC}(A_1, \underline{\mathbb{C}})[-1]$ correspond to the summands $\epsilon_1 q^2 + \epsilon_1 q^3$ of $P_M(q)$, which are the 2 additional dimensions expected from the previous paragraph. As a result, $P_M(q)=\mathbf{1}+\mathbf{2}q+(\mathbf{2}+\epsilon_1) q^2+(\mathbf{2}+\epsilon_1) q^3+\mathbf{2}q^4+\mathbf{1}q^5$.

\noindent $\bullet$ $I=I_{\beta+3\alpha }$

In this case, $|\Phi(I)|=2$, $d=18$, $d^{\vee}=8$, $n=4$. $\operatorname{Hess}(M, y)$ is still connected. Arguing with Proposition~\ref{Brletterlem} in the same manner as the previous case, we deduce that $P_M(q)$ must have summands $\mathbf{1}+\mathbf{2}q+\mathbf{2}q^2+\mathbf{2}q^3+\mathbf{1}q^4$. There are 4 additional dimensions to be figured out.

According to Table~\ref{tab:listhifg2}, $A_1$ is still the maximal nilpotent orbit contained in the image of $\pi_I$. Let $N$ be an element from $A_1$. Then $\pi_{I}^{-1}(N)=\overline{X_{st}}$. Hence we have $H^*(\pi_{I}^{-1}(N))=\mathbb{C}[0] \oplus \mathbb{C}^2[-2] \oplus \mathbb{C}[-4]$ and $H^*(R\pi_{I*}\underline{\mathbb{C}}[8]|_N)=\mathbb{C}[8] \oplus \mathbb{C}^2[6] \oplus \mathbb{C}[4]$. $R\pi_{I*}\underline{\mathbb{C}}[8]$ is still the direct sum of shifted copies of $\operatorname{IC}(A_1, \underline{\mathbb{C}})$ and $\operatorname{IC}(\{0\}, \underline{\mathbb{C}})$, and $\operatorname{IC}(A_1, \underline{\mathbb{C}})|_N \cong \mathbb{C}[6]$, $\operatorname{IC}(\{0\}, \underline{\mathbb{C}})|_N \cong 0$. By the same Hard Lefschetz argument, $R\pi_{I*}\underline{\mathbb{C}}[8]$ must have summands $\operatorname{IC}(A_1, \underline{\mathbb{C}})[2]$, $\operatorname{IC}(A_1,\underline{\mathbb{C}} )[-2]$ and $\operatorname{IC}(A_1, \underline{\mathbb{C}})^{\oplus 2}$. By Lemma~\ref{dotcomplem} (4) and Table~\ref{tab:sprcorg2}, they correspond to the summands $\epsilon_1 q + \epsilon_1 q^3 +2\epsilon_1 q^2$, which are exactly the 4 additional dimensions. Hence $P_M(q)=\mathbf{1}+(\mathbf{2}+\epsilon_1) q+(\mathbf{2}+2\epsilon_1) q^2+(\mathbf{2}+\epsilon_1) q^3+\mathbf{1}q^4$, where $2 \epsilon_1$ means $\epsilon_1 \oplus \epsilon_1$.

\noindent $\bullet$ $I=I_{\beta+2\alpha }$

In this case, $|\Phi(I)|=3$, $d=17$, $d^{\vee}=9$, $n=3$. $\operatorname{Hess}(M, y)$ is connected. Similarly, $P_M(q)$ must have summands $\mathbf{1}+(\mathbf{2}+\epsilon_1) q+(\mathbf{2}+\epsilon_1) q^2+\mathbf{1} q^3$. There are 4 additional dimensions to be figured out.

Now $\tilde{A}_1$ is the maximal orbit contained in the image of $\pi_I$. Let $N$ be an element from $\tilde{A}_1$. $\pi_{I}^{-1}(N)=\overline{X_{t}} \cong \mathbb{P}^1$. Hence we have $H^*(\pi_{I}^{-1}(N))=\mathbb{C}[0] \oplus \mathbb{C}[-2]$ and $H^*(R\pi_{I*}\underline{\mathbb{C}}[9]|_N)=\mathbb{C}[9] \oplus \mathbb{C}[7]$. $R\pi_{I*}\underline{\mathbb{C}}[9]$ is the direct sum of shifted copies of $\operatorname{IC}(\tilde{A}_1, \underline{\mathbb{C}})$, $\operatorname{IC}(A_1, \underline{\mathbb{C}})$ and $\operatorname{IC}(\{0\}, \underline{\mathbb{C}})$, and $\operatorname{IC}(\tilde{A}_1, \underline{\mathbb{C}})|_N \cong \mathbb{C}[8]$. Therefore, $\operatorname{IC}(\tilde{A}_1, \underline{\mathbb{C}})[1]$ and $\operatorname{IC}(\tilde{A}_1, \underline{\mathbb{C}})[-1]$ are the only summands of $R\pi_{I*}\underline{\mathbb{C}}[9]$ supported on $\tilde{A}_1$. They correspond to the summands $\chi_2 q + \chi_2 q^2$, the 4 additional dimensions. Then $P_M(q)=\mathbf{1}+(\mathbf{2}+\epsilon_1+ \chi_2) q+(\mathbf{2}+\epsilon_1+\chi_2) q^2+\mathbf{1} q^3$.

\noindent $\bullet$ $I=I_{\beta+\alpha }$

In this case, $|\Phi(I)|=4$, $d=16$, $d^{\vee}=10$, $n=2$. $\operatorname{Hess}(M, y)$ is connected. $P_M(q)$ must have summands $\mathbf{1}+(\mathbf{2}+\epsilon_1+ \chi_2) q +\mathbf{1}q^2$. There are 5 additional dimensions.

$G_2(a_1)$ is the maximal orbit contained in the image of $\pi_I$. Let $N$ be an element from $G_2(a_1)$. $\pi_{I}^{-1}(N)=\{\text{3 distinct points}\}$ and $A(N) \cong S_3$ acts on $\pi_{I}^{-1}(N)$ by the natural permutation action (see Remark~\ref{subregularSF}). Hence $H^*(\pi_{I}^{-1}(N))=\mathbb{C}^3[0]$ and $H^*(R\pi_{I*}\underline{\mathbb{C}}[10]|_N) = \mathbb{C}^3 [10]$. $A(N)$ acts on $H^*(R\pi_{I*}\underline{\mathbb{C}}[10]|_N) = \mathbb{C}^3 [10]$ by the permutation action (see Lemma~\ref{ANactioncompatible}). We know that $\operatorname{IC}(G_2(a_1), \psi_3)|_N \cong \mathbb{C}[10]$, $\operatorname{IC}(G_2(a_1), \psi_{21})|_N \cong \mathbb{C}^2[10]$ and that $\psi_3 + \psi_{21}$ is the 3-dimensional permutation representation of $A(N)$. Therefore, $\operatorname{IC}(G_2(a_1), \psi_3)$ and $\operatorname{IC}(G_2(a_1), \psi_{21})$ must be the only summands of $R\pi_{I*}\underline{\mathbb{C}}[10]$ supported on $G_2(a_1)$. They correspond to summands $(\chi_1 + \epsilon_2)q$. Up to now, $P_M(q)$ have summands $\mathbf{1}+(\mathbf{2}+\epsilon_1+ \chi_2+\chi_1 + \epsilon_2) q +\mathbf{1}q^2$. There remains a 2-dimensional subrepresentation of $H^2(\operatorname{Hess}(M, y))$ to be figured out, which we denote by $\chi$. Then $P_M(q)=\mathbf{1}+(\mathbf{2}+\epsilon_1+ \chi_2+\chi_1 + \epsilon_2 + \chi) q +\mathbf{1}q^2$. 

Since $\chi$ comes from $\operatorname{IC}$-sheaves supported on nilpotent orbits strictly smaller than $G_2(a_1)$, it can be written as an integral combination $\chi=a\mathbf{1}+b \epsilon_1 +c \chi_2$. We will make a linear system of equations to solve for $a$, $b$ and $c$. 

Let $J$ be any subset of the set of simple roots $\Delta=\{\alpha, \beta\}$ of $G_2$. There always exists a semisimple element $s_J \in \mathfrak{g}$ such that $C_{\mathfrak{g}}(s_J)$ is a Levi subalgebra of $\mathfrak{g}$ whose Weyl group is $W_J$. Let $n_J \in C_{\mathfrak{g}}(s_J)$ be a regular nilpotent element and define $x_J=s_J+n_J$. Then $x_J$ is a regular element of $\mathfrak{g}$. According to Theorem~\ref{genlichpin}, $\operatorname{dim}_{\mathbb{C}}H^2(\operatorname{Hess}(M,x_J))=\operatorname{dim}_{\mathbb{C}}H^2(\operatorname{Hess}(M, y))^{W_J}$. \cite{precup2013affine}[Corollary 5.8] tells us how to compute $\operatorname{dim}_{\mathbb{C}}H^*(\operatorname{Hess}(M,x_J))$ by inspecting the intersection of $\operatorname{Hess}(M, x_J)$ with every Schubert cell of $G/B$, and the inspections are carried out in Appendix~\ref{appendix:bettinumber}. By Table~\ref{tab:betti1} and Table~\ref{tab:betti2}, we have the following. 
\[
\begin{array}{c | c | c  }
  &  J=\{\alpha, \beta\}, \   W_J=W & J=\{\beta\}, \  W_J=\langle t \rangle \\
      \hline
    \operatorname{dim}_{\mathbb{C}}H^2(\operatorname{Hess}(M, x_J)) &  2 &          6 \\
  \end{array}
\]

On the other hand, $\operatorname{dim}_{\mathbb{C}}H^2(\operatorname{Hess}(M, y))^{W_J}=\langle \chi_{H^2(\operatorname{Hess}(M, y))}, \operatorname{Ind}_{W_J}^W \mathbf{1} \rangle$, where $\chi_{H^2(\operatorname{Hess}(M, y))}$ is the character of $H^2(\operatorname{Hess}(M, y))$ as a $W$ representation and $\langle \  , \  \rangle$ is the inner product of characters. So far we have the following information,
\[
\left\{
      \begin{array}{l}
      \chi_{H^2(\operatorname{Hess}(M, y))}= \mathbf{2}+\epsilon_1+ \chi_2+\chi_1 + \epsilon_2 +  \chi          \\
      \operatorname{Ind}_W^W \mathbf{1}=\mathbf{1} \\
      \operatorname{Ind}_{\langle t \rangle}^W \mathbf{1}= \mathbf{1} + \epsilon_2 + \chi_1 +\chi_2 \\
      \langle \chi_{H^2(\operatorname{Hess}(M, y))}, \operatorname{Ind}_{W}^W \mathbf{1} \rangle =2 \\
      \langle \chi_{H^2(\operatorname{Hess}(M, y))}, \operatorname{Ind}_{\langle t \rangle}^W \mathbf{1} \rangle =6 \\
      \chi= a\mathbf{1}+b \epsilon_1 +c \chi_2 \\
      \operatorname{dim}_{\mathbb{C}} \chi=2
      \end{array}
              \right.
\]
which can be turned into the following linear system of equations.
\[ 
\left\{
      \begin{array}{l}
      a+2 = 2  \\
      a+2+1+1+1+c=6 \\
      a+b+2c=2
                \end{array}
              \right.
\]
We get $a=0$, $b=0$, $c=1$. Therefore, $P_M(q)=\mathbf{1}+(\mathbf{2}+\epsilon_1+ 2 \chi_2+\chi_1 + \epsilon_2) q +\mathbf{1}q^2$.

\noindent $\bullet$ $I=I_{\alpha }$

In this case, $|\Phi(I)|=5$, $d=15$, $d^{\vee}=11$, $n=1$. $\operatorname{Hess}(M, y)$ is no longer connected, so Proposition~\ref{Brletterlem} (3) does not apply. However, Proposition~\ref{Brletterlem} (1) implies that $H^0(\operatorname{Hess}(M, y))$ and $H^2(\operatorname{Hess}(M, y))$ are isomorphic as $W$ representations, hence $P_M(q)$ is divisible by $(1+q)$.

$G_2(a_1)$ is the maximal orbit contained in the image of $\pi_I$. Let $N$ be an element from $G_2(a_1)$. $\pi_{I}^{-1}(N) \cong \mathbb{P}^1 \amalg \mathbb{P}^1 \amalg \mathbb{P}^1$ and $A(N) \cong S_3$ acts on the set of irreducible components of $\pi_{I}^{-1}(N)$ by the natural permutation action (see Remark~\ref{subregularSF}). Hence $H^*(\pi_{I}^{-1}(N))=\mathbb{C}^3[0] \oplus \mathbb{C}^3[-2]$ and $H^*(R\pi_{I*}\underline{\mathbb{C}}[11]|_N) = \mathbb{C}^3[11] \oplus \mathbb{C}^3[9]$, where $A(N)$ acts on both $\mathbb{C}^3[11]$ and $\mathbb{C}^3[9]$ by the permutation action. Since $\operatorname{IC}(G_2(a_1), \psi_3)|_N \cong \mathbb{C}[10]$, $\operatorname{IC}(G_2(a_1), \psi_{21})|_N \cong \mathbb{C}^2[10]$ and $\psi_3 + \psi_{21}$ is the 3-dimensional permutation representation of $A(N)$, $\operatorname{IC}(G_2(a_1), \psi_{3})[1]$, $\operatorname{IC}(G_2(a_1), \psi_{21})[1]$, $\operatorname{IC}(G_2(a_1), \psi_{3})[-1]$ and $\operatorname{IC}(G_2(a_1), \psi_{21})[-1]$ are the only summands of $R\pi_{I*}\underline{\mathbb{C}}[11]$ supported on $G_2(a_1)$. They corresponds to summands $(\chi_1 +\epsilon_2)(1+q)$. Then $P_M(q)=(\chi_1+\epsilon_2+\chi)(1+q)$, where $\chi$ is a 3-dimensional subrepresentation of $H^2(\operatorname{Hess}(M, y))$ still to be figured out. Because $\chi$ comes from $\operatorname{IC}$-sheaves supported on nilpotent orbits strictly smaller than $G_2(a_1)$, we write $\chi=a\mathbf{1}+b \epsilon_1 +c \chi_2$ and solve for $a$, $b$ and $c$.

Now we use Theorem~\ref{genlichpin} for $J=\{\alpha\}$ and $J=\{\beta\}$. By Table~\ref{tab:betti3} and Table~\ref{tab:betti4}, we have
\[
\begin{array}{c | c | c  }
  &  J=\{\alpha\}, \   W_J=\langle s \rangle & J=\{\beta\}, \  W_J=\langle t \rangle \\
      \hline
    \operatorname{dim}_{\mathbb{C}}H^2(\operatorname{Hess}(M, x_J)) &  3 &          4 \\
  \end{array}
\]
Combining the information,
\[
\left\{
      \begin{array}{l}
      \chi_{H^2(\operatorname{Hess}(M, y))}= \chi_1+\epsilon_2+\chi          \\
      \operatorname{Ind}_{\langle s \rangle}^W \mathbf{1}=\mathbf{1}+ \epsilon_1 +  \chi_1 +\chi_2 \\
      \operatorname{Ind}_{\langle t \rangle}^W \mathbf{1}= \mathbf{1} + \epsilon_2 + \chi_1 +\chi_2 \\
      \langle \chi_{H^2(\operatorname{Hess}(M, y))}, \operatorname{Ind}_{\langle s \rangle}^W \mathbf{1} \rangle =3 \\
      \langle \chi_{H^2(\operatorname{Hess}(M, y))}, \operatorname{Ind}_{\langle t \rangle}^W \mathbf{1} \rangle =4 \\
      \chi= a\mathbf{1}+b \epsilon_1 +c \chi_2 \\
      \operatorname{dim}_{\mathbb{C}} \chi=3
      \end{array}
              \right.
\]
we get the following system.
\[ 
\left\{
      \begin{array}{l}
      1+a+b+c = 3  \\
      1+1+a+c=4 \\
      a+b+2c=3
                \end{array}
              \right.
\]
Then $a=1$, $b=0$ and $c=1$. As a result, $P_M(q)=(\mathbf{1}+\chi_2 +\chi_1+\epsilon_2)(1+q)$.

\noindent $\bullet$ $I=I_{\beta }$

Still, $|\Phi(I)|=5$, $d=15$, $d^{\vee}=11$, $n=1$. By an argument similar to the previous case, $P_M(q)$ is divisible by $(1+q)$.

$G_2(a_1)$ is the maximal orbit contained in the image of $\pi_I$. Let $N$ be an element from $G_2(a_1)$. $\pi_{I}^{-1}(N) = \overline{X_{s}} \cong \mathbb{P}^1$ and $A(N) \cong S_3$ fixes the unique irreducible component of $\pi_{I}^{-1}(N)$ (see Remark~\ref{subregularSF}). Therefore, $H^*(\pi_{I}^{-1}(N)) = \mathbb{C}[0] \oplus \mathbb{C}[-2]$ and $H^*(R\pi_{I*}\underline{\mathbb{C}}[11]|_N)=\mathbb{C}[11] \oplus \mathbb{C}[9]$, where $A(N)$ acts on both $\mathbb{C}[11]$ and $\mathbb{C}[9]$ trivially. Hence $\operatorname{IC}(G_2(a_1), \psi_3)[1]$ and $\operatorname{IC}(G_2(a_1), \psi_3)[-1]$ are the only summands of $R\pi_{I*}\underline{\mathbb{C}}[11]$ supported on $G_2(a_1)$. They correspond to the summands $\chi_1 (1+q)$ and $P_M(q)=(\chi_1 + \chi)(1+q)$ where $\chi=a\mathbf{1}+b \epsilon_1 +c \chi_2$ is a 4-dimensional subrepresentation of $H^2(\operatorname{Hess}(M, y))$ to be figured out.

Use Theorem~\ref{genlichpin} for $J=\{\alpha\}$ and $J=\{\beta\}$ again. By Table~\ref{tab:betti5} and Table~\ref{tab:betti6}, we have
\[
\begin{array}{c | c | c  }
  &  J=\{\alpha\}, \   W_J=\langle s \rangle & J=\{\beta\}, \  W_J=\langle t \rangle \\
      \hline
    \operatorname{dim}_{\mathbb{C}}H^2(\operatorname{Hess}(M, x_J)) &  4 &          3 \\
  \end{array}
\]
Combining the information,
\[
\left\{
      \begin{array}{l}
      \chi_{H^2(\operatorname{Hess}(M, y))}= \chi_1+\chi          \\
      \operatorname{Ind}_{\langle s \rangle}^W \mathbf{1}=\mathbf{1}+ \epsilon_1 +  \chi_1 +\chi_2 \\
      \operatorname{Ind}_{\langle t \rangle}^W \mathbf{1}= \mathbf{1} + \epsilon_2 + \chi_1 +\chi_2 \\
      \langle \chi_{H^2(\operatorname{Hess}(M, y))}, \operatorname{Ind}_{\langle s \rangle}^W \mathbf{1} \rangle =4 \\
      \langle \chi_{H^2(\operatorname{Hess}(M, y))}, \operatorname{Ind}_{\langle t \rangle}^W \mathbf{1} \rangle =3 \\
      \chi= a\mathbf{1}+b \epsilon_1 +c \chi_2 \\
      \operatorname{dim}_{\mathbb{C}} \chi=4
      \end{array}
              \right.
\]
we get the following system.
\[ 
\left\{
      \begin{array}{l}
      a+b+c+1=4  \\
      a+c+1=3 \\
      a+b+2c=4
                \end{array}
              \right.
\]
Then $a=b=c=1$ and $P_M(q)=(\mathbf{1}+\epsilon_1+\chi_2 +\chi_1)(1+q)$.

\noindent $\bullet$ $I=I_{\alpha, \beta }$

In this case, $I=\mathfrak{u}$, $M=\mathfrak{b}$, and the map $\pi_M: G \times^B M \longrightarrow \mathfrak{g}$ is isomorphic to the Grothendieck simultaneous resolution. Let $\tilde{\mathfrak{g}}^{rs}=\pi_M^{-1}(\mathfrak{g}^{rs})$. Then the restriction $\pi_M: \tilde{\mathfrak{g}}^{rs} \longrightarrow \mathfrak{g}^{rs}$ is a $W$-torsor. Therefore, $H^*(\operatorname{Hess}(M, y))$ is the regular representation of $W$ and $P_M(q)=\mathbf{1}+ \epsilon_1 + 2\chi_2 + 2\chi_1 + \epsilon_2$.

Now we have finished the classification of all dot actions for type $G_2$.

In summary, we have the following theorem.

\begin{theorem}[Brosnan, Xue] \label{thmBX} 
Table~\ref{tab:listdotg2} gives the complete list of Tymoczko's dot actions on the cohomology of regular semisimple Hessenberg varieties when $G$ is of type $G_2$.
\begin{table}[h]
\caption{Dot actions for type $G_2$} \label{tab:listdotg2}
\centering
\begin{tabular}{|l|l|}
\hline
$I$ & $P_M(q)$ \\ \hline
$I_{\emptyset}$ & $\mathbf{1}+ \mathbf{2}q + \mathbf{2}q^2 + \mathbf{2}q^3 + \mathbf{2}q^4 + \mathbf{2}q^5 + \mathbf{1}q^6$  \\ \hline
$I_{2\beta+3\alpha}$ & $\mathbf{1}+\mathbf{2}q+(\mathbf{2}+\epsilon_1) q^2+(\mathbf{2}+\epsilon_1) q^3+\mathbf{2}q^4+\mathbf{1}q^5
$  \\ \hline
$I_{\beta+3\alpha}$ & $\mathbf{1}+(\mathbf{2}+\epsilon_1) q+(\mathbf{2}+2\epsilon_1) q^2+(\mathbf{2}+\epsilon_1) q^3+\mathbf{1}q^4$ \\ \hline
$I_{\beta+2\alpha}$ & $\mathbf{1}+(\mathbf{2}+\epsilon_1+ \chi_2) q+(\mathbf{2}+\epsilon_1+\chi_2) q^2+\mathbf{1} q^3$  \\ \hline
$I_{\beta+\alpha}$ & $\mathbf{1}+(\mathbf{2}+\epsilon_1+ 2 \chi_2+\chi_1 + \epsilon_2) q +\mathbf{1}q^2$ \\ \hline
$I_{\alpha}$ & $(\mathbf{1}+\chi_2 +\chi_1+\epsilon_2)(1+q)$ \\ \hline
$I_{\beta}$ & $(\mathbf{1}+\epsilon_1+\chi_2 +\chi_1)(1+q)$ \\ \hline
$I_{\alpha, \beta}$ & $\mathbf{1}+ \epsilon_1 + 2\chi_2 + 2\chi_1 + \epsilon_2$ \\ \hline
\end{tabular}
\end{table}
\end{theorem}


\section{Type \texorpdfstring{$F_4$}{} and \texorpdfstring{$E_6$}{}} \label{F4E6pf}
In this section we complete the proof of Theorem~\ref{dclpthm}. To be more specific, all that remains is the following proposition.

\begin{proposition} \label{F4E6}
Let $G$ be a connected simple algebraic group over $\mathbb{C}$ of type $F_4$ or $E_6$. Let $B$ be a Borel subgroup of $G$ and $N \in \mathfrak{g}$ be a distinguished nilpotent element whose associated parabolic $P$ contains $B$. Following the notation in subsection \ref{assparab}, let $P=LU_P$ be the Levi decomposition of $P$ so that $\mathfrak{l}=\mathfrak{g}(0)$ and $\mathfrak{u}_P=\bigoplus_{i > 0}\mathfrak{g}(i)$. $L \cap B$ is a Borel subgroup of $L$. Then for any $(L \cap B)$-stable subspace $U \subset \mathfrak{g}(2)$, the variety $X_U=(L, L \cap B, \mathfrak{g}(2), U, N)$ is paved by affines whenever it is not empty.
\end{proposition}

We prove the statement above by giving sufficiently concrete descriptions of all the nonempty $(L, L \cap B, \mathfrak{g}(2), U, N)$'s for every distinguished nilpotent orbit of $F_4$ and $E_6$. The proposition should be evident towards the end of this section. 

\subsection{Computational setups} \label{compsetup}
Let $G$ denote a connected simple algebraic group over $\mathbb{C}$ of type $F_4$ or $E_6$. Fix a Borel subgroup $B$ of $G$ and choose a maximal torus $T \subset B$. Let $\Phi$ denote the set of roots with respect to $T$ and $B$ and $\Delta$ denote the set of simple roots. By abuse of notation, $\Phi$ also denotes the Dynkin diagram of $\mathfrak{g}$.

For any root $\gamma \in \Phi$, let $X_{\gamma}$ be the 1-dimensional unipotent subgroup of $G$ whose Lie algebra is the root space $\mathfrak{g}_{\gamma}$. Let $x_{\gamma}: \mathbb{C} \longrightarrow X_{\gamma}$ be the group isomorphism defined via the exponential map and the choice of a Chevalley basis element in $\mathfrak{g}_{\gamma}$. For each positive root $\gamma$, choose a nonzero vector $E_{\gamma} \in \mathfrak{g}_{\gamma}$. Let $E_{-\gamma}$ be the unique vector in $\mathfrak{g}_{-\gamma}$ so that $\left\{E_{\gamma}, \left[E_{\gamma}, E_{-\gamma}\right], E_{-\gamma}\right\}$ is an $\mathfrak{sl}_2$-triple. Let $H_{\gamma}= \left[E_{\gamma}, E_{-\gamma}\right]$. The four lemmas from Lemma~\ref{flowlem} to Lemma~\ref{prodlem} are still true for type $F_4$ and $E_6$.

In our presentation of $G$ above, we have chosen a Borel subgroup $B$ and a maximal torus $T \subset B$ to begin with. Therefore, for each distinguished nilpotent orbit, we need to pick a representative $N$ so that its associated parabolic $P$ contains the preselected $B$. To facilitate the choice, we recall some information regarding the classification of nilpotent orbits of a simple Lie algebra.

There are at least three different ways of classifying nilpotent orbits: weighted Dynkin diagrams, the Bala-Carter theory and pseudo-Levi subalgebras.

\noindent $\bullet$ weighted Dynkin diagrams   

By the Jacobson-Morozov theorem, every nilpotent element $N \in \mathfrak{g}$ can be embedded in a $\mathfrak{sl}_2$-triple. By considering $\mathfrak{g}$ as an $\mathfrak{sl}_2$-representation, we thereby obtain the weighted Dynkin diagram associated to $N$, which is the Dynkin diagram of $\mathfrak{g}$ with every node labeled by a number from $\{0, 1, 2  \}$. It is known that the weighted Dynkin diagram is uniquely determined by the conjugacy class of $N$, and that the diagram is even (labeled only by $0$ and $2$) if $N$ is distinguished. The main drawback of this method is that not all of the $3^l$ weighted Dynkin diagrams of an algebra $\mathfrak{g}$ of rank $l$ correspond to nilpotent orbits, and there is no simple algorithm to determine which diagram does. For type $F_4$ and $E_6$, the weighted Dynkin diagrams of all their distinguished nilpotent orbits are listed in \cite{bala1976classes}[p.~416].

\noindent $\bullet$ Bala-Carter theory 

The Bala-Carter theorem \cite{bala1976classesii}[Theorem 6.1] states that there is a bijection between the set of conjugacy classes of nilpotent elements of $\mathfrak{g}$ and $G$-conjugacy classes of pairs $(R, P_R)$, where $R$ is a semisimple subgroup of parabolic type in $G$ and $P_R$ is a distinguished parabolic subgroup of $R$. As a consequence, (as we have recalled in subsection \ref{undist}) for every nilpotent element $N \in \mathfrak{g}$, there exists a minimal Levi subalgebra $\mathfrak{m}$ of $\mathfrak{g}$ containing $N$ so that $N$ is distinguished in $\mathfrak{m}$. Under this scheme, the distinguished nilpotent orbits of type $F_4$ and $E_6$ are denoted by symbols $X_i(a_j)$, where $X_i$ is either $F_4$ or $E_6$ (meaning the minimal Levi subalgebra is $\mathfrak{g}$ itself) and $j$ is the semisimple rank of the associated distinguished parabolic subgroup. According to \cite{bala1976classes}[p.~416], $F_4$ has 4 distinguished nilpotent orbits: $F_4$ (the regular orbit), $F_4(a_1)$, $F_4(a_2)$ and $F_4(a_3)$. $E_6$ has 3 distinguished orbits: $E_6$, $E_6(a_1)$ and $E_6(a_3)$.

\noindent $\bullet$ pseudo-Levi subalgebras

\begin{definition} \label{dfnpseudoLevi}
A pseudo-Levi subalgebra of $\mathfrak{g}$ is a subalgebra that is $G$-conjugate to a subalgebra of the form 
\[ \mathfrak{g}'=\mathfrak{t} \oplus \bigoplus_{\alpha \in \Psi} \mathfrak{g}_{\alpha},  \]
where $\Psi$ is an additively closed subrootsystem of $\Phi$.
\end{definition}

Pseudo-Levi subalgebras turn out to be rather tractable. 

In the first place, we have the following result.

\begin{proposition}[{\cite{sommers1998generalization}}, Proposition 2] \label{sommersprop}
Pseudo-Levi subalgebras are the subalgebras of $\mathfrak{g}$ of the form $C_{\mathfrak{g}}(t)$ where $t$ is a semisimple element of $G$.
\end{proposition}

In the second place, the Dynkin diagrams of all the additively closed subrootsystems $\Psi$ can be obtained from the extended Dynkin diagram $\tilde{\Phi}$ of $\mathfrak{g}$ by the Borel-de Siebenthal theory: start with $\tilde{\Phi}$, remove certain nodes, and possibly repeat the same process on the connected components of the resulting diagram. In particular, the Dynkin diagrams of the pseudo-Levi subalgebras of the same rank as $\mathfrak{g}$ are those obtained from $\tilde{\Phi}$ by removing exactly one node. 

In the third place, for type $A_n$, $C_n$, $G_2$, $F_4$ and $E_6$, the $G$-conjugacy class of the pseudo-Levi subalgebra $\mathfrak{g}'$ is determined by the isomorphism type of its root system $\Psi$ and the lengths of the simple roots of $\Psi$.

These facts, together with the following theorem, can help us find a representative for each distinguished nilpotent orbit of type $F_4$ and $E_6$.

\begin{theorem}[{\cite{sommers1998generalization}}, Theorem 13] \label{sommersthm}
When $G$ is of the adjoint type, there is a bijection between: $G$-conjugacy classes of pairs $(\mathfrak{l}, N)$, where $\mathfrak{l}$ is a pseudo-Levi subalgebra of $\mathfrak{g}$ and $N$ is a distinguished nilpotent element in $\mathfrak{l}$; and $G$-conjugacy classes of pairs $(N, C)$, where $N$ is a nilpotent element of $\mathfrak{g}$ and $C$ is a conjugacy class in the component group $A(N)=C_G(N)/C_G^{\circ}(N)$.
\end{theorem}

Furthermore, to make sure that the associated parabolic subgroup of the chosen representative does contain the preselected Borel subgroup $B$, we need the following lemma.

\begin{lemma} \label{reprlem}
Let $\mathfrak{g}$ be a reductive Lie algebra over $\mathbb{C}$ of semisimple rank $l$ with a decomposition
\[ \mathfrak{g}=\mathfrak{h} \oplus \bigoplus_{\alpha \in \Phi} \mathfrak{g}_{\alpha},  \]
where $\mathfrak{h}$ is a Cartan subalgebra and $\Phi$ is the root system of $\mathfrak{g}$. Choose a set of simple roots $\Delta=\{\alpha_1, \alpha_2, \ldots, \alpha_l \}$. For each positive root $\gamma \in \Phi$, choose a nonzero vector $E_{\gamma} \in \mathfrak{g}_{\gamma}$. Let $E_{-\gamma}$ be the unique vector in $\mathfrak{g}_{-\gamma}$ so that $\left\{E_{\gamma}, \left[E_{\gamma}, E_{-\gamma}\right], E_{-\gamma}\right\}$ is an $\mathfrak{sl}_2$-triple. Let $H_{\gamma}= \left[E_{\gamma}, E_{-\gamma}\right]$. Then $N=\sum_{i=1}^l E_{\alpha_i}$ is a regular nilpotent element of $\mathfrak{g}$ and there exist a semisimple element $H \in \mathfrak{h}$ and a nilpotent element $Y \in \mathfrak{g}$ so that $\{ N, H, Y \}$ forms an $\mathfrak{sl}_2$-triple. 
\end{lemma}

\begin{proof}
By the Jacobson-Morozov theorem, we can always embed $N$ in an $\mathfrak{sl}_2$-triple. The point of the lemma is to show that there exists such an $\mathfrak{sl}_2$-triple so that the semisimple element $H$ lies in the preselected Cartan subalgebra $\mathfrak{h}$.

Let $A=[A_{ij}]_{l \times l}$ be the Cartan matrix of $\Phi$, where $A_{ij}=\langle \alpha_j, \alpha_i \rangle=\alpha_j(H_{\alpha_i})$. Define the $1 \times l$ row vector $[x_1, x_2, \ldots, x_l]=[2, 2, \ldots, 2]A^{-1}$. Set $H=\sum_{i=1}^l x_i H_{\alpha_i}$ and $Y=\sum_{i=1}^l x_i E_{-\alpha_i}$. 

For every $j \in \{1, 2, \ldots, l \}$, $\alpha_j(H)=\sum_{i=1}^l x_i \alpha_j(H_{\alpha_i})=\sum_{i=1}^l x_i A_{ij}=2$. Therefore, $[H, N]=\sum_{i=1}^l [H, E_{\alpha_i}]=\sum_{i=1}^l \alpha_i(H) E_{\alpha_i}=2N$. Similarly, $[H, Y]=-2Y$. For distinct simple roots $\alpha_i$ and $\alpha_j$, $\alpha_i - \alpha_j$ is not a root. Hence $[N, Y]=\sum_{i=1}^l x_i [E_{\alpha_i}, E_{-\alpha_i}]=\sum_{i=1}^l x_i H_{\alpha_i}=H$. Therefore, $\{ N, H, Y \}$ is an $\mathfrak{sl}_2$-triple. This also means that the weighted Dynkin diagram associated to $N$ has every node labeled by 2, so $N$ must be a regular nilpotent element.
\end{proof}

As will become clear later, the following lemma is the last piece of the puzzle to prove Proposition~\ref{F4E6}.

\begin{lemma} \label{surfacelem}
For every smooth projective rational surface over $\mathbb{C}$ with finitely many marked points, there exists an affine paving of it whose unique 2-cell contains all the marked points. In particular, every smooth projective rational surface is paved by affines.
\end{lemma}

\begin{proof}
It is well-known that any smooth projective rational surface $S$ can be obtained from successive blowups of a minimal rational surface. To be precise, there exists a chain of morphisms $\pi_k: S_k \longrightarrow S_{k-1}$ ($k=1, 2, \ldots, n$) so that: each $S_k$ is a smooth projective rational surface, $S_n \cong S$ and $S_0$ is a minimal rational surface; each $\pi_k$ is the blowup of $S_{k-1}$ at a single point. We prove the lemma by induction on the number of blowups in the chain. 

When the number of blowups is 0, the surface $S$ itself is a minimal rational surface. That is, $S$ is either $\mathbb{P}^2$ or a Hirzebruch surface. Then the lemma is clear from the explicit descriptions of these two types of minimal surfaces.

Assume the lemma is true when the number of blowups is less than or equal to $n-1$. For any surface $S$ obtained by $n$ blowups, consider the last morphism $\pi_n: S \longrightarrow S_{n-1}$. Let $\{q_1, q_2, \ldots, q_m \}$ be the set of marked points on $S$ and let $p \in S_{n-1}$ be the blown-up point. Consider $S_{n-1}$ with a set of marked points $\{ \pi_n(q_1), \pi_n(q_2), \ldots, \pi_n(q_m) \} \cup \{p\}$. By the inductive hypothesis, there exists an affine paving of $S_{n-1}$ whose unique 2-cell $\mathbb{A}^2$ contains $\{ \pi_n(q_1), \pi_n(q_2), \ldots, \pi_n(q_m) \} \cup \{p\}$. Then the blowup $\operatorname{Bl}_p \mathbb{A}^2$, as an open subset of $S$, contains all the points $q_1, q_2, \ldots, q_m$. The classic definition of the blowup of $\mathbb{C}^2$ at the origin is $\operatorname{Bl}_{(0, 0)}\mathbb{C}^2=\{( (x,y), [u:v] ) \in \mathbb{C}^2 \times \mathbb{P}^1 \ |\ xv-yu=0 \}$. Hence $\operatorname{Bl}_p \mathbb{A}^2$ is a line bundle over the exceptional divisor $\mathbb{P}^1$, and we use $\operatorname{pr}_2: \operatorname{Bl}_p \mathbb{A}^2 \longrightarrow \mathbb{P}^1$ to denote the projection. Picking $x \in \mathbb{P}^1$ which is not equal to any $\operatorname{pr}_2(q_i)$ ($i=1, 2, \ldots, m$), we can decompose $\operatorname{Bl}_p \mathbb{A}^2$ into a 1-cell $\operatorname{pr}_2^{-1}(x) \cong \mathbb{A}^1$ and a 2-cell $\operatorname{pr}_2^{-1}(\mathbb{P}^1 \setminus \{x\}) \cong \mathbb{A}^2$. It is clear that $\operatorname{pr}_2^{-1}(\mathbb{P}^1 \setminus \{x\})$, $\operatorname{pr}_2^{-1}(x)$ and the inverse images under $\pi_n$ of the 1-cells and 0-cells of $S_{n-1}$ form an affine paving of $S$, and that $q_1, q_2, \ldots, q_m$ all lie in the unique 2-cell $\operatorname{pr}_2^{-1}(\mathbb{P}^1 \setminus \{x\}) \cong \mathbb{A}^2$. The induction step is now complete and we have finished the proof.
\end{proof}

\begin{remark}
De Concini and Maffei \cite{de2022paving} used this lemma to show that Springer fibers for $G$ of type $E_7$ are paved by affines.
\end{remark}

\subsection{The case of $F_4$}
We now prove Proposition~\ref{F4E6} when the group $G$ is of type $F_4$. We use the presentation of $G$ outlined in the beginning of subsection \ref{compsetup}. That is, we fix a Borel subgroup $B$ and a maximal torus $T \subset B$, by which we obtain the decomposition
\[ \mathfrak{g}=\mathfrak{t} \oplus \bigoplus_{\alpha \in \Phi} \mathfrak{g}_{\alpha},  \]
a 1-dimensional unipotent subgroup $X_{\gamma}$ for each root $\gamma \in \Phi$, and an $\mathfrak{sl}_2$-triple $\{ E_{\gamma}, H_{\gamma}, E_{-\gamma} \}$ for each positive root $\gamma \in \Phi^+$.

As mentioned earlier, $F_4$ has 4 distinguished nilpotent orbits: $F_4$, $F_4(a_1)$, $F_4(a_2)$ and $F_4(a_3)$. Note that $(L, L \cap B, \mathfrak{g}(2), U, N)$ is a closed subvariety of the flag variety $L/L \cap B$. For the regular nilpotent orbit $F_4$, $L/L \cap B$ is a single point, so the proposition is automatically true. For $F_4(a_1)$, since the semisimple rank of the associated parabolic subgroup of a representative is 1, $L / L \cap B$ is $\mathbb{P}^1$ and there is nothing more to prove either. For $F_4(a_3)$, the nonempty $(L, L \cap B, \mathfrak{g}(2), U, N)$'s are described concretely in \cite{de1988homology}[section 4.2]. In the light of Lemma~\ref{surfacelem}, they are clearly all paved by affines. Therefore, the only orbit that needs to be dealt with is $F_4(a_2)$. We start the computation by finding a good representative of this orbit. 

We give the Dynkin diagram of $F_4$ the Bourbaki labeling.
\[ F_4 :\quad \begin{dynkinDiagram}[root radius=.12cm, edge length=1.0cm]{F}{4}
\dynkinLabelRoots{ \alpha_1,\alpha_2,\alpha_3, \alpha_4}
\end{dynkinDiagram}  \]

Adding the lowest root $\alpha_0=-(2\alpha_1 + 3\alpha_2 + 4\alpha_3 + 2\alpha_4)$, we have the extended Dynkin diagram of $F_4$.
\[ F_4 :\quad \begin{dynkinDiagram}[extended, root radius=.12cm, edge length=1.0cm]{F}{4}
\dynkinLabelRoots{\alpha_0, \alpha_1,\alpha_2,\alpha_3, \alpha_4}
\end{dynkinDiagram}  \]

From \cite{bala1976classes}[p.~417], we obtain the weighted Dynkin diagram corresponding to $F_4(a_2)$.
\[ F_4 :\quad \begin{dynkinDiagram}[root radius=.12cm, edge length=1.0cm]{F}{4}
\dynkinLabelRoots{0, 2, 0, 2}
\end{dynkinDiagram}  \]

We expect a good representative $N$ of $F_4(a_2)$ to satisfy the following three requirements:
\begin{itemize}
\item[(1)] $N$ can be embedded into an $\mathfrak{sl}_2$-triple $\{N, H, Y\}$ so that $H \in \mathfrak{t}$ and $\gamma(H) \ge 0$ for every positive root $\gamma \in \Phi^+$. As a consequence, the associated parabolic subgroup $P$ of $N$ contains the preselected Borel subgroup $B$.
\item[(2)] Let $\mathfrak{g}=\bigoplus_{i \in \mathbb{Z}} \mathfrak{g}(i)$ be the weight space decomposition of $\mathfrak{g}$ via the $\mathfrak{sl}_2$-triple  $\{N, H, Y\}$. We need to have 
$\mathfrak{g}(0)=\mathfrak{t} \oplus \operatorname{span}_{\mathbb{C}} \{E_{\alpha_1}, E_{-\alpha_1}, E_{\alpha_3}, E_{-\alpha_3}\}$ and 
$\mathfrak{g}(2)=\operatorname{span}_{\mathbb{C}} \{E_{\alpha_2}, E_{\alpha_4}, E_{\alpha_{1}+\alpha_{2}}, \linebreak E_{\alpha_{2}+\alpha_{3}}, E_{\alpha_{3}+\alpha_{4}}, E_{\alpha_{1}+\alpha_{2}+\alpha_{3}}, E_{\alpha_{2}+2\alpha_{3}}, E_{\alpha_{1}+\alpha_{2}+2\alpha_{3}}\}$. This makes sure the weighted Dynkin diagram associated to $N$ is indeed the one given above.
\item[(3)] $N \in \mathfrak{g}(2)$. This is always a consequence of (1).
\end{itemize}
The idea of finding such an $N$ comes from Theorem~\ref{sommersthm}. In particular, the first table in \cite{sommers1998generalization}[p.~557] tells us that a representative of $F_4(a_2)$ can be a regular nilpotent element of a pseudo-Levi subalgebra of type $A_1 + C_3$. From the eight roots belonging to $\mathfrak{g}(2)$, we pick 4 and rename them as $\beta_0=\alpha_{1}+\alpha_{2}+2\alpha_{3}$, $\beta_2=\alpha_{1}+\alpha_{2}$, $\beta_3=\alpha_{3}+\alpha_{4}$ and $\beta_4=\alpha_{2}+\alpha_{3}$. We can draw a diagram by treating these 4 roots as nodes and connecting $\beta_i$ and $\beta_j$ by $\langle \beta_i, \beta_j \rangle \langle \beta_j, \beta_i \rangle$ bonds ($i, j \in \{0, 2, 3, 4\}$), with arrows pointing from long roots to short roots. The resulting diagram happens to be the same as the Dynkin diagram of type $A_1 + C_3$.
\[ A_1 + C_3: \quad \begin{dynkinDiagram}[root radius=.12cm, edge length=1.0cm]{A}{1}
\dynkinLabelRoots{\beta_0}
\end{dynkinDiagram}  \quad \quad
\begin{dynkinDiagram}[root radius=.12cm, edge length=1.0cm]{C}{3}
\dynkinLabelRoots{\beta_4, \beta_3, \beta_2}
\end{dynkinDiagram}  \]
Let $\Psi$ denote the additively closed subrootsystem of $\Phi$ generated by $\{\beta_0, \beta_2, \beta_3, \beta_4\}$. We claim that $\Psi$ is root system of type $A_1 + C_3$ and $\mathfrak{g}'=\mathfrak{t} \oplus \bigoplus_{\alpha \in \Psi}\mathfrak{g}_{\alpha}$ is a pseudo-Levi subalgebra of the same type. To see this, note that there is only one $G$-conjugacy class of pseudo-Levi subalgebras of type $A_1 + C_3$, whose ``standard'' Dynkin diagram is obtained from the extended Dynkin diagram of $F_4$ by removing the node $\alpha_1$ (the Borel-de Siebenthal theory).
\[ A_1 + C_3: \quad \begin{dynkinDiagram}[root radius=.12cm, edge length=1.0cm]{A}{1}
\dynkinLabelRoots{\alpha_0}
\end{dynkinDiagram}  \quad \quad
\begin{dynkinDiagram}[root radius=.12cm, edge length=1.0cm]{C}{3}
\dynkinLabelRoots{\alpha_4, \alpha_3, \alpha_2}
\end{dynkinDiagram}   \]
The removed node $\alpha_1$ is related to the rest by $\alpha_1=(1/2)(-\alpha_0-3\alpha_{2}-4\alpha_{3}-2\alpha_{4})$. Define $\beta_1$ to be the same linear combination of $\beta_0$, $\beta_2$, $\beta_3$ and $\beta_4$: $\beta_1=(1/2)(-\beta_0-3\beta_{2}-4\beta_{3}-2\beta_{4})=-2\alpha_{1}-3\alpha_{2}-4\alpha_{3}-2\alpha_{4}=\alpha_0$. Now that $\beta_1$, $\beta_2$, $\beta_3$, $\beta_4$ are 4 distinct roots of $\Phi$ and the diagram formed by them is exactly the Dynkin diagram of $F_4$,
\[ F_4 :\quad \begin{dynkinDiagram}[root radius=.12cm, edge length=1.0cm]{F}{4}
\dynkinLabelRoots{ \beta_1,\beta_2,\beta_3, \beta_4}
\end{dynkinDiagram}  \]
$\{\beta_1, \beta_2, \beta_4, \beta_4\}$ must be another set of simple roots of $\Phi$ with $\beta_0$ being the corresponding lowest root. Since $\Psi$ is additively generated by $\{\beta_0, \beta_2, \beta_3, \beta_4\}$ (the new set of nodes with $\beta_1$ removed), it must be a subrootsystem of $\Phi$ of type $A_1 + C_3$. Hence $\mathfrak{g}'=\mathfrak{t} \oplus \bigoplus_{\alpha \in \Psi}\mathfrak{g}_{\alpha}$ is a pseudo-Levi subalgebra of the same type. Let $N=E_{\beta_0}+E_{\beta_2}+E_{\beta_3}+E_{\beta_4}=E_{\alpha_{1}+\alpha_{2}+2\alpha_{3}}+E_{\alpha_{1}+\alpha_{2}}+E_{\alpha_{3}+\alpha_{4}}+E_{\alpha_{2}+\alpha_{3}}$. By Lemma~\ref{reprlem}, $N$ is a regular nilpotent element of $\mathfrak{g}'$ and there exists a semisimple element $H \in \mathfrak{t}$ and a nilpotent element $Y \in \mathfrak{g}'$ so that $\{ N, H, Y \}$ forms an $\mathfrak{sl}_2$-triple. As a consequence, $\beta_i(H)=2$ for $i=0, 2, 3, 4$. Solving the linear system of equations,
\[     \left\{
      \begin{array}{llll}
                  \beta_0(H)=\alpha_1(H)+\alpha_2(H)+2\alpha_3(H)=2 \\
                  \beta_2(H)=\alpha_1(H)+\alpha_2(H)=2 \\
                  \beta_3(H)=\alpha_3(H)+\alpha_4(H)=2 \\
		  \beta_4(H)=\alpha_2(H)+\alpha_3(H)=2
                \end{array}
              \right.  \]
we get $\alpha_1(H)=0$, $\alpha_2(H)=2$, $\alpha_3(H)=0$ and $\alpha_4(H)=2$. This means that the weighted Dynkin diagram associated to $N$ is the one we started with.
\[ F_4: \quad \begin{dynkinDiagram}[root radius=.12cm, edge length=1.0cm]{F}{4}
\dynkinLabelRoots{0, 2, 0, 2}
\end{dynkinDiagram}  \]
Therefore, $N=E_{\alpha_{1}+\alpha_{2}+2\alpha_{3}}+E_{\alpha_{1}+\alpha_{2}}+E_{\alpha_{3}+\alpha_{4}}+E_{\alpha_{2}+\alpha_{3}}$ is indeed a representative of $F_4(a_2)$ and does satisfy all 3 requirements mentioned earlier.

Let $P$ be the associated parabolic of $N$ and $P=LU_P$ be the Levi decomposition. We compute the nonempty $(L, L \cap B, \mathfrak{g}(2), U, N)$'s and show that they are all paved by affines. 

In this case, $L=\langle T, X_{\alpha_1}, X_{-\alpha_1}, X_{\alpha_3}, X_{-\alpha_3} \rangle$ and $L \cap B = \langle T, X_{\alpha_1}, X_{\alpha_3} \rangle$, where the angle brackets stand for group generation. As a consequence, $L/ L \cap B \cong \mathbb{P}^1 \times \mathbb{P}^1$. By Lemma~\ref{prehsp}, for any $(L, L \cap B, \mathfrak{g}(2), U, N)$ to be nonempty, $U$ can have codimension at most 2 in $\mathfrak{g}(2)$. When $\operatorname{codim}U$ is 0 or 2, the corresponding $(L, L \cap B, \mathfrak{g}(2), U, N)$ is respectively $\mathbb{P}^1 \times \mathbb{P}^1$ or a finite set of points, hence paved by affines. The only nontrivial case is when $\operatorname{codim}U=1$. There are only two such $(L \cap B)$-stable subspace $U$ of $\mathfrak{g}(2)$: 
\[ U=\bigoplus_{\alpha \in \Phi(U)} \mathfrak{g}_{\alpha}, \quad \Phi(U)=\Phi(\mathfrak{g}(2)) \setminus \{\alpha_2\}, \text{ or }  \]
\[ U=\bigoplus_{\alpha \in \Phi(U)} \mathfrak{g}_{\alpha}, \quad \Phi(U)=\Phi(\mathfrak{g}(2)) \setminus \{\alpha_4\}, \]
where $\Phi(U)$ and $\Phi(\mathfrak{g}(2))$ denote the sets of roots belonging to $U$ and $\mathfrak{g}(2)$ respectively. Recall that $\Phi(\mathfrak{g}(2))=\{\alpha_2, \alpha_4, \alpha_{1}+\alpha_{2}, \alpha_{2}+\alpha_{3}, \alpha_{3}+\alpha_{4}, \alpha_{1}+\alpha_{2}+\alpha_{3}, \alpha_{2}+2\alpha_{3}, \alpha_{1}+\alpha_{2}+2\alpha_{3}\}$, hence we can only remove $\alpha_2$ or $\alpha_4$ in order for $U$ to be an $(L \cap B)$-stable subspace of codimension 1.


\noindent $\bullet$ the case of $\Phi(U)=\Phi(\mathfrak{g}(2)) \setminus \{\alpha_2\}$

Let $s_1$ and $s_3$ be the simple reflections associated to $\alpha_1$ and $\alpha_3$ respectively. Let $x_{\alpha_1}: \mathbb{C} \longrightarrow X_{\alpha_1}$ and $x_{\alpha_3}: \mathbb{C} \longrightarrow  X_{\alpha_3}$ be the group isomorphisms mentioned in subsection \ref{compsetup}. Let $\mathbb{C} \times \mathbb{C}$ denote the unique 2-cell of $L/ L \cap B \cong \mathbb{P}^1 \times \mathbb{P}^1$, and we can present it as the following set:
\[ \mathbb{C} \times \mathbb{C}=\{x_{\alpha_1}(z_1)\dot{s}_1x_{\alpha_3}(z_3)\dot{s}_3(L \cap B) \in L/L \cap B \ |\ z_1, z_3 \in \mathbb{C}\}.  \]
Here $x_{\alpha_1}(z_1)\dot{s}_1x_{\alpha_3}(z_3)\dot{s}_3(L \cap B)$ is considered a left $(L \cap B)$-coset, hence a point of the flag variety $L/ L \cap B$ (for details of this presentation, see Equation~\ref{orbitfixedpoint} and Lemma~\ref{dimlem} (2)). 

Consider the intersection of $X_U=(L, L \cap B, \mathfrak{g}(2), U, N)$ with the 2-cell $\mathbb{C} \times \mathbb{C}$ of $L/L \cap B$. By definition, 
\[ X_U \cap (\mathbb{C} \times \mathbb{C}) \cong \{ (z_1, z_3) \in \mathbb{C}^2  \ |\ \dot{s}_3 \cdot x_{\alpha_3}(-z_3) \cdot \dot{s}_1 \cdot x_{\alpha_1}(-z_1) \cdot N \in U \}. \]
By Lemma~\ref{flowlem} and Lemma~\ref{permlem}, we know that 
\[ \dot{s}_3 \cdot x_{\alpha_3}(-z_3) \cdot \dot{s}_1 \cdot x_{\alpha_1}(-z_1) \cdot N=\sum_{\gamma \in \Phi(\mathfrak{g}(2))} f_{\gamma}(z_1, z_3) E_{\gamma}, \text{ where } f_{\gamma} \in \mathbb{C}[z_1, z_3].  \]
Since $\Phi(U)=\Phi(\mathfrak{g}(2)) \setminus \{\alpha_2\}$, in order for $\dot{s}_3 \cdot x_{\alpha_3}(-z_3) \cdot \dot{s}_1 \cdot x_{\alpha_1}(-z_1) \cdot N$ to be in $U$, it is necessary and sufficient that $f_{\alpha_2}(z_1, z_3)=0$. Using the same two lemmas and computing carefully with the representative $N=E_{\alpha_{1}+\alpha_{2}+2\alpha_{3}}+E_{\alpha_{1}+\alpha_{2}}+E_{\alpha_{3}+\alpha_{4}}+E_{\alpha_{2}+\alpha_{3}}$, we get that 
\[ f_{\alpha_2}(z_1, z_3)=a z_3^2+b z_1z_3+1,  \]
where $a$ and $b$ are some nonzero constants. Therefore,
\[ X_U \cap (\mathbb{C} \times \mathbb{C}) \cong \{ (z_1, z_3) \in \mathbb{C}^2 \ |\ a z_3^2+b z_1z_3+1=0 \}.  \]
There is an isomorphism between $\mathbb{C}^{\times}$ and $X_U \cap (\mathbb{C} \times \mathbb{C})$:
\begin{displaymath} 
\begin{array}{llll}
\mathbb{C}^{\times} & \longrightarrow & X_U \cap (\mathbb{C} \times \mathbb{C}) \\
z &  \longmapsto & (-\frac{a}{b}z-\frac{1}{b}z^{-1},z)
\end{array}
\end{displaymath}

Note that $\overline{X_U \cap (\mathbb{C} \times \mathbb{C})}$ is a connected component of $X_U$. By Lemma~\ref{prehsp} and the isomorphism above, $\overline{X_U \cap (\mathbb{C} \times \mathbb{C})}$ is a smooth projective rational curve, that is, $\mathbb{P}^1$. As a result, $X_U$ is either $\mathbb{P}^1$ or a disjoint union thereof, hence paved by affines.


\noindent $\bullet$ the case of $\Phi(U)=\Phi(\mathfrak{g}(2)) \setminus \{\alpha_4\}$

Let $s_1$, $s_3$, $x_{\alpha_1}$, $x_{\alpha_3}$, $X_{\alpha_1}$ and $X_{\alpha_3}$ be the same as above. Let $\mathbb{C} \times \{\infty\}$ denote the 1-cell $X_{\alpha_1}\dot{s}_1 (L \cap B)$ of $L/ L \cap B \cong \mathbb{P}^1 \times \mathbb{P}^1$, which we present as:
\[ \mathbb{C} \times \{\infty\}=\{x_{\alpha_1}(z_1)\dot{s}_1(L \cap B) \in L/L \cap B \ |\ z_1\in \mathbb{C}\}.  \]
Consider the intersection $X_U \cap (\mathbb{C} \times \{\infty\}) \cong \{ z_1 \in \mathbb{C}  \ |\ \dot{s}_1 \cdot x_{\alpha_1}(-z_1) \cdot N \in U \}$. Using a similar method as above, it is easy to show that
\[ \dot{s}_1 \cdot x_{\alpha_1}(-z_1) \cdot N=\sum_{\gamma \in \Phi(\mathfrak{g}(2))} f_{\gamma}(z_1) E_{\gamma} \text{ and } f_{\alpha_4}(z_1)=0. \]
As a result, $ X_U \cap (\mathbb{C} \times \{\infty\})= \mathbb{C} \times \{\infty\}$ and $X_U$ is still either $\mathbb{P}^1$ or a disjoint union thereof, hence paved by affines.

We have finished the case of $F_4$.

\subsection{The case of $E_6$}
We prove Proposition~\ref{F4E6} when $G$ is of type $E_6$. In this case, the increased dimension of $L / L \cap B$ makes the computation more complicated, but the idea is exactly the same. Therefore, we omit the details of those arguments that have exact counterparts in the case of $F_4$.

$E_6$ has 3 distinguished nilpotent orbits: $E_6$, $E_6(a_1)$ and $E_6(a_3)$. The flag varieties $L / L \cap B$ for the orbits $E_6$ and $E_6(a_1)$ are $\mathbb{P}^0$ and $\mathbb{P}^1$ respectively, so there is nothing to prove. We only need to deal with $E_6(a_3)$. Similar to type $F_4$, we start by finding a good representative of the orbit $E_6(a_3)$.

We give the Dynkin diagram of $E_6$ the following nonstandard labeling (neither Bourbaki nor GAP).
\[ E_6 :\quad \begin{dynkinDiagram}[root radius=.12cm, edge length=1.0cm]{E}{6}
\dynkinLabelRoots{\alpha_1, \alpha_6, \alpha_2, \alpha_3, \alpha_4, \alpha_5}
\end{dynkinDiagram}  \]

Adding the lowest root $\alpha_0=-(\alpha_{1}+2\alpha_{2}+3\alpha_{3}+2\alpha_{4}+\alpha_{5}+2\alpha_{6} )$, we have the extended Dynkin diagram of $E_6$.
\[ E_6 :\quad \begin{dynkinDiagram}[extended, root radius=.12cm, edge length=1.0cm]{E}{6}
\dynkinLabelRoots{\alpha_0, \alpha_1, \alpha_6, \alpha_2, \alpha_3, \alpha_4, \alpha_5}
\end{dynkinDiagram}  \]

From \cite{bala1976classes}[p.~417], we obtain the weighted Dynkin diagram corresponding to $E_6(a_3)$.
\[ E_6 :\quad \begin{dynkinDiagram}[root radius=.12cm, edge length=1.0cm]{E}{6}
\dynkinLabelRoots{2, 0, 0, 2, 0, 2}
\end{dynkinDiagram}  \]

The reason for choosing the nonstandard labeling is the notational symmetry that odd-numbered simple roots $\alpha_1$, $\alpha_3$ and $\alpha_5$ are of weight 2 and even-numbered simple roots $\alpha_2$, $\alpha_4$ and $\alpha_6$ are of weight 0 in the weighted Dynkin diagram above.

We expect a good representative $N$ of $E_6(a_3)$ to satisfy the following three requirements:
\begin{itemize}
\item[(1)] $N$ can be embedded into an $\mathfrak{sl}_2$-triple $\{N, H, Y\}$ so that $H \in \mathfrak{t}$ and $\gamma(H) \ge 0$ for every positive root $\alpha \in \Phi$.
\item[(2)] Let $\mathfrak{g}=\bigoplus_{i \in \mathbb{Z}} \mathfrak{g}(i)$ be the weight space decomposition of $\mathfrak{g}$ via the $\mathfrak{sl}_2$-triple  $\{N, H, Y\}$. We need $\mathfrak{g}(0)=\mathfrak{t} \oplus \operatorname{span}_{\mathbb{C}}\{E_{\alpha_2}, E_{-\alpha_2}, E_{\alpha_4}, E_{-\alpha_4}, E_{\alpha_6}, E_{-\alpha_6}\}$ and $\mathfrak{g}(2)=\operatorname{span}_{\mathbb{C}}\{E_{\alpha_1}, E_{\alpha_3}, E_{\alpha_5}, \linebreak E_{\alpha_{1}+\alpha_{2}}, E_{\alpha_{2}+\alpha_{3} }, E_{\alpha_{3}+\alpha_{4} }, E_{\alpha_{3}+\alpha_{6} }, E_{\alpha_{4}+\alpha_{5} }, E_{\alpha_{2}+\alpha_{3}+\alpha_{4} }, E_{\alpha_{2}+\alpha_{3}+\alpha_{6} }, E_{\alpha_{3}+\alpha_{4}+\alpha_{6} }, E_{\alpha_{2}+\alpha_{3}+\alpha_{4}+\alpha_{6} }\}$.
\item[(3)] $N \in \mathfrak{g}(2)$.
\end{itemize}
The second table in \cite{sommers1998generalization}[p.~557] tells us that a representative of $E_6(a_3)$ can be a regular nilpotent element of a pseudo-Levi subalgebra of type $A_5 + A_1$, whose ``standard'' Dynkin diagram can be obtained from the extended Dynkin diagram of $E_6$ by removing the node $\alpha_2$.
\[ A_5 + A_1: \quad \begin{dynkinDiagram}[root radius=.12cm, edge length=1.0cm]{A}{5}
\dynkinLabelRoots{\alpha_0, \alpha_6, \alpha_3, \alpha_4, \alpha_5}
\end{dynkinDiagram}  \quad \quad
\begin{dynkinDiagram}[root radius=.12cm, edge length=1.0cm]{A}{1}
\dynkinLabelRoots{\alpha_1}
\end{dynkinDiagram}  \]
From the twelve roots belonging to $\mathfrak{g}(2)$, we pick 6 and rename them as $\beta_0=\alpha_{2}+\alpha_{3} $, $\beta_6=\alpha_{4}+\alpha_{5} $, $\beta_3=\alpha_{3}+\alpha_{6} $, $\beta_4=\alpha_{1}+\alpha_{2} $, $\beta_5=\alpha_{3}+\alpha_{4} $ and $\beta_1=\alpha_{2}+\alpha_{3}+\alpha_{4}+\alpha_{6} $. Draw a diagram by treating the 6 roots as nodes and connecting $\beta_i$ and $\beta_j$ by $\langle \beta_i, \beta_j \rangle \langle \beta_j, \beta_i \rangle$ bonds ($i, j \in \{0, 1, 3, 4, 5, 6\}$), we get exactly the Dynkin diagram of type $A_5 + A_1$.
\[ A_5 + A_1: \quad \begin{dynkinDiagram}[root radius=.12cm, edge length=1.0cm]{A}{5}
\dynkinLabelRoots{\beta_0, \beta_6, \beta_3, \beta_4, \beta_5}
\end{dynkinDiagram}  \quad \quad
\begin{dynkinDiagram}[root radius=.12cm, edge length=1.0cm]{A}{1}
\dynkinLabelRoots{\beta_1}
\end{dynkinDiagram}  \]
Let $N=E_{\alpha_{2}+\alpha_{3}} +E_{\alpha_{4}+\alpha_{5} } +E_{\alpha_{3}+\alpha_{6}} +E_{\alpha_{1}+\alpha_{2} } +E_{\alpha_{3}+\alpha_{4} } +E_{\alpha_{2}+\alpha_{3}+\alpha_{4}+\alpha_{6} }$. Arguing in exactly the same way as for type $F_4$, we can show that $N$ is a representative of $E_6(a_3)$ that satisfies all 3 requirements mentioned earlier.

Let $P$ be the associated parabolic of $N$ and $P=LU_P$ be the Levi decomposition. We compute all the $(L, L \cap B, \mathfrak{g}(2), U, N)$'s that are possibly nonempty for dimension reason and show that they are paved by affines.

In this case, $L=\langle T, X_{\alpha_2}, X_{-\alpha_2}, X_{\alpha_4}, X_{-\alpha_4}, X_{\alpha_6}, X_{-\alpha_6} \rangle$ and $L \cap B =\langle T, X_{\alpha_2}, X_{\alpha_4}, X_{\alpha_6} \rangle$. As a consequence, $L / L \cap B \cong \mathbb{P}^1 \times \mathbb{P}^1 \times \mathbb{P}^1$. Let $s_2$, $s_4$ and $s_6$ be the simple reflections associated to $\alpha_2$, $\alpha_4$ and $\alpha_6$ respectively. Let $x_{\alpha_2}: \mathbb{C} \longrightarrow  X_{\alpha_2}$, $x_{\alpha_4}: \mathbb{C} \longrightarrow  X_{\alpha_4}$ and $x_{\alpha_6}: \mathbb{C} \longrightarrow  X_{\alpha_6}$ be the corresponding group isomorphisms. We present all 8 cells of $L / L \cap B \cong \mathbb{P}^1 \times \mathbb{P}^1 \times \mathbb{P}^1$ set-theoretically in Table~\ref{tab:cellstructure}.

By Lemma~\ref{prehsp}, for any $X_U=(L, L \cap B, \mathfrak{g}(2), U, N)$ to be possibly nonempty, $U$ can have codimension at most 3 in $\mathfrak{g}(2)$. When $\operatorname{codim}U$ is 0 or 3, the corresponding $X_U$ is respectively $\mathbb{P}^1 \times \mathbb{P}^1 \times \mathbb{P}^1$ or a finite set of points, hence paved by affines. The nontrivial cases are $\operatorname{codim}U=1, 2$. There are altogether 11 such $(L \cap B)$-stable subspaces $U$ of $\mathfrak{g}(2)$. They are divided into 6 groups and listed in Table~\ref{tab:grouping}. The computation of $X_U$ is almost the same within each group. 

We compute one example from each group with details and merely list the results for the rest. We will intersect $X_U$ with as many cells as necessary and give sufficiently concrete descriptions of the intersections, from which we deduce that $X_U$ is paved by affines. The various intersections will be displayed in tables of the same shape as Table~\ref{tab:locationofcells}, which shows the location of each cell in the table. The equations needed for computation are listed in Appendix~\ref{appendix:equations}.

\begin{table}
\caption{Cells of $L/ L \cap B$} \label{tab:cellstructure}
\centering
\begin{tabular}{|c|l|}
\hline
notation of cell & presentation of cell  \\ \hline
$\mathbb{C} \times \mathbb{C} \times \mathbb{C}$ & $\{\ x_{\alpha_2}(z_1) \dot{s}_2 x_{\alpha_4}(z_2) \dot{s}_4 x_{\alpha_6}(z_3) \dot{s}_6 (L \cap B) \in L / L \cap B\ |\ z_1, z_2, z_3 \in \mathbb{C}\}$  \\ \hline
$\mathbb{C} \times \mathbb{C} \times \{ \infty \}$ & $\{\ x_{\alpha_2}(z_1) \dot{s}_2 x_{\alpha_4}(z_2) \dot{s}_4 (L \cap B) \in L / L \cap B\ |\ z_1, z_2 \in \mathbb{C}\}$ \\ \hline
$\mathbb{C} \times \{ \infty \} \times \mathbb{C}$ & $\{\ x_{\alpha_2}(z_1) \dot{s}_2 x_{\alpha_6}(z_3) \dot{s}_6 (L \cap B) \in L / L \cap B\ |\ z_1, z_3 \in \mathbb{C}\}$   \\ \hline
$\{ \infty \} \times \mathbb{C} \times \mathbb{C}$ & $\{\ x_{\alpha_4}(z_2) \dot{s}_4 x_{\alpha_6}(z_3) \dot{s}_6 (L \cap B) \in L / L \cap B\ |\ z_2, z_3 \in \mathbb{C}\} $ \\ \hline
$\mathbb{C} \times \{ \infty \} \times \{ \infty \} $ & $\{\ x_{\alpha_2}(z_1) \dot{s}_2 (L \cap B) \in L / L \cap B\ |\ z_1 \in \mathbb{C}\}$  \\ \hline
$\{ \infty \} \times \mathbb{C} \times \{ \infty \} $ & $\{\ x_{\alpha_4}(z_2) \dot{s}_4 (L \cap B) \in L / L \cap B\ |\ z_2 \in \mathbb{C}\}$ \\ \hline
$\{ \infty \} \times \{ \infty \} \times \mathbb{C} $ & $\{\ x_{\alpha_6}(z_3) \dot{s}_6 (L \cap B) \in L / L \cap B\ |\ z_3 \in \mathbb{C}\} $ \\ \hline
$\{ \infty \} \times \{ \infty \} \times \{ \infty \}$ & $\{\dot{e}(L \cap B)\}$  \\ \hline
\end{tabular}
\end{table}

\begin{table}
\caption{The $U$'s of codimension 1 and 2} \label{tab:grouping}
\centering
\begin{tabular}{|c|c|c|c|c|}
\hline
 group    &   $\operatorname{codim}U$    &  \multicolumn{3}{c|}{$\Phi(U)$}   \\ \hline
 1st      &   1 & $\Phi(\mathfrak{g}(2)) \setminus \{ \alpha_3 \}$ & & \\ \hline
 2nd      &   1 & $\Phi(\mathfrak{g}(2)) \setminus \{ \alpha_1 \}$ & $\Phi(\mathfrak{g}(2)) \setminus \{ \alpha_5 \}$ & \\ \hline
 3rd      &   2 & $\Phi(\mathfrak{g}(2)) \setminus \{ \alpha_1, \alpha_3 \}$ & $\Phi(\mathfrak{g}(2)) \setminus \{ \alpha_3, \alpha_5 \}$ &   \\ \hline
 4th      &   2 & $\Phi(\mathfrak{g}(2)) \setminus \{ \alpha_1, \alpha_5 \}$ & &   \\ \hline
 5th      &   2 & $\Phi(\mathfrak{g}(2)) \setminus \{ \alpha_1, \alpha_{1}+\alpha_{2} \}$ & $\Phi(\mathfrak{g}(2)) \setminus \{ \alpha_5, \alpha_{4}+\alpha_{5}  \}$ &  \\ \hline
 6th      &   2 & $\Phi(\mathfrak{g}(2)) \setminus \{ \alpha_3, \alpha_{2}+\alpha_{3} \}$ & $\Phi(\mathfrak{g}(2)) \setminus \{ \alpha_3, \alpha_{3}+\alpha_{4} \} $ & $\Phi(\mathfrak{g}(2)) \setminus \{ \alpha_3, \alpha_{3}+\alpha_{6} \}$    \\ \hline
\end{tabular}
\end{table}

\begin{table}
\caption{The location of cells} \label{tab:locationofcells}
\centering
\begin{tabular}{|c|c|c|c|}
\hline
dimension of cell & \multicolumn{3}{c|}{location of each cell in the table} \\ \hline
3         & \multicolumn{3}{c|}{$\mathbb{C }\times \mathbb{C} \times \mathbb{C}$}     \\ \hline
2         & $\mathbb{C} \times \mathbb{C} \times \{ \infty \}$       & $\mathbb{C} \times \{ \infty \} \times \mathbb{C}$      & $\{ \infty \} \times \mathbb{C} \times \mathbb{C}$     \\ \hline
1         & $\mathbb{C} \times \{ \infty \} \times \{ \infty \}$       & $\{ \infty \} \times \mathbb{C} \times \{ \infty \}$       & $\{ \infty \} \times \{ \infty \} \times \mathbb{C}$     \\ \hline
0         & \multicolumn{3}{c|}{ $\{ \infty \} \times \{ \infty \} \times \{ \infty \}$ }     \\ \hline
\end{tabular}
\end{table}


\subsubsection{The 1st group}
We compute the only example where $\Phi(U)=\Phi(\mathfrak{g}(2)) \setminus \{ \alpha_3 \} $.

Intersect $X_U$ with the 3-cell $\mathbb{C} \times \mathbb{C} \times \mathbb{C} $. By definition of the quintuple and set-theoretic presentation of the cell, we know that 
\[ X_U \cap (\mathbb{C} \times \mathbb{C} \times \mathbb{C}) \cong \{(z_1, z_2, z_3) \in \mathbb{C}^3 \ |\ \dot{s}_6 \cdot x_{\alpha_6}(-z_3) \cdot \dot{s}_4 \cdot x_{\alpha_4}(-z_2) \cdot \dot{s}_2 \cdot x_{\alpha_2}(-z_1) \cdot N  \in U  \}.  \]
\pagebreak
Recall that $\dot{s}_6 \cdot x_{\alpha_6}(-z_3) \cdot \dot{s}_4 \cdot x_{\alpha_4}(-z_2) \cdot \dot{s}_2 \cdot x_{\alpha_2}(-z_1) \cdot N=\sum_{\gamma \in \Phi(\mathfrak{g}(2))} f_{\gamma}(z_1, z_2, z_3) E_{\gamma}$, where $f_{\gamma}(z_1, z_2, z_3) \in \mathbb{C}[z_1, z_2, z_3]$. Since $\Phi(U)=\Phi(\mathfrak{g}(2)) \setminus \{ \alpha_3 \} $, in order for $\dot{s}_6 \cdot x_{\alpha_6}(-z_3) \cdot \dot{s}_4 \cdot x_{\alpha_4}(-z_2) \cdot \dot{s}_2 \cdot x_{\alpha_2}(-z_1) \cdot N$ to be in $U$, we only need $f_{\alpha_3}(z_1, z_2, z_3)=0$. According to Equation~\ref{CCC}, $f_{\alpha_3}(z_1, z_2, z_3)=1+az_1z_2+bz_1z_3+cz_2z_3$, where $a$, $b$ and $c$ are some nonzero constants. Therefore, 
\[ X_U \cap (\mathbb{C} \times \mathbb{C} \times \mathbb{C}) \cong \{\ (z_1, z_2, z_3) \in \mathbb{C}^3\ |\ 1+az_1z_2+bz_1z_3+cz_2z_3=0 \}  \]
and it is a smooth quadratic hypersurface of $\mathbb{C}^3$. Such a hypersurface is known to be birationally equivalent to $\mathbb{C}^2$ via a stereographic projection, hence $X_U$ must be a smooth projective rational surface. By Lemma~\ref{surfacelem}, $X_U$ is paved by affines.

\subsubsection{The 2nd group}
We demonstrate the example where $\Phi(U)=\Phi(\mathfrak{g}(2)) \setminus \{ \alpha_1 \} $. 

The intersection of $X_U$ with each cell is given in Table~\ref{tab:codim1_1}. We explain the computation with more details. We know that 
\[ X_U \cap (\{ \infty \} \times \mathbb{C} \times \mathbb{C}) \cong \{\ (z_2, z_3) \in \mathbb{C}^2\ |\ \dot{s}_6 \cdot x_{\alpha_6}(-z_3) \cdot \dot{s}_4 \cdot x_{\alpha_4}(-z_2) \cdot N \in U\}.  \]
In order for $\dot{s}_6 \cdot x_{\alpha_6}(-z_3) \cdot \dot{s}_4 \cdot x_{\alpha_4}(-z_2) \cdot N$ to be in $U$, we only need the coefficient $f_{\alpha_1}(z_2, z_3)$ of $E_{\alpha_1}$ to be 0. According to Equation~\ref{ICC}, $f_{\alpha_1}(z_2, z_3)=0$. Therefore, $X_U \cap (\{ \infty \} \times \mathbb{C} \times \mathbb{C})$ is the entire cell. The other intersections are determined in exactly the same way, and we deduce that $X_U \cong \{ \infty \} \times \mathbb{P}^1 \times \mathbb{P}^1$, which is clearly paved by affines.

\begin{table}
\caption{ $\Phi(U)=\Phi(\mathfrak{g}(2)) \setminus \{ \alpha_1 \} $ } \label{tab:codim1_1}
\centering
\begin{tabular}{|c|c|c|c|}
\hline
dimension of cells  & \multicolumn{3}{c|}{intersection of $X_U$ with cells} \\ \hline
3          & \multicolumn{3}{c|}{ $\emptyset$ }      \\ \hline
2          &   \hspace{0.6cm} $\emptyset $ \hspace{0.6cm}    &    $\emptyset$     &  entire cell     \\ \hline
1          &   $\emptyset$      &     entire cell     &  entire cell     \\ \hline
0          & \multicolumn{3}{c|}{ entire cell}      \\ \hline
\end{tabular}
\end{table}

When $\Phi(U)=\Phi(\mathfrak{g}(2)) \setminus \{ \alpha_5 \} $, the intersections are given in Table~\ref{tab:codim1_5}, and we deduce that $X_U \cong \mathbb{P}^1 \times \{ \infty \} \times \mathbb{P}^1$.

\begin{table}
\caption{ $\Phi(U)=\Phi(\mathfrak{g}(2)) \setminus \{ \alpha_5 \} $ } \label{tab:codim1_5}
\centering
\begin{tabular}{|c|c|c|c|}
\hline
dimension of cells  & \multicolumn{3}{c|}{intersection of $X_U$ with cells} \\ \hline
3          & \multicolumn{3}{c|}{ $\emptyset$ }      \\ \hline
2          &    $\emptyset $       &    entire cell     &   $\emptyset $       \\ \hline
1          &    entire cell     &     $\emptyset $      &  entire cell     \\ \hline
0          & \multicolumn{3}{c|}{ entire cell }      \\ \hline
\end{tabular}
\end{table}

\subsubsection{The 3rd group}
We demonstrate the example where $\Phi(U)=\Phi(\mathfrak{g}(2)) \setminus \{ \alpha_1, \alpha_3 \} $. 

The intersections are given in Table~\ref{tab:codim2_13}. To compute $X_U \cap (\{ \infty \} \times \mathbb{C} \times \mathbb{C})$, note that 
\[ X_U \cap (\{ \infty \} \times \mathbb{C} \times \mathbb{C}) \cong \{\ (z_2, z_3) \in \mathbb{C}^2\ |\ \dot{s}_6 \cdot x_{\alpha_6}(-z_3) \cdot \dot{s}_4 \cdot x_{\alpha_4}(-z_2) \cdot N \in U\} \]
and we need $f_{\alpha_1}=f_{\alpha_3}=0$. According to Equation~\ref{ICC}, $f_{\alpha_1}(z_2, z_3)=0$ and $f_{\alpha_3}(z_2, z_3)=az_3 + bz_2$, where $a$ and $b$ are nonzero. Therefore, 
\[  X_U \cap (\{ \infty \} \times \mathbb{C} \times \mathbb{C}) \cong \{\ (z_2, z_3) \in \mathbb{C}^2\ |\ az_3 + bz_2 =0\}  \]
and it is clearly isomorphic to $\mathbb{C}$. As a result, $X_U \cong \mathbb{P}^1$ and it is paved by affines.

When $\Phi(U)=\Phi(\mathfrak{g}(2)) \setminus \{ \alpha_3, \alpha_5 \} $, the intersections are given in Table~\ref{tab:codim2_35} and we have $X_U \cong \mathbb{P}^1$ as well.

\begin{table}
\caption{ $\Phi(U)=\Phi(\mathfrak{g}(2)) \setminus \{ \alpha_1, \alpha_3 \} $ } \label{tab:codim2_13}
\centering
\begin{tabular}{|c|c|c|c|}
\hline
dimension of cells  & \multicolumn{3}{c|}{intersection of $X_U$ with cells} \\ \hline
3          & \multicolumn{3}{c|}{ $\emptyset$ }      \\ \hline
2          &   \hspace{0.4cm} $\emptyset $ \hspace{0.4cm}      &    \hspace{0.4cm} $\emptyset $ \hspace{0.4cm}     &   \hspace{0.4cm} $\mathbb{C} $ \hspace{0.4cm}     \\ \hline
1          &   \hspace{0.4cm} $\emptyset $ \hspace{0.4cm}      &    \hspace{0.4cm} $\emptyset $ \hspace{0.4cm}     &  \hspace{0.4cm} $\emptyset $ \hspace{0.4cm}      \\ \hline
0          & \multicolumn{3}{c|}{ entire cell }      \\ \hline
\end{tabular}
\end{table}

\begin{table}
\caption{ $\Phi(U)=\Phi(\mathfrak{g}(2)) \setminus \{ \alpha_3, \alpha_5 \} $ } \label{tab:codim2_35}
\centering
\begin{tabular}{|c|c|c|c|}
\hline
dimension of cells  & \multicolumn{3}{c|}{intersection of $X_U$ with cells} \\ \hline
3          & \multicolumn{3}{c|}{ $\emptyset$ }      \\ \hline
2          &   \hspace{0.4cm} $\emptyset $ \hspace{0.4cm}      &    \hspace{0.4cm} $\mathbb{C} $ \hspace{0.4cm}     &  \hspace{0.4cm} $\emptyset $ \hspace{0.4cm}      \\ \hline
1          &   \hspace{0.4cm} $\emptyset $ \hspace{0.4cm}      &    \hspace{0.4cm} $\emptyset $ \hspace{0.4cm}     &  \hspace{0.4cm} $\emptyset $ \hspace{0.4cm}      \\ \hline
0          & \multicolumn{3}{c|}{ entire cell }      \\ \hline
\end{tabular}
\end{table}

\subsubsection{The 4th group}
We compute the only example where $\Phi(U)=\Phi(\mathfrak{g}(2)) \setminus \{ \alpha_1, \alpha_5 \} $.

The intersections are given in Table~\ref{tab:codim2_15}. To compute $X_U \cap (\{ \infty \} \times \{ \infty \} \times \mathbb{C})$, note that
\[ X_U \cap (\{ \infty \} \times \{ \infty \} \times \mathbb{C}) \cong \{\ z_3 \in \mathbb{C}\ |\ \dot{s}_6 \cdot x_{\alpha_6}(-z_3) \cdot N \in U\}  \]
and we need $f_{\alpha_1}=f_{\alpha_5}=0$. According to Equation~\ref{IIC}, $f_{\alpha_1}(z_3)=f_{\alpha_5}(z_3)=0$, so $X_U \cap (\{ \infty \} \times \{ \infty \} \times \mathbb{C})$ is the entire cell and $X_U \cong \mathbb{P}^1$.

\begin{table}
\caption{ $\Phi(U)=\Phi(\mathfrak{g}(2)) \setminus \{ \alpha_1, \alpha_5 \} $ } \label{tab:codim2_15}
\centering
\begin{tabular}{|c|c|c|c|}
\hline
dimension of cells  & \multicolumn{3}{c|}{intersection of $X_U$ with cells} \\ \hline
3          & \multicolumn{3}{c|}{ $\emptyset$ }      \\ \hline
2          &   \hspace{0.5cm} $\emptyset $ \hspace{0.5cm}      &    \hspace{0.5cm} $\emptyset $ \hspace{0.5cm}     &   $\emptyset $       \\ \hline
1          &   \hspace{0.5cm} $\emptyset $ \hspace{0.5cm}      &    \hspace{0.5cm} $\emptyset $ \hspace{0.5cm}     &  entire cell    \\ \hline
0          & \multicolumn{3}{c|}{ entire cell }      \\ \hline
\end{tabular}
\end{table}

\subsubsection{The 5th group}
When $\Phi(U)=\Phi(\mathfrak{g}(2)) \setminus \{ \alpha_1, \alpha_{1}+\alpha_{2}  \} $, the intersection of $X_U$ with every cell is the empty set. This is because every equation from Equation~\ref{CCC} to Equation~\ref{III} has either $E_{\alpha_1}$ or $E_{\alpha_{1}+\alpha_{2} }$ as a summand (note that $N=E_{\alpha_{2}+\alpha_{3}} +E_{\alpha_{4}+\alpha_{5} } +E_{\alpha_{3}+\alpha_{6}} +E_{\alpha_{1}+\alpha_{2} } +E_{\alpha_{3}+\alpha_{4} } +E_{\alpha_{2}+\alpha_{3}+\alpha_{4}+\alpha_{6} }$, so it has $E_{\alpha_{1}+\alpha_{2} }$ as a summand as well). Hence $X_U = \emptyset$ and there is nothing to prove. 

When $\Phi(U)=\Phi(\mathfrak{g}(2)) \setminus \{ \alpha_5, \alpha_{4}+\alpha_{5}  \} $, $X_U = \emptyset$ by the same argument.

\subsubsection{The 6th group}
We demonstrate the example where $\Phi(U)=\Phi(\mathfrak{g}(2)) \setminus \{ \alpha_3, \alpha_{2}+\alpha_{3}  \} $. 

The intersections are given in Table~\ref{tab:codim2_323}. To compute $X_U \cap (\mathbb{C} \times \mathbb{C} \times \mathbb{C})$, note that 
\[ X_U \cap (\mathbb{C} \times \mathbb{C} \times \mathbb{C}) \cong \{(z_1, z_2, z_3) \in \mathbb{C}^3 \ |\ \dot{s}_6 \cdot x_{\alpha_6}(-z_3) \cdot \dot{s}_4 \cdot x_{\alpha_4}(-z_2) \cdot \dot{s}_2 \cdot x_{\alpha_2}(-z_1) \cdot N  \in U  \}  \]
and we need $f_{\alpha_3}=f_{\alpha_{2}+\alpha_{3} }=0$. According to Equation~\ref{CCC}, $f_{\alpha_3}(z_1, z_2, z_3)=1+az_1z_2+bz_1z_3+cz_2z_3$ and $f_{\alpha_{2}+\alpha_{3}}(z_1, z_2, z_3)=dz_3 + ez_2$, where $a$, $b$, $c$, $d$, $e$ are all nonzero. Therefore, 
\[  X_U \cap (\mathbb{C} \times \mathbb{C} \times \mathbb{C}) \cong \{(z_1, z_2, z_3) \in \mathbb{C}^3 \ |\ 1+az_1z_2+bz_1z_3+cz_2z_3=0, dz_3 + ez_2=0  \}.  \]
Combining the two equations to eliminate $z_3$, we get 
\[ X_U \cap (\mathbb{C} \times \mathbb{C} \times \mathbb{C}) \cong \{\ (z_1, z_2) \in \mathbb{C}^2\ |\ d+ (ad-be)z_1z_2 - cez_2^2=0 \}.  \]
Now it is clear that 
\begin{equation*}
X_U \cap (\mathbb{C} \times \mathbb{C} \times \mathbb{C}) \cong
\left\{
      \begin{array}{ll}
      \mathbb{C}^{\times} &  \text{ if } ad-be \neq 0 \\
      \mathbb{C} \amalg \mathbb{C} & \text{ if }  ad-be=0
                \end{array}
\right.
\end{equation*}

To compute $X_U \cap (\{ \infty \} \times \mathbb{C} \times \mathbb{C})$, note that
\[ X_U \cap (\{ \infty \} \times \mathbb{C} \times \mathbb{C}) \cong \{\ (z_2, z_3) \in \mathbb{C}^2\ |\ \dot{s}_6 \cdot x_{\alpha_6}(-z_3) \cdot \dot{s}_4 \cdot x_{\alpha_4}(-z_2) \cdot N \in U\}  \]
and we need  $f_{\alpha_3}=f_{\alpha_{2}+\alpha_{3} }=0$ as well. According to Equation~\ref{ICC}, $f_{\alpha_3}(z_2, z_3)=a z_3 + b z_2$ and $f_{\alpha_{2}+\alpha_{3}}(z_2, z_3)=1 + cz_2z_3$, where $a$, $b$, $c$ are nonzero. Hence
\begin{equation*} 
\begin{aligned}
 X_U \cap (\{ \infty \} \times \mathbb{C} \times \mathbb{C}) & \cong & \{\ (z_2, z_3) \in \mathbb{C}^2\ |\ a z_3 + b z_2=0, 1 + cz_2z_3=0 \} \\
 & = & \{(\sqrt{a/(bc)}, -\sqrt{b/(ac)}), (-\sqrt{a/(bc)}, \sqrt{b/(ac))}\}
\end{aligned}
\end{equation*}
Combining the intersections of $X_U$ with the two cells, we see that $X_U$ is either $\mathbb{P}^1$ or $\mathbb{P}^1 \amalg \mathbb{P}^1$, hence paved by affines.

\begin{table}
\caption{ $\Phi(U)=\Phi(\mathfrak{g}(2)) \setminus \{ \alpha_3, \alpha_{2}+\alpha_{3}  \} $ } \label{tab:codim2_323}
\centering
\begin{tabular}{|c|c|c|c|}
\hline
dimension of cells  & \multicolumn{3}{c|}{intersection of $X_U$ with cells} \\ \hline
3          & \multicolumn{3}{c|}{ $\mathbb{C}^{\times} \, $ or $\, \mathbb{C} \amalg \mathbb{C}$ }      \\ \hline
2          &   \hspace{0.4cm} $\emptyset $ \hspace{0.4cm}      &    \hspace{0.4cm} $\emptyset $ \hspace{0.4cm}     &  2 points      \\ \hline
1          &   \hspace{0.4cm} $\emptyset $ \hspace{0.4cm}      &    \hspace{0.4cm} $\emptyset $ \hspace{0.4cm}     &  \hspace{0.4cm} $\emptyset $ \hspace{0.4cm}      \\ \hline
0          & \multicolumn{3}{c|}{ \hspace{0.4cm} $\emptyset $ \hspace{0.4cm} }      \\ \hline
\end{tabular}
\end{table}

\begin{table}
\caption{ $\Phi(U)=\Phi(\mathfrak{g}(2)) \setminus \{ \alpha_3, \alpha_{3}+\alpha_{4}  \} $ } \label{tab:codim2_334}
\centering
\begin{tabular}{|c|c|c|c|}
\hline
dimension of cells  & \multicolumn{3}{c|}{intersection of $X_U$ with cells} \\ \hline
3          & \multicolumn{3}{c|}{ $\mathbb{C}^{\times} \, $ or $\, \mathbb{C} \amalg \mathbb{C}$ }      \\ \hline
2          &   \hspace{0.4cm} $\emptyset $ \hspace{0.4cm}      &    2 points     &  \hspace{0.4cm} $\emptyset $ \hspace{0.4cm}      \\ \hline
1          &   \hspace{0.4cm} $\emptyset $ \hspace{0.4cm}      &    \hspace{0.4cm} $\emptyset $ \hspace{0.4cm}     &  \hspace{0.4cm} $\emptyset $ \hspace{0.4cm}      \\ \hline
0          & \multicolumn{3}{c|}{ \hspace{0.4cm} $\emptyset $ \hspace{0.4cm} }      \\ \hline
\end{tabular}
\end{table}

\begin{table}
\caption{ $\Phi(U)=\Phi(\mathfrak{g}(2)) \setminus \{ \alpha_3, \alpha_{3}+\alpha_{6} \} $ } \label{tab:codim2_336}
\centering
\begin{tabular}{|c|c|c|c|}
\hline
dimension of cells  & \multicolumn{3}{c|}{intersection of $X_U$ with cells} \\ \hline
3          & \multicolumn{3}{c|}{ $\mathbb{C}^{\times} \, $ or $\, \mathbb{C} \amalg \mathbb{C}$ }      \\ \hline
2          &   2 points      &    \hspace{0.4cm} $\emptyset $ \hspace{0.4cm}     &  \hspace{0.4cm} $\emptyset $ \hspace{0.4cm}      \\ \hline
1          &   \hspace{0.4cm} $\emptyset $ \hspace{0.4cm}      &    \hspace{0.4cm} $\emptyset $ \hspace{0.4cm}     &  \hspace{0.4cm} $\emptyset $ \hspace{0.4cm}      \\ \hline
0          & \multicolumn{3}{c|}{ \hspace{0.4cm} $\emptyset $ \hspace{0.4cm} }      \\ \hline
\end{tabular}
\end{table}

When $\Phi(U)=\Phi(\mathfrak{g}(2)) \setminus \{ \alpha_3, \alpha_{3}+\alpha_{4}  \} $ or $\Phi(U)=\Phi(\mathfrak{g}(2)) \setminus \{ \alpha_3, \alpha_{3}+\alpha_{6}  \} $, the intersections are given in Table~\ref{tab:codim2_334} and Table~\ref{tab:codim2_336} respectively. By a similar argument, they are both either $\mathbb{P}^1$ or $\mathbb{P}^1 \amalg \mathbb{P}^1$. We have now finished the proof for the case of $E_6$.

\appendix
\section{Betti Numbers of \texorpdfstring{$\operatorname{Hess}(M, x_J)$}{}} \label{appendix:bettinumber}
We compute $\operatorname{dim}_{\mathbb{C}}H^*(\operatorname{Hess}(M, x_J))$ for those regular Hessenberg varieties $\operatorname{Hess}(M, x_J)$ involved in subsection \ref{dotactioncomputation}. 

Let $G$ be a connected algebraic group over $\mathbb{C}$ of type $G_2$. We use the same setups as in subsection \ref{strofg2}. Let $M$ be a Hessenberg subspace of $\mathfrak{g}$ and $J$ be a subset of the set of simple roots $\Delta=\{\alpha, \beta\}$. There always exists a semisimple element $s_J \in \mathfrak{g}$ such that $C_{\mathfrak{g}}(s_J)$ is a Levi subalgebra of $\mathfrak{g}$ whose Weyl group is $W_J$. Let $n_J \in C_{\mathfrak{g}}(s_J)$ be a regular nilpotent element and define $x_J=s_J+n_J$. Then $x_J$ is a regular element of $\mathfrak{g}$ and $\operatorname{Hess}(M, x_J)$ is a regular Hessenberg variety. For simplicity of notation, let $L$ denote the Levi subgroup of $G$ with Lie algebra $\mathfrak{l}=C_{\mathfrak{g}}(s_J)$. As a result, $W_L=W_J$. By Lemma~\ref{wl1}, each $w \in W$ can be written uniquely as $w=yv$ with $y \in W_L$ and $v \in W^L$, where $W^L= \{ v \in W \  | \  \Phi_v \subset \Phi(\mathfrak{u}_P) \}$. Define $M_v=\dot{v} \cdot M \cap \mathfrak{l}$. The following theorem is a consequence of \cite{precup2013affine}[Theorem 4.10] and \cite{precup2013affine}[Corollary 5.8].

\begin{theorem*}[Precup] \label{Prbetticomp}
The regular Hessenberg variety $\operatorname{Hess}(M, x_J)$ is paved by affines. Moreover:
\begin{itemize}
\item[(1)] Every nonempty intersection $X_w \cap \operatorname{Hess}(M, x_J)$ is an affine space.
\item[(2)] $X_w \cap \operatorname{Hess}(M, x_J)$ is nonempty if and only if $J \subset y(\Phi(M_v))$.
\item[(3)] When $X_w \cap \operatorname{Hess}(M, x_J) \neq \emptyset$, 
\[ \operatorname{dim}_{\mathbb{C}} X_w \cap \operatorname{Hess}(M, x_J) = |\Phi_y \cap y(\Phi^-(M_v))| + |y(\Phi_v) \cap w(\Phi^-(M))|.  \]
\end{itemize}
\end{theorem*}

Clearly, the theorem above gives us a way of computing $\operatorname{dim}_{\mathbb{C}}H^*(\operatorname{Hess}(M, x_J))$ by inspecting the intersection $X_w \cap \operatorname{Hess}(M, x_J)$ for each $w \in W$. The inspections are carried out for all the $\operatorname{Hess}(M, x_J)$'s involved in subsection \ref{dotactioncomputation}, and the results are summarized in the following tables. A blank entry in the table means the corresponding intersection $X_w \cap \operatorname{Hess}(M, x_J)$ is empty.

\begin{table}[ht]
\def\arraystretch{1.2}
\caption{$I=I_{\beta+\alpha }$, $M=I^{\perp}$, $J=\{\alpha, \beta\}$, $W_J=W$} \label{tab:betti1}
\centering
\begin{tabular}{|c|c|c|c|c|c|c|}
\hline
                $w \in W$                                                                                                                         & $e$ & $r^{-1}$  & $r^{-2}$  & $r^{-3}$  & $r^{-4}$  & $r^{-5}$  \\ \hline
$\operatorname{dim}_{\mathbb{C}} X_w \cap \operatorname{Hess}(M, x_J)$                                                                      & 0   &           &           & 2         &           &           \\ \hline
       $w \in W$                                                                                                                                  & $t$ & $tr^{-1}$ & $tr^{-2}$ & $tr^{-3}$ & $tr^{-4}$ & $tr^{-5}$ \\ \hline
$\operatorname{dim}_{\mathbb{C}} X_w \cap \operatorname{Hess}(M, x_J)$ & 1   & 1         &           &           &           &           \\ \hline
\end{tabular}
\end{table}

\begin{table}[ht]
\def\arraystretch{1.2}
\caption{$I=I_{\beta+\alpha }$, $M=I^{\perp}$, $J=\{\beta\}$, $W_J=\langle t \rangle$} \label{tab:betti2}
\centering
\begin{tabular}{|c|c|c|c|c|c|c|}
\hline
$w \in W^L$                                                            & $e$ & $tr^{-1}$ & $tr^{-2}$ & $tr^{-3}$ & $r^{-4}$  & $r^{-5}$  \\ \hline
$\operatorname{dim}_{\mathbb{C}} X_w \cap \operatorname{Hess}(M, x_J)$ & 0   & 1         & 1         & 1         & 1         & 1         \\ \hline
$w \in tW^L$                                                           & $t$ & $r^{-1}$  & $r^{-2}$  & $r^{-3}$  & $tr^{-4}$ & $tr^{-5}$ \\ \hline
$\operatorname{dim}_{\mathbb{C}} X_w \cap \operatorname{Hess}(M, x_J)$ & 1   &           &           & 2         &           &           \\ \hline
\end{tabular}
\end{table}

\begin{table}[ht]
\def\arraystretch{1.2}
\caption{$I=I_{\alpha }$, $M=I^{\perp}$, $J=\{\alpha\}$, $W_J=\langle s \rangle$} \label{tab:betti3}
\centering
\begin{tabular}{|c|c|c|c|c|c|c|}
\hline
$w \in W^L$                                                            & $e$ & $t$ & $sr^2$  & $sr^3$  & $r^{4}$  & $r^{5}$  \\ \hline
$\operatorname{dim}_{\mathbb{C}} X_w \cap \operatorname{Hess}(M, x_J)$ & 0   & 1   & 1       & 1       & 0        & 0        \\ \hline
$w \in sW^L$                                                           & $s$ & $r$ & $r^{2}$ & $r^{3}$ & $sr^{4}$ & $sr^{5}$ \\ \hline
$\operatorname{dim}_{\mathbb{C}} X_w \cap \operatorname{Hess}(M, x_J)$ &     &     &         &         &          &          \\ \hline
\end{tabular}
\end{table}

\begin{table}[ht]
\def\arraystretch{1.2}
\caption{$I=I_{\alpha }$, $M=I^{\perp}$, $J=\{\beta\}$, $W_J=\langle t \rangle$} \label{tab:betti4}
\centering
\begin{tabular}{|c|c|c|c|c|c|c|}
\hline
$w \in W^L$                                                            & $e$ & $tr^{-1}$ & $tr^{-2}$ & $tr^{-3}$ & $r^{-4}$  & $r^{-5}$  \\ \hline
$\operatorname{dim}_{\mathbb{C}} X_w \cap \operatorname{Hess}(M, x_J)$ & 0   & 0         & 0         & 0         & 1         & 1         \\ \hline
$w \in tW^L$                                                           & $t$ & $r^{-1}$  & $r^{-2}$  & $r^{-3}$  & $tr^{-4}$ & $tr^{-5}$ \\ \hline
$\operatorname{dim}_{\mathbb{C}} X_w \cap \operatorname{Hess}(M, x_J)$ & 1   &           &           & 1         &           &           \\ \hline
\end{tabular}
\end{table}

\begin{table}[ht]
\def\arraystretch{1.2}
\caption{$I=I_{\beta }$, $M=I^{\perp}$, $J=\{\alpha\}$, $W_J=\langle s \rangle$} \label{tab:betti5}
\centering
\begin{tabular}{|c|c|c|c|c|c|c|}
\hline
$w \in W^L$                                                            & $e$ & $t$ & $sr^2$  & $sr^3$  & $r^{4}$  & $r^{5}$  \\ \hline
$\operatorname{dim}_{\mathbb{C}} X_w \cap \operatorname{Hess}(M, x_J)$ & 0   & 0   & 0       & 0       & 1        & 1        \\ \hline
$w \in sW^L$                                                           & $s$ & $r$ & $r^{2}$ & $r^{3}$ & $sr^{4}$ & $sr^{5}$ \\ \hline
$\operatorname{dim}_{\mathbb{C}} X_w \cap \operatorname{Hess}(M, x_J)$ & 1   &     &         & 1       &          &          \\ \hline
\end{tabular}
\end{table}

\begin{table}[ht]
\def\arraystretch{1.2}
\caption{$I=I_{\beta }$, $M=I^{\perp}$, $J=\{\beta\}$, $W_J=\langle t \rangle$} \label{tab:betti6}
\centering
\begin{tabular}{|c|c|c|c|c|c|c|}
\hline
$w \in W^L$                                                            & $e$ & $tr^{-1}$ & $tr^{-2}$ & $tr^{-3}$ & $r^{-4}$  & $r^{-5}$  \\ \hline
$\operatorname{dim}_{\mathbb{C}} X_w \cap \operatorname{Hess}(M, x_J)$ & 0   & 1         & 1         & 1         & 0         & 0         \\ \hline
$w \in tW^L$                                                           & $t$ & $r^{-1}$  & $r^{-2}$  & $r^{-3}$  & $tr^{-4}$ & $tr^{-5}$ \\ \hline
$\operatorname{dim}_{\mathbb{C}} X_w \cap \operatorname{Hess}(M, x_J)$ &     &           &           &           &           &           \\ \hline
\end{tabular}
\end{table}

\clearpage

\section{Equations for Type \texorpdfstring{$E_6$}{}} \label{appendix:equations}
This appendix contains all the equations necessary for the proof of Proposition~\ref{F4E6} for type $E_6$. The wildcard symbol $*$ represents a random nonzero number, whose exact value is not needed for our purpose.

$X_U \cap (\mathbb{C} \times \mathbb{C} \times \mathbb{C}) \cong \{(z_1, z_2, z_3) \in \mathbb{C}^3 \ |\ \dot{s}_6 \cdot x_{\alpha_6}(-z_3) \cdot \dot{s}_4 \cdot x_{\alpha_4}(-z_2) \cdot \dot{s}_2 \cdot x_{\alpha_2}(-z_1) \cdot N  \in U  \}$
\begin{equation} \label{CCC}
\begin{aligned}
& \dot{s}_6 \cdot x_{\alpha_6}(-z_3) \cdot \dot{s}_4 \cdot x_{\alpha_4}(-z_2) \cdot \dot{s}_2 \cdot x_{\alpha_2}(-z_1) \cdot N  \\
 = &(*z_1+*z_2)E_{\alpha_{3}+\alpha_{6}}+(1+*z_1z_2+*z_1z_3+*z_2z_3)E_{\alpha_3}+E_{\alpha_{3}+\alpha_{4}+\alpha_{6}} \\
 + & (*z_3+*z_1)E_{\alpha_{3}+\alpha_{4}}+E_{\alpha_5}+E_{\alpha_1}+E_{\alpha_{2}+\alpha_{3}+\alpha_{6}} \\
 + & (*z_3+*z_2)E_{\alpha_{2}+\alpha_{3}}+E_{\alpha_{2}+\alpha_{3}+\alpha_{4}}
\end{aligned}
\end{equation}

$X_U \cap (\mathbb{C} \times \mathbb{C} \times \{ \infty \} ) \cong \{\ (z_1, z_2) \in \mathbb{C}^2 \ |\ \dot{s}_4 \cdot x_{\alpha_4}(-z_2) \cdot \dot{s}_2 \cdot x_{\alpha_2}(-z_1) \cdot N \in U\}$
\begin{equation} \label{CCI}
\begin{aligned}
& \dot{s}_4 \cdot x_{\alpha_4}(-z_2) \cdot \dot{s}_2 \cdot x_{\alpha_2}(-z_1) \cdot N \\
=& (*z_1+*z_2)E_{\alpha_3} + E_{\alpha_{3}+\alpha_{4}} + E_{\alpha_5} + E_{\alpha_1} +E_{\alpha_{2}+\alpha_{3}} \\
+& (1+*z_1z_2)E_{\alpha_{3}+\alpha_{6}}+ E_{\alpha_{2}+\alpha_{3}+\alpha_{4}+\alpha_6}+*z_2 E_{\alpha_{2}+\alpha_{3}+\alpha_{6}} + *z_1E_{\alpha_{3}+\alpha_{4}+\alpha_{6}}
\end{aligned}
\end{equation}

$X_U \cap (\mathbb{C} \times \{ \infty \} \times \mathbb{C}) \cong \{\ (z_1, z_3) \in \mathbb{C}^2 \ |\ \dot{s}_6 \cdot x_{\alpha_6}(-z_3) \cdot \dot{s}_2 \cdot x_{\alpha_2}(-z_1) \cdot N \in U\}$
\begin{equation} \label{CIC}
\begin{aligned}
& \dot{s}_6 \cdot x_{\alpha_6}(-z_3) \cdot \dot{s}_2 \cdot x_{\alpha_2}(-z_1) \cdot N \\
= & E_{\alpha_{3}+\alpha_{6} } + (*z_3 + *z_1)E_{\alpha_3}+E_{\alpha_{4}+\alpha_{5} } +E_{\alpha_{2}+\alpha_{3} } + E_{\alpha_1} \\
+ & E_{\alpha_{2}+\alpha_{3}+\alpha_{4}+\alpha_{6} } + *z_3 E_{\alpha_{2}+\alpha_{3}+\alpha_{4} } + (1+ *z_1z_3)E_{\alpha_{3}+\alpha_{4} } + *z_1 E_{\alpha_{3}+\alpha_{4}+\alpha_{6} }
\end{aligned}
\end{equation}

$X_U \cap (\{ \infty \} \times \mathbb{C} \times \mathbb{C}) \cong \{\ (z_2, z_3) \in \mathbb{C}^2\ |\ \dot{s}_6 \cdot x_{\alpha_6}(-z_3) \cdot \dot{s}_4 \cdot x_{\alpha_4}(-z_2) \cdot N \in U\}$
\begin{equation} \label{ICC}
\begin{aligned}
& \dot{s}_6 \cdot x_{\alpha_6}(-z_3) \cdot \dot{s}_4 \cdot x_{\alpha_4}(-z_2) \cdot N \\
=& E_{\alpha_{2}+\alpha_{3}+\alpha_{4}+\alpha_{6} } + *z_3 E_{\alpha_{2}+\alpha_{3}+\alpha_{4} } + *z_2 E_{\alpha_{2}+\alpha_{3}+\alpha_{6} }+  (1+ *z_2z_3)E_{\alpha_{2}+\alpha_{3} } \\
+& E_{\alpha_{3}+\alpha_{6} } +(*z_3 +*z_2)E_{\alpha_3} +E_{\alpha_{1}+\alpha_{2} } + E_{\alpha_5} + E_{\alpha_{3}+\alpha_{4} }
\end{aligned}
\end{equation}

$X_U \cap (\mathbb{C} \times \{ \infty \} \times \{ \infty \} ) \cong \{\ z_1 \in \mathbb{C}\ |\ \dot{s}_2 \cdot x_{\alpha_2}(-z_1) \cdot N \in U\}$
\begin{equation} \label{CII}
\begin{aligned}
& \dot{s}_2 \cdot x_{\alpha_2}(-z_1) \cdot N \\
=& E_{\alpha_3}+E_{\alpha_{4}+\alpha_{5} } +E_{\alpha_{2}+\alpha_{3}+\alpha_{6} } + E_{\alpha_1} + E_{\alpha_{2}+\alpha_{3}+\alpha_{4} } \\
+& E_{\alpha_{3}+\alpha_{4}+\alpha_{6} } + *z_1 E_{\alpha_{3}+\alpha_{6} } + *z_1 E_{\alpha_{3}+\alpha_{4} }
\end{aligned}
\end{equation}

$X_U \cap (\{ \infty \} \times \mathbb{C} \times \{ \infty \} ) \cong \{\ z_2 \in \mathbb{C}\ |\ \dot{s}_4 \cdot x_{\alpha_4}(-z_2) \cdot N \in U\}$
\begin{equation} \label{ICI}
\begin{aligned}
& \dot{s}_4 \cdot x_{\alpha_4}(-z_2) \cdot N \\
=& E_{\alpha_{2}+\alpha_{3}+\alpha_{4} } + *z_2 E_{\alpha_{2}+\alpha_{3} } + E_{\alpha_3} + E_{\alpha_{1}+\alpha_{2} } + E_{\alpha_5} \\
+& E_{\alpha_{2}+\alpha_{3}+\alpha_{6} } + E_{\alpha_{3}+\alpha_{4}+\alpha_{6} } + *z_2 E_{\alpha_{3}+\alpha_{6} }
\end{aligned}
\end{equation}

$X_U \cap (\{ \infty \} \times \{ \infty \} \times \mathbb{C}) \cong \{\ z_3 \in \mathbb{C}\ |\ \dot{s}_6 \cdot x_{\alpha_6}(-z_3) \cdot N \in U\}$
\begin{equation} \label{IIC}
\begin{aligned}
& \dot{s}_6 \cdot x_{\alpha_6}(-z_3) \cdot N \\
=& E_{\alpha_{2}+\alpha_{3}+\alpha_{6} } +*z_3 E_{\alpha_{2}+\alpha_{3} } +E_{\alpha_{3}+\alpha_{4}+\alpha_{6} } + *z_3 E_{\alpha_{3}+\alpha_{4} } \\
+& E_{\alpha_{1}+\alpha_{2} } + E_{\alpha_{4}+\alpha_{5} } + E_{\alpha_3} + E_{\alpha_{2}+\alpha_{3}+\alpha_{4} }
\end{aligned}
\end{equation}

\begin{equation} \label{III}
X_U \cap (\{ \infty \} \times \{ \infty \} \times \{ \infty \} )=
\left\{
      \begin{array}{ll}
      \{ \infty \} \times \{ \infty \} \times \{ \infty \} &  \text{ if } N \in U \\
      \emptyset & \text{ if } N \notin U
                \end{array}
\right.
\end{equation}

\end{document}